\documentclass[a4paper,reqno]{amsart}
\pdfoutput=1

\usepackage[all,hyperref]{paper_diening}
\renewcommand{\norm}[1]{\left\lVert#1\right\rVert}
\renewcommand{\abs}[1]{\left\lvert#1\right\rvert}

\usepackage{amsmath,amsthm,amssymb}

\usepackage{caption}
\usepackage{subcaption}

\usepackage[]{algorithm2e}
\usepackage{framed}

\usepackage{epsfig}
\usepackage[utf8]{inputenc}
\usepackage[T1]{fontenc}
\usepackage{dsfont}
\usepackage{faktor}
\usepackage{xparse}
\let\realItem\item 
\makeatletter
\NewDocumentCommand\myItem{ o }{%
   \IfNoValueTF{#1}%
      {\realItem}
      {\realItem[#1]\def\@currentlabel{#1}}
}
\makeatother
\usepackage{enumitem}
\setlist[enumerate]{
    before=\let\item\myItem,       
    label=\textnormal{(\alph*)}, 
    widest=(2')                    
}

\usepackage[backend=biber,maxbibnames=5,sorting=nyt,maxalphanames=5,style=alphabetic,bibencoding=utf8,giveninits,url=false,isbn=false]{biblatex}
\AtBeginBibliography{\scriptsize}

\usepackage{tikz}
\usetikzlibrary{plotmarks}
\usepackage{pgfplots}
\pgfplotsset{compat=1.18}

\usepackage{mathrsfs}

\usepackage{xcolor}
\definecolor{intro_color1}{HTML}{1b9e77}
\definecolor{intro_color2}{HTML}{d95f02}
\definecolor{intro_color3}{HTML}{7570b3}
\definecolor{intro_color4}{HTML}{000000}

\definecolor{green_light}{HTML}{66c2a5}
\definecolor{orange_light}{HTML}{fc8d62}
\definecolor{purple_light}{HTML}{8da0cb}

\numberwithin{equation}{section}

\newcommand{\Div}{\divergence}

\newcommand{\dd}{\,\mathrm{d}}

\providecommand{\seminormtmp}[2]{{#1[{#2}#1]}}
\providecommand{\seminorm}[1]{\seminormtmp{}{#1}}

\newcommand{\disc}{\mathcal{D}}

\begin{document}

\title[Reaching the equilibrium]{Reaching the equilibrium: Long-term stable approximations for stochastic non-Newtonian Stokes equations with transport noise}

\author{Jerome Droniou}%
\address[J. Droniou]{IMAG, Univ. Montpellier, CNRS, Montpellier, France. School of Mathematics, Monash University, Australia}%
\email{jerome.droniou@umontpellier.fr}%

\author{Kim-Ngan Le}%
\address[N. Le]{School of Mathematics, Monash University, Australia}%
\email{ngan.le@monash.edu}%

\author{J\"{o}rn Wichmann}%
\address[J. Wichmann]{School of Mathematics, Monash University, Australia}%
\email{joern.wichmann@monash.edu}%

\thanks{This work was partially supported by the Australian Government through the Australian Research Council’s Discovery
Projects funding scheme (grant number DP220100937). }

\begin{abstract}
We propose and analyse a novel, fully discrete numerical algorithm for the approximation of the generalised Stokes system forced by transport noise -- a prototype model for non-Newtonian fluids including turbulence. Utilising the Gradient Discretisation Method, we show that the algorithm is long-term stable for a broad class of particular Gradient Discretisations. Building on the long-term stability and the derived continuity of the algorithm's solution operator, we construct two sequences of approximate invariant measures. At the moment, each sequence lacks one important feature: either the existence of a limit measure, or the invariance with respect to the discrete semigroup. We derive an abstract condition that merges both properties, recovering the existence of an invariant measure. We provide an example for which invariance and existence hold simultaneously, and characterise the invariant measure completely. We close the article by conducting two numerical experiments that show the influence of transport noise on the dynamics of power-law fluids; in particular, we find that transport noise enhances the dissipation of kinetic energy, the mixing of particles, as well as the size of vortices.  
\end{abstract}



\maketitle

\section{Introduction}
Accurate fluid prediction is an important tool in many applications, from weather forecasting to aircraft design. However, due to the complexity of the dynamics, its mathematical description is anything but trivial. For the study of turbulence, different models have been proposed: stochastic fluid equations incorporate randomness to capture small scale features of turbulent flow. In the physics community, they have been introduced by Kraichnan in the late '50s~\cite{Kraichnan_1959}. Later, a rigorous mathematical consideration was initiated by Bensoussan and Temam~\cite{BENSOUSSAN1973195}. Since then, stochastic fluid equations have received broad attention from both communities; for relevant literature, we refer to the preface of the book~\cite{Flandoli2023}. 

One specific stochastic structure -- transport noise -- gained much interest in recent times, since it preserves key features from the deterministic situation such as the system's pathwise energy, while simultaneously accounting for the small scale features of turbulence. Its mathematical investigation has been initiated by Mikulevicius and Rozovskii~\cite{MikuRozov}. More recently, Flandoli and Pappalettera have shown that transport noise canonically emerges from additive small scale perturbations~\cite{Flandoli2022}.

In this article, we investigate the \textit{generalised Stokes equations} -- a prototype model for non-Newtonian fluids -- forced by transport noise and supplemented by general Dirichlet boundary conditions. We are particularly interested in the long time behaviour of this model. Intuitively, an initially perturbed fluid that isn't influenced by any external force (i.e., homogeneous Dirichlet boundary condition) will eventually evolve to its steady state. For stochastic equations, this steady state is an invariant measure; see, e.g., the book~\cite{DaPrato1996}. 

The theoretical investigation of invariant measures has flourished; see, e.g.,~\cite{Flandoli1995,Goldys2005,Hairer2006,Liu2011,Chekroun2012,Gess2016,Brzeniak2017,CotiZelati2019}. Most of the results address additive noises only. The reason is simple: the system under consideration needs to behave well for arbitrarily large time horizons. This is achieved by ensuring that the dissipation is sufficiently dominant compared to the intensity of the noise. 
Multiplicative noises generally don't allow the system to be well behaved; and consequently, no invariant measure exists.

But transport noise is different. A formal calculation (see Section~\ref{sec:long-stable-motiv}) reveals that transport noise leads to an energy (in)equality, which strongly suggests the existence of an invariant measure even under inhomogeneous Dirichlet boundary conditions. To the best of our knowledge, nothing is known about this measure: Does it exist rigorously? Is it unique? What are its properties? How can we construct it?

In this article, we turn our attention towards the last mentioned question: we address a computable construction of approximate invariant measures. The construction is founded on a novel numerical algorithm that mimics the energy (in)equality on the discrete level. The algorithm builds on the seminal works of Crank and Nicolson~\cite{Crank1947}, and Temam~\cite{Temam1968}, in combination with a generic Gradient Discretisation (GD); see, e.g., the book~\cite{Droniou2018}. We trace the dependence of the constants on the discretisation parameters (time horizon, time step size, and spatial discretisation) explicitly. This has the advantage that we can pass with discretisation parameters to their corresponding asymptotes individually; for instance, the validity of the energy (in)equality guarantees the \textit{long-term} stability of the algorithm for arbitrary large time horizons. Consequently, already for fixed time step size and spatial discretisation, one can discuss the existence of an invariant measure for the discretised semigroup generated by the algorithm. This is favourable since it enables a direct comparison of the invariant measures for the discrete and continuous dynamics. 

\subsection{Our contributions} Utilising the Gradient Discretisation Method (GDM), we verify that the newly proposed algorithm is long-term stable for a broad class of particular GDs. Since the algorithm has two intrinsic time scales (as we discuss in more details in Section~\ref{sec:long-term}), it gives rise to two canonical sequences of measures that are both promising constructors for the discrete invariant measure. The first sequence of measures is well prepared for probabilistic arguments. Hence, we show that any accumulation point of this sequence of measures is necessarily invariant; the second sequence of measures enjoys improved regularity which implies tightness of the sequence of measures. Consequently, the second sequence of measures has an accumulation point. However, for general boundary conditions, we are currently not able to prove the existence of a limit measure and the invariance with respect to the discrete semigroup for both sequences simultaneously. For homogeneous boundary conditions, we verify the convergence of both sequences to the unique invariant measure; in this case, we additionally characterise the invariant measure completely. We accommodate the theoretical findings with numerical simulations. In summary, our main contributions in this article include:
\begin{itemize}
    \item the design of a novel and long-term stable algorithm; 
    \item the theoretical investigation of the algorithm's solution operator;  
    \item the construction and feature verification of approximate invariant measures;
    \item and the validation of the algorithm based on two numerical experiments. 
\end{itemize}
To the best of our knowledge, this is the first result that addresses computable approximations of the invariant measure for stochastic non-Newtonian fluids. 

\subsection{Related literature}

\textbf{A unified analytical framework: GDM.}
The main advantage of the GDM is a clear detachment of theory and application: by replacing a concrete spatial discretisation with a proxy (the GD), one derives model specific constraints on the GD through the investigation of the generic algorithm. In this way, the theoretical results are true for many particular GDs, as long as they match the constraints. Thus, the GDM eliminates the need for case-by-case studies of particular discretisations, and enables a unified analytical framework. If one wants to use the algorithm in applications, then one selects \textit{a posteriori} a suitable GD that satisfies all constraints. Only this selection has to be done case-by-case. 

The GDM has found broad success in the analysis of discretisations for various deterministic equations, e.g.,  porous media equations~\cite{DroniouLe2020}, linear advection problems~\cite{DroniouRobert2019}, poro-mechanical models~\cite{Bonaldi2021}, and Stokes equations~\cite{Droniou2015}. For more references and details, we refer to the book~\cite{Droniou2018}. 
Its use for stochastic equations is less established. Only recently, the GDM was used to analyse discretisations for stochastic evolution equations driven by Leray--Lions operators~\cite{MR4410739}, and the stochastic Stefan problem~\cite{Droniou2024}.

\textbf{Numerical algorithms for SPDEs and their relation to invariant measures.}
In recent years, many authors have contributed towards the design and analysis of numerical algorithms for stochastic partial differential equations (SPDEs). A non-exhaustive list is given by:
\cite{Jentzen2009,MR2465711,Majee2017,MR3843574,MR4298537,MR4261330,Diening2022Averaged,MR4565984,Klioba2024}. Most of these results address the convergence of the \textit{strong error} of the respective algorithms on a fixed time horizon. But for the construction of an invariant measure, the dependence on the time horizon needs to be traced explicitly. Especially, as soon as a Gronwall argument is involved (which is the case for most non-linear SPDEs and multiplicative noises), the constants grow too rapidly.

The analysis of the \textit{weak error} of numerical algorithms is often linked to the investigation of approximate invariant measures. For example, in a sequence of works, Bréhier and collaborators analysed various time discretisations such as the implicit Euler scheme and explicit tamed exponential Euler scheme for the approximation of semi-linear SPDEs forced by additive noise; see, e.g.,~\cite{Brhier2013,Brhier2016,BrhierSample2016,Brhier2022,Brhier2024}. They showed that these algorithms can be used to construct approximate invariant measures and quantified their convergence rates by investigating the weak error of the schemes. Additional related results can be found, e.g., in~\cite{ChenWang2016,KovcsLang2020,Chen2020,Cui2021,Chen2023} and the references therein. We want to emphasize that the weak error analysis of numerical algorithms is of its own interest. More details can be found, e.g., in~\cite{Debussche2008,Debussche2010,Wang2015} and the book~\cite{Kruse2014}. 

An extension of the error analysis of numerical algorithms for stochastic fluid equations, such as the Navier--Stokes and (generalised) Stokes systems, is by far non-trivial. It itself has received broad attention; see, e.g.,~\cite{Brzezniak2013,BreitDod2021,Bessaih2022,Doghman2022,breit2024means} and~\cite{Feng2021,Feng2022,Le2024Spacetime,Li2024} for Navier--Stokes and (generalised) Stokes systems, respectively. We highlight the recent article~\cite{breit2024means}; in it, the authors derive and analyse a time-discrete algorithm for the approximation of the Navier--Stokes system forced by transport noise. As it turns out, transport noise enables the algorithm to converge optimally, i.e., the convergence rate of the strong error is of order $1/2$. 

The numerical approximation of invariant measures for stochastic fluid equations is widely unexplored. The only result we could find is the recent work of Glatt-Holtz and Mondaini~\cite{GlattHoltz2024}, where they investigate a semi-implicit in time and spectral Galerkin in space numerical approximation. They show that the discrete semigroup, generated by their discretisation, gives rise to an invariant measure. However, their method is tailored for additive noise and it cannot be utilised for transport noise.  

\subsection{Structure}
The remaining article is structured as follows: We start by presenting the model in Section~\ref{sec:Setup-Motivation}. Moreover, we introduce notation and give a heuristic argument on why the generalised Stokes system forced by transport noise has an invariant measure; the next section (Section~\ref{sec:disc-tools}) contains details about time and space discretisations. Additionally, we present the fully discrete algorithm, as well as its long-term stability; Section~\ref{sec:long-term} introduces concepts for the long-term dynamics of the algorithm, including the discrete semigroup and invariant measure. We present two constructors for approximate invariant measures and show that they satisfy invariance and existence, respectively. We close the section by discussing a sufficient condition to merge both properties, and by noting that this condition is satisfied for homogeneous boundary conditions (often the only ones considered in numerical analysis of Stokes models), in which case the invariant measure is trivial; for the sake of completeness, however, and to highlight the particular challenges they pose, we opted to consider generic boundary conditions throughout the paper. Section~\ref{sec:proofs} is dedicated to the verification of the claims presented in the two previous sections. A key ingredient is the continuity of the solution operator generated by the algorithm, which we derive in full detail. Following the theoretical results, we conduct two numerical experiments in Section~\ref{sec:num-sim}. We introduce the experimental configurations, the concrete choices of parameters, as well as the results; finally, we close the article with a conclusion in Section~\ref{sec:conclusion}. We supplement the article with an appendix (Appendix~\ref{sec:algo}) that summarises the newly proposed algorithm, in the hope that this simplifies its use in applications.

\section{Setup and Motivation} \label{sec:Setup-Motivation}
Let $\big( \Omega, (\mathcal{F}_t)_{t \geq 0}, \mathcal{F}, \mathbb{P} \big)$ be a filtered probability space that satisfies the usual conditions. All probabilities are expressed with respect to this ambient probability space. 

\subsection{Model}
Let $\mathcal{O} \subset \mathbb{R}^n$, $n \in \mathbb{N}$, be a bounded polygonal domain. Moreover, let $u_0$, $g$ and $\sigma$ be deterministic and time-independent vector fields that prescribe initial condition, boundary condition, and noise coefficient, respectively. Finally, let $W$ be an $(\mathcal{F}_t)_{t \geq 0}$-Wiener process. The generalised Stokes system with inhomogeneous Dirichlet boundary condition and random transport noise is given by:
\begin{alignat}{2} \label{eq:gen-Stokes}
\begin{aligned}
    \dd u - \left[\Div S(\varepsilon u)  -  \nabla \zeta \right] \dd t&= (\sigma \cdot \nabla)  u \circ \dd W(t) \qquad &&\text{ in } (0,\infty) \times \mathcal{O}, \\
    \Div u &= 0 &&\text{ in } (0,\infty) \times \mathcal{O}, \\
    u &= g &&\text{ in } (0,\infty) \times \partial \mathcal{O}, \\
    u(0) &= u_0 &&\text{ in } \mathcal{O},
\end{aligned}
\end{alignat}
where $u$, $S$, $\varepsilon u$, and $\zeta$ denote velocity, viscous stress tensor, strain rate tensor, and pressure, respectively. The stochastic integral is interpreted in the Stratonovich sense. The first and second equations encode the conservation of momentum and incompressibility of the fluid, respectively. 

The velocity is a random vector field $u: \Omega \times [0,\infty) \times \mathcal{O} \to \mathbb{R}^n$; the pressure is a random scalar $\zeta:  \Omega \times [0,\infty) \times \mathcal{O} \to \mathbb{R}$; the random matrix-valued strain rate tensor $\varepsilon u:\Omega \times [0,\infty) \times \mathcal{O} \to \mathbb{R}^{n\times n}$ is the symmetric gradient of velocity, that is, $\varepsilon u := \tfrac{\nabla u + (\nabla u)^T}{2}$; and the matrix-valued viscous stress tensor $S$ is a function of the strain rate tensor which depends on the fluid rheology. In this article, we solely consider power-law fluids with the following rheology: 
\begin{align} \label{eq:def-S}
    S(A) := (\kappa + \abs{A}^2)^{(p-2)/2} A, \qquad A \in \mathbb{R}^{n \times n},~\kappa \geq 0,~ p \in (1,\infty).
\end{align}
We will frequently use the following tensor:
\begin{align*}
    V(A) := (\kappa + \abs{A}^2)^{(p-2)/4} A, \quad A \in \mathbb{R}^{n \times n},~ \kappa \geq 0,~ p \in (1,\infty),
\end{align*}
which is closely related to the viscous stress tensor~$S$; see Lemma~\ref{lem:relation-of-tensors}.

Instead of discussing the generalised Stokes system with inhomogeneous Dirichlet boundary condition, we investigate the generalised Stokes system with homogeneous Dirichlet boundary condition but shifted spatial operators and time-integrated pressure. Introducing the new variables $v := u - g$ and $\pi :=\int \zeta \dd t $, we reformulate~\eqref{eq:gen-Stokes} equivalently as follows: 
\begin{alignat}{2} \label{eq:gen-Stokes-modified}
\begin{aligned}
    \dd v - \Div S(\varepsilon v + \varepsilon g) \dd t +  \dd \nabla \pi  &= (\sigma \cdot \nabla)  (v + g) \circ \dd W(t) \quad &&\text{ in } (0,\infty) \times \mathcal{O}, \\
    \Div v &= 0 &&\text{ in } (0,\infty) \times \mathcal{O}, \\
    v &= 0 &&\text{ in } (0,\infty) \times \partial \mathcal{O}, \\
    v(0) &= v^\mathrm{in} := u_0 - g &&\text{ in } \mathcal{O},
\end{aligned}
\end{alignat}
where we implicitly used that $g$ doesn't depend on time and is solenoidal. The unknowns of the transformed system of equations are the vector field~$v$ and time-integrated pressure~$\pi$. The lack of regularity for the original pressure motivates its substitution with the time-integrated pressure, which typically behaves better; see, e.g.,~\cite{Wichmann2024} and the references therein.  

Before we provide a formal argument on why the system~\eqref{eq:gen-Stokes-modified} has an invariant measure, we shortly introduce notations, classical function spaces, and the data assumption used throughout this article. 

\subsection{Notation}
The euclidean inner products for vectors and matrices are denoted by $a\cdot b := \sum_{j=1}^n a_j b_j$ and $A:B := \sum_{i,j=1}^n A_{i,j} B_{i,j}$, respectively. We write $f \lesssim g$ for two non-negative quantities $f$ and $g$ if $f$ is bounded by $g$ up to a multiplicative constant. If the constant depends on a parameter~$q$, then we write $f \lesssim_q g$. Accordingly we define $\gtrsim$ and $\eqsim$. We denote by $c$ and $C$ generic constants which can change their values from line to line. 

\subsection{Function spaces}
We don't distinguish scalar-, vector-, and matrix-valued functions. Let $C^\infty(\mathcal{O})$ be the space of smooth vector fields. Similarly, let $C^\infty_c(\mathcal{O})$ be the space of smooth and compactly supported vector fields. We incorporate incompressibility and boundary condition for velocity and mean-free condition for pressure in the functional analytic framework; i.e., we define velocity and pressure spaces, for $k \in \mathbb{N}_0$ and $q \in [1,\infty]$, by:
\begin{align*}
\mathbb{W}^{k,q}_{0}(\mathcal{O}) &:=\overline{\left\{ u \in C^\infty_c(\mathcal{O}) \right\}}^{\norm{\cdot}_{W^{k,q}(\mathcal{O})}}, \\
\mathbb{W}^{k,q}_{0,\Div}(\mathcal{O}) &:=\overline{\left\{ u \in C^\infty_c(\mathcal{O}):~\Div u = 0 ~\text{ in } \mathcal{O} \right\}}^{\norm{\cdot}_{W^{k,q}(\mathcal{O})}}, \\
\mathbb{W}^{k,q}_{\Div}(\mathcal{O}) &:=\overline{\left\{ u \in C^\infty(\mathcal{O}):~\Div u = 0 ~\text{ in } \mathcal{O}\right\}}^{\norm{\cdot}_{W^{k,q}(\mathcal{O})}},
\end{align*}
and
\begin{align*}
L^q_0(\mathcal{O}) &:=\overline{\left\{ \pi \in C^\infty(\mathcal{O}):~\int_{\mathcal O}\pi \dd x = 0 \right\}}^{\norm{\cdot}_{L^{q}(\mathcal{O})}},
\end{align*}
respectively. Here, $\norm{\cdot}_{L^q(\mathcal{O})}$ and $\norm{\cdot}_{W^{k,q}(\mathcal{O})}$ denote the classical Lebesgue and Sobolev norms, respectively. For $q < \infty$, we abbreviate $\mathbb{L}^q(\mathcal{O}) = \mathbb{W}^{0,q}_{0}(\mathcal{O})$ and $\mathbb{L}^q_{\Div}(\mathcal{O}) = \mathbb{W}^{0,q}_{\Div}(\mathcal{O})$. We write $\left( \cdot, \cdot \right)$ for the inner products in $L^2(\mathcal{O})$, $\mathbb{L}^2(\mathcal{O})$, and $\big( \mathbb{L}^2(\mathcal{O}) \big)^n$.

We denote the space of strongly continuous and bounded functions on $\mathbb{L}^2(\mathcal{O})$ by~$C_b\big(\mathbb{L}^2(\mathcal{O}) \big)$. The Borel $\sigma$-algebra on $\mathbb{L}^2(\mathcal{O})$ is denoted by $\mathcal{B}\big( \mathbb{L}^2(\mathcal{O})\big)$.

\subsection{Data assumption} 
Throughout this article, we restrict ourselves to the following data.
\begin{assumption}[Data conditions] \label{ass:data-strict}
Let the data satisfy:
\begin{itemize}
        \item (boundary condition)~$g \in \mathbb{W}^{2,\infty}_{\Div}(\mathcal{O})$;
    \item (noise coefficient)~$\sigma \in \mathbb{W}^{1,\infty}_{\Div}(\mathcal{O})$;
    \item (analytic initial velocity)~$v^\mathrm{in} \in \mathbb{L}_{\Div}^2(\mathcal{O})$. 
\end{itemize}
\end{assumption}

\subsection{Long-term stability} \label{sec:long-stable-motiv}
The classical construction of an invariant measure via ergodic averages requires the system to be stable for arbitrary large times. This is the reason why, typically, only additive noises have been considered previously. Here, we give a short argument that motivates the existence of an invariant measure for transport noise.

An application of It\^o's formula for $u \mapsto \norm{u}^2_{L^2(\mathcal{O})}$ shows formally:
\begin{align*}
    &\norm{v_t}^2_{L^2(\mathcal{O})} + 2\int_0^t \int_{\mathcal{O}} S(\varepsilon v + \varepsilon g) : \varepsilon v \dd x \dd s \\
    &\quad = \norm{v_0}^2_{L^2(\mathcal{O})} + 2\int_0^t \int_{\mathcal{O}} (\sigma \cdot \nabla)g \cdot v \dd x \circ \dd W_s + 2\int_0^t \int_{\mathcal{O}} (\sigma \cdot \nabla)v \cdot v \dd x \circ \dd W_s.
\end{align*}
Importantly, since $\Div \sigma = 0$ it follows that
\begin{align*}
    \int_0^t \int_{\mathcal{O}} (\sigma \cdot \nabla)v \cdot v \dd x \circ \dd W_s = -  \int_0^t \int_{\mathcal{O}} \Div \sigma \frac{\abs{v}^2}{2} \dd x \circ \dd W_s = 0,
\end{align*}
which eliminates the non-linear stochastic integral in the energy identity. 

In particular, if the boundary condition is trivial, then the energy identity reads:
\begin{align*}
    \norm{v_t}^2_{L^2(\mathcal{O})} + 2\int_0^t \int_{\mathcal{O}} S(\varepsilon v) : \varepsilon v \dd x \dd s  = \norm{v_0}^2_{L^2(\mathcal{O})}.
\end{align*}
Thus, even though the dynamics of the solution is driven by a stochastic process, the energy of the system is dissipated pathwise and monotonically. The dissipation rate is quantified by the non-linear tensor~$V$, since
\begin{align*}
    \int_{\mathcal{O}} S(\varepsilon v) : \varepsilon v \dd x = \norm{V(\varepsilon v)}_{L^2(\mathcal{O})}^2.
\end{align*}

Inhomogeneous boundary conditions destroy the pathwise energy identity; instead, one finds an energy inequality: 
\begin{align} \label{eq:inequality-energy}
    \mathbb{E}\left[ \norm{v_t}^2_{L^2(\mathcal{O})} + \int_0^t  \norm{V(\varepsilon v + \varepsilon g)}_{\mathbb{L}^2(\mathcal{O})}^2 \dd s \right] \lesssim \mathbb{E}\left[\norm{v_0}^2_{L^2(\mathcal{O})} \right] + C(\sigma, g) t.
\end{align}
Crucially, the inequality grows at most linearly in the time horizon~$t$, which is sufficient for the derivation of an invariant measure.

For the derivation of Inequality~\eqref{eq:inequality-energy}, one needs to compute the expectation of the Stratonovich integral. Since It\^o integrals define centered martingales, their expectation vanishes and, thus, the expectation of the Stratonovich integral equals the expectation of the It\^o--Stratonovich corrector:
\begin{align*}
   \mathbb{E}{}&\left[ 2\int_0^t \int_{\mathcal{O}} (\sigma \cdot \nabla)g \cdot v \dd x \circ \dd W_s \right] \\
   ={}& \mathbb{E}\left[ \int_0^t \int_{\mathcal{O}} \mathfrak{P}[
   (\sigma \cdot \nabla)g] \cdot \mathfrak{P}[(\sigma \cdot \nabla) (v+g)] \dd x \dd s \right].
\end{align*}
Here, $\mathfrak{P}$ denotes the Helmholtz--Leray projection. Fortunately, the dissipation dominates the expectation of the Stratonovich integral, which eventually leads to~\eqref{eq:inequality-energy}. 

Certainly, many details are left aside and the arguments have to be justified for a rigorous derivation of the Inequality~\eqref{eq:inequality-energy}, as well as the energy identity for homogeneous boundary conditions. But the corresponding result for our numerical algorithm is rigorous and worked out in full detail; see Theorem~\ref{thm:apriori-vel} and its proof in Section~\ref{sec:stability-of-algo}.

\section{Discretisation} \label{sec:disc-tools}
In this section, we introduce an extension of the Gradient Discretisation Method (GDM) adapted to the stochastic generalised Stokes equations. Moreover, we discuss a special discretisation of the noise coefficient, which preserves key properties of the analytic solution on the discrete level. We close this section by presenting a new semi-implicit time-stepping algorithm, and establishing its stability for arbitrary large time horizons.

\subsection{Gradient scheme}
The GDM is a unified framework for the analysis of classical discretisation methods; see, e.g., the book~\cite{Droniou2018}. It includes in particular conforming and non-conforming finite elements, finite volumes, and mimetic finite differences. The GD for Stokes equations has been introduced in~\cite{Droniou2015}. We follow their approach but adapt the function spaces to our framework. In the remaining part of this section, let $p \in (1,\infty)$ be the fixed growth index of $S$; see~\eqref{eq:def-S}.

\begin{definition}[Gradient Discretisation (GD)]\label{def:gradient-discretisation}
A Gradient Discretisation~$\disc$, in short: GD, is defined by 
\begin{align*}
\disc = (X_{\disc,0}, Y_\disc, \nabla_\disc,\varepsilon_\disc,\chi_\disc, \Div_\disc,\Pi_\disc ),
\end{align*}
where:
\begin{enumerate}
\item The set of discrete velocity unknowns $X_{\disc,0}$ is a finite dimensional real vector space;
\item The set of discrete pressure unknowns $Y_\disc$ is a finite dimensional real vector space;
\item The linear mapping $\nabla_\disc:X_{\disc,0}\to \big(L^p(\mathcal{O})\big)^{n\times n}$ is the reconstruction of the discrete gradient;
\item \label{it:gdm-velocity-norm} The linear mapping $\varepsilon_\disc: X_{\disc,0} \to \big(L^p(\mathcal{O})\big)^{n\times n}$ is the reconstruction of the discrete symmetric gradient. It must be chosen such that $\Vert \varepsilon_\disc \cdot \Vert_{L^p(\mathcal{O})}$ is a norm on $X_{\disc,0}$;
\item \label{it:gdm-pressure-norm} The linear mapping $\chi_\disc: Y_\disc \to L^{p \wedge p'}(\mathcal{O})$  (with $p\wedge p':=\min(p,p')$ and $p' := p /(p-1)$) is the reconstruction of the approximate pressure. It must be chosen such that $\Vert \chi_\disc \cdot \Vert_{ L^{p \wedge p'}(\mathcal{O})}$ is a norm on $Y_{\disc,0}$, where $Y_{\disc,0}= \{ q \in Y_{\disc,0}: \int_{\mathcal{O}} \chi_\disc q \dd x = 0\}$ denotes the set of mean-value free discrete pressure unknowns;
\item The linear mapping $\Div_\disc: X_{\disc,0} \to L^p(\mathcal{O})$ is the reconstruction of the discrete divergence. We then define $E_{\disc,0} := \{ v \in X_{\disc,0}: \, \forall q \in Y_{\disc,0}: \left( \chi_\disc q, \Div_\disc v \right) = 0\} $ as the set of discretely divergence-free velocity unknowns;
\item The linear mapping $\Pi_\disc:X_{\disc,0}\to \big(L^2(\mathcal{O})\big)^{ n}$ is the reconstruction operator of the approximate velocity. It must be chosen such that $\Pi_\disc$ is one-to-one on~$E_{\disc,0}$.
\end{enumerate}
\end{definition}
Note that the one-to-one condition on $\Pi_\disc$ is one of the main differences of our framework compared to usual GDs; it is used in particular to ensure that the inverse inequality \eqref{eq:inverse-first} makes sense. In practice, however, this injectivity property is not a very restrictive assumption.

Reconstructions of different operators don't necessarily need to be strictly related to each other. But some notion of compatibility is needed to ensure well-posedness of algorithms based on GDs. Next, we introduce these compatibility conditions. They should be read as properties of a particular GD and need to be verified on a case-by-case basis. For the remaining part of this section, let $\disc$ be a GD. 

\begin{definition}[Coercivity constant] The \emph{'coercivity constant'}~$C_{\disc}(p)$ is defined by:
 \begin{align} \label{def:discrete-poincare}
C_{\disc}(p) := \sup_{v \in X_{\disc,0} } \frac{\norm{\Pi_\disc v}_{L^{p}(\mathcal{O})} + \norm{\Div_\disc v}_{L^{p}(\mathcal{O})} + \norm{\nabla_\disc v}_{L^{p}(\mathcal{O})}}{\norm{\varepsilon_\disc v}_{L^p(\mathcal{O})}}.
\end{align}  
\end{definition}

\begin{remark}
The coercivity constant quantifies in particular two well-known constants in the context of GDs: the Korn constant and the Poincaré constant. 
\end{remark}

\begin{definition}[Inf-sup constant] The \emph{'inf-sup constant'}~$\beta_\disc(p)$ is defined by:
\begin{align} \label{def:discrete-LBB}
\beta_\disc(p) := \inf_{q \in Y_{\disc,0} } \sup_{v \in X_{\disc,0} } \frac{ \left(\chi_\disc q, \Div_\disc v \right)}{\norm{\chi_\disc q}_{L^{p \wedge p'}(\mathcal{O})}(\norm{\varepsilon_\disc v}_{L^p(\mathcal{O})} + \norm{\varepsilon_\disc v}_{L^{p'}(\mathcal{O})})}.
\end{align}
The GD is called \emph{'inf-sup-stable'} if~$\beta_\disc(p) > 0$.
\end{definition}

\begin{definition}[Inverse estimate constant] \label{def:inverse-estimates}
The \emph{'inverse estimate constant'}~$\mathfrak{B}_\disc(p)$ is defined by:
\begin{align}
\label{eq:inverse-first}
       \mathfrak{B}_\disc(p) :={}& \sup_{v \in X_{\disc,0}} \frac{\norm{ \varepsilon_\disc v}_{L^p(\mathcal{O})}}{\norm{\Pi_\disc v}_{L^2(\mathcal{O})}}.
\end{align}
\end{definition}

\begin{remark}
This constant estimates the discrete derivative (of a discrete vector or a function) in terms of the Lebesgue norm of the vector/function. As a consequence, we expect it to blow up by a constant depending on the dimension of the discretisation space (for mesh-based methods, this would for example be as the inverse of the mesh size).
\end{remark}

\begin{definition}[$L^2$-projection onto discretely divergence free velocity] \label{def:l2-projection}
The \emph{'$L^2$-projection onto~$E_{\disc,0}$'} (or simply \emph{'$L^2$-projection'}) $\mathcal{K}_\disc : \mathbb{L}^2(\mathcal{O}) \to E_{\disc,0}$ is defined by: 
\begin{align} \label{eq:l2-projection-orthogonal}
    \forall \xi \in E_{\disc,0}: \quad \left( u - \Pi_\disc \mathcal{K}_\disc u, \Pi_\disc \xi \right) = 0.
\end{align}
The \emph{'Sobolev-stability constant of the $L^2$-projection'}~$C^{\mathrm{Sob}}_\disc(p) $ is defined by:
\begin{align} \label{eq:l2-proj-sobolev}
    C^{\mathrm{Sob}}_\disc(p) := \sup_{z \in \mathbb{W}^{1,p}(\mathcal{O})} \frac{\norm{\varepsilon_\disc \mathcal{K}_\disc z}_{L^{p}(\mathcal{O})}}{\norm{z}_{W^{1,p}(\mathcal{O})}}.
\end{align}
\end{definition}

\begin{remark}
The $L^2$-projection is a powerful tool in the analysis of numerical schemes of parabolic evolution equations and has been studied extensively for finite element spaces; see, e.g.,~\cite{Diening2021,Boman2006,Bank2013} and the references therein. However, its stability for constrained subspaces, such as $E_{\disc,0}$, remains to be explored. 
\end{remark}

Having introduced the compatibility conditions, we can state our choice of GDs, which we will assume from now on. 
\begin{assumption}[GD condition]
We assume that $\disc$ is $\Gamma$-stable, in the sense that there exists a constant $\Gamma \in [1,\infty)$ such that 
\begin{itemize}
    \item (controlled coercivity constants) $C_\disc(p) + C_\disc(p') \leq \Gamma$;
    \item (controlled inf-sup constant) $\beta_\disc(p) \geq \Gamma^{-1}$;
    \item (controlled Sobolev-stability constants) $C^{\mathrm{Sob}}_\disc(p) + C^{\mathrm{Sob}}_\disc(p') \leq \Gamma$.
\end{itemize}
\end{assumption}

\subsection{Noise coefficient}
An important ingredient in the energy evolution of the analytic solution is the cancellation of the stochastic integral: 
\begin{align*}
    \int_0^t \int_{\mathcal{O}} (\sigma \cdot \nabla)v \cdot v \dd x \circ \dd W_s = -  \int_0^t \int_{\mathcal{O}} \Div \sigma \frac{\abs{v}^2}{2} \dd x \circ \dd W_s = 0.
\end{align*}
This identity heavily relies on the validity of the integration by parts formula and~$\Div \sigma =~0$. However, an integration by parts formula is unavailable for generic GDs. We overcome this issue by invoking a symmetrisation strategy due to Temam~\cite{Temam1968}: First integrating by parts on the analytic level and then discretising yields
\begin{align} \nonumber
    \int_{\mathcal{O}} (\sigma \cdot \nabla)v \cdot \xi \dd x &= \frac{1}{2} \int_{\mathcal{O}} (\sigma \cdot \nabla)v \cdot \xi \dd x- \frac{1}{2} \int_{\mathcal{O}} (\sigma \cdot \nabla)\xi \cdot v \dd x \\ \label{def:discrete-noise-coefficient}
    &\approx \frac{1}{2} \left( \sigma^T\nabla_\disc v, \Pi_\disc \xi \right) - \frac{1}{2}  \left(  \Pi_\disc v, \sigma^T\nabla_\disc \xi  \right) =:  B_\disc^{\sigma}( v, \xi).
\end{align}
The bi-linear map~$B_\disc^{\sigma}:~X_{\disc,0} \times X_{\disc,0} \to \mathbb{R}$ is constructed such that it vanishes on its diagonal. This will eventually lead to the cancellation of the stochastic integral on the discrete level.

\subsection{Algorithm}
We propose the following fully-discrete algorithm, which combines the GDM with a semi-implicit time-stepping algorithm -- also called Crank--Nicolson scheme due to their seminal work~\cite{Crank1947} -- started at a suitably projected initial state. 

We assume that the analytic initial velocity~$v^\mathrm{in} \in \mathbb{L}^2(\mathcal{O})$ is given. 

\underline{Step 1:~Initialisation.} Since the Crank--Nicolson scheme doesn't introduce artificial numerical dissipation, it is sensitive with respect to initialisation, i.e., failure of matching constraints such as incompressibility or boundary conditions at initial time will persist at all times. At initialisation, we ensure validity of both properties by utilising the discrete Helmholtz decomposition. Let $(v_\disc^0, \pi_\disc^0) \in X_{\disc,0} \times Y_{\disc,0}$ be defined by:
\begin{subequations} \label{eq:disc-Helmholtz}
\begin{alignat}{2} \label{eq:disc-Helmholtz-01}
&\forall \xi \in X_{\disc,0}:\qquad \left( \Pi_\disc v^0_\disc, \Pi_\disc \xi \right) + \left( \chi_\disc \pi^0_\disc, \Div_\disc \xi \right) &&= \left( v^\mathrm{in}, \Pi_\disc \xi \right), \\  \label{eq:disc-Helmholtz-02}
 &\forall q \in Y_{\disc,0}:\qquad \left(  \Div_\disc v^0_\disc,\chi_\disc q \right) &&= 0.
\end{alignat}
\end{subequations}
Notice that, since $\Pi_\disc$ is one-to-one, the discrete Helmholtz decomposition is well-defined for any inf-sup stable GD. Moreover, $\norm{\Pi_\disc v_\disc^0}_{L^2(\mathcal{O})} \leq \norm{ v^{\mathrm{in}}}_{L^2(\mathcal{O})}$.

\underline{Step 2:~Time-stepping.} Let $\tau > 0$ be the time step size. For $n \in \mathbb{N}_0$, we recursively seek $(v_\disc^{n+1}, \pi_\disc^{n+1}) \in X_{\disc,0} \times Y_{\disc,0}$ such that $\mathbb{P}$-a.s.~for all $(\xi,q) \in X_{\disc,0} \times Y_{\disc,0}$:
\begin{subequations} \label{eq:time-stepping}
\begin{align} \label{eq:Evolution-01}
&\left( \Pi_\disc v^{n+1}_{\disc} -\Pi_\disc v^n_\disc, \Pi_\disc \xi \right) + \tau \left( S( \varepsilon_\disc v^{n+1/2}_\disc + \varepsilon g ), \varepsilon_\disc \xi \right)  \\ \nonumber
&\quad - \left( \chi_\disc \pi^{n+1}_{\disc} - \chi_\disc \pi^n_{\disc}, \Div_\disc \xi \right)  = \left[ B_\disc^\sigma( v^{n+1/2}_\disc, \xi)+ \left( (\sigma \cdot \nabla)g , \Pi_{\disc} \xi \right) \right] \Delta_{n+1} W ,\\ \label{eq:Evolution-02}
 & \left(  \Div_\disc v^{n+1/2}_\disc,\chi_\disc q \right) = 0,
\end{align}
\end{subequations}
where $v^{n+1/2}_\disc:= \frac{v^{n+1}_\disc + v^{n}_\disc}{2}$, $\Delta_{n+1} W := W(t_{n+1}) - W(t_n)$, and $t_n = n \tau$.

Solving this semi-implicit scheme requires us to find the solution to a non-linear system of equations, which itself is a computationally demanding task; nevertheless, we favour the semi-implicit scheme since it enables improved stability estimates for the full range of growth rates~$p \in (1,\infty)$. A summary of the tools needed for the implementation of the algorithm is given in Appendix~\ref{sec:algo}.

From a theoretical perspective, the well-posedness of the time-stepping scheme follows from standard monotone operator theory and the Banach-Necas-Babuska Theorem; see, e.g.,~\cite[Section~3]{MR3607728} and~\cite[Theorem~2.6]{Ern2004}, respectively. The former relies on the cancellation of the noise coefficient~$B^\sigma_\disc$ on its diagonal; and the latter requires an inf-sup stable GD. 

\subsection{Stability} \label{sec:stability}
While the existence is a first step towards concrete approximation results, quantitative stability results are needed for the construction and identification of limit variables such as limit velocity and pressure, and an invariant measure -- the steady state of the stochastic evolution dynamics. For the construction of the invariant measure, it is particularly important to trace the dependence of the stability constants on the time horizon. To keep the exposition at a reasonable length, we restrict ourselves to the discussion of velocity. 

The next theorem addresses the long-term stability of approximate velocity on two time scales: the evolution at integer and shifted integer times. Especially the second scale will become important for the construction of a limit measure -- a candidate for the invariant measure.   
\begin{theorem} \label{thm:apriori-vel}
Let $q > 0$ be a moment of interest.
There exists a constant $C >0$ such that for all $N \in \mathbb{N}$:
\begin{align} 
\begin{aligned}\label{eq:apriori-vel}
&\mathbb{E}\left[ \max_{n \leq N} \norm{\Pi_\disc v^{n}_\disc}_{L^2(\mathcal{O})}^{2q} + \left( \sum_{n=0}^{N-1} \tau  \norm{\varepsilon_\disc v^{n+1/2}_\disc}_{L^p(\mathcal{O})}^p \right)^{q} \right] \\
&\hspace{12em} \leq C \left(  \norm{ v^\mathrm{in}}_{L^2(\mathcal{O})}^{2q} + (\tau N)^q \right).
\end{aligned}
\end{align} 
Here $C$ depends on the selected moment, noise coefficient, boundary condition, $p$ and $\kappa$, volume of domain, and $\Gamma$-stability of the GD.
\end{theorem}

\begin{remark}
\begin{itemize}
    \item We want to emphasise that homogeneous boundary conditions lead to the pathwise energy equality: $\mathbb{P}$-a.s for all $N \in \mathbb{N}$,
\begin{align}\label{eq:path-stab}
\norm{\Pi_\disc v^N_\disc}_{L^2(\mathcal{O})}^2 + 2\sum_{n=0}^{N-1} \tau \norm{V(\varepsilon_\disc v^{n+1/2}_\disc)}_{L^2(\mathcal{O})}^2 =  \norm{\Pi_\disc \mathcal{K}_\disc v^\mathrm{in}}_{L^2(\mathcal{O})}^2.
\end{align}
Moreover, the stability assumption on the $L^2$-projection can be dropped in this case. 
\item Velocity differences always satisfy a pathwise energy equivalence, regardless of the boundary conditions; indeed, let $(v^n_\disc)_{n \in \mathbb{N}_0}$  and $(w^n_\disc)_{n \in \mathbb{N}_0}$ be velocity fields generated by~\eqref{eq:disc-Helmholtz} and~\eqref{eq:time-stepping} when started in $v^\mathrm{in}$ and $w^\mathrm{in}$, respectively. Then, $\mathbb{P}$-a.s.~for all $N \in \mathbb{N}$,
\begin{align*} 
    \norm{\Pi_\disc v^N -\Pi_\disc w^N }_{L^2(\mathcal{O})}^2 &+ \sum_{n = 0}^{N-1} \tau \norm{ V(\varepsilon_\disc v^{n+1/2} + \varepsilon g) - V(\varepsilon_\disc w^{n+1/2} + \varepsilon g)}_{L^2(\mathcal{O})}^2 \\
    &\eqsim  \norm{\Pi_\disc \mathcal{K}_\disc v^\mathrm{in} -\Pi_\disc \mathcal{K}_\disc w^\mathrm{in} }_{L^2(\mathcal{O})}^2.
\end{align*}
\end{itemize}
\end{remark}

\section{Long-term dynamics} \label{sec:long-term}
The long-term dynamics of our algorithm are encoded in the invariant measure of the algorithm's discrete transition semigroup, which we define in this section. Building on two time scales of the algorithm, we propose and investigate two sequences of measures that are promising candidates for the construction of the invariant measure. However, each sequence lacks one important feature: either the existence of a limit measure, or the invariance with respect to the semigroup. We derive an abstract condition on the distance of consecutive velocities that enables us to merge the desired properties of both sequences, recovering the existence of an invariant measure. We close the section by discussing consequences of homogeneous boundary conditions.

\subsection{Transition semigroups}
In order to trace the dependence of our algorithm on the initial condition, we introduce the following operators.
\begin{definition}[Solution operators] 
Let~$v^\mathrm{in} \in \mathbb{L}^2(\mathcal{O})$. Moreover, let $(v_\disc^n)_{n\in \mathbb{N}_0}$ be the velocity generated by~\eqref{eq:disc-Helmholtz} and~\eqref{eq:time-stepping} when initialised in~$v^\mathrm{in}$. 
\begin{itemize}
    \item (integer operator) Let~$\mathfrak{S}:\mathbb{N}_0 \times \mathbb{L}^2(\mathcal{O}) \to E_{\disc,0}$ be defined by 
    \begin{align*}
        \mathfrak{S}(n,v^\mathrm{in}) := v_\disc^n;
    \end{align*}
    \item (shifted integer operator) Let~$\mathfrak{S}^{1/2}:\mathbb{N}_0 \times \mathbb{L}^2(\mathcal{O}) \to E_{\disc,0}$ be defined by 
    \begin{align*}
        \mathfrak{S}^{1/2}(n,v^\mathrm{in}) := \frac{\mathfrak{S}(n+1,v^\mathrm{in}) + \mathfrak{S}(n,v^\mathrm{in})}{2}.
    \end{align*}
\end{itemize}
\end{definition}
These operators map a discrete time index~$n \in \mathbb{N}_0$ (corresponding to the continuous time~$t_n=n\tau$) and an analytic initial velocity~$v^\mathrm{in}$ to the approximate velocity at the $n$-th index and the arithmetic mean of consecutive approximate velocities, respectively.

With the help of the integer operator, we define the discrete transition semigroup that propagates over discrete time the distribution of reconstructed velocity.  
\begin{definition}[Discrete transition semigroup] \label{def:discrete-trans-group}
Let $\mathcal{P}_{\tau,\disc}$ be defined by 
\begin{align*} 
    (\mathcal{P}^n_{\tau,\disc} f)(v^\mathrm{in}) := \mathbb{E} \left[ f\big(\Pi_\disc \mathfrak{S}(n,v^\mathrm{in}) \big) \right], \quad n \in \mathbb{N}_0,~f \in C_b\big( \mathbb{L}^2(\mathcal{O}) \big),~v^\mathrm{in} \in \mathbb{L}^2(\mathcal{O}).
\end{align*}
It is called the \emph{'discrete transition semigroup'} (or simply \emph{'discrete semigroup'}).
We write~$\mathcal{P}_{\tau,\disc}$ for both the semigroup and the semigroup evaluated at $1$.
\end{definition}
In the above definition, we have already indicated (by denoting the dependence on the time index as a power instead of an argument) that $\mathcal{P}_{\tau,\disc}$ is indeed a semigroup. The next lemma justifies this.
\begin{lemma} \label{lem:semigroup}
The operator~$\mathcal{P}_{\tau,\disc}$ is a semigroup, i.e.,
\begin{itemize}
    \item for all $n \in \mathbb{N}$, $f \in C_b\big(\mathbb{L}^2(\mathcal{O})\big)$ and $v^\mathrm{in} \in \mathbb{L}^2(\mathcal{O})$: 
    \begin{align*}
        (\mathcal{P}^n_{\tau,\disc} f)(v^\mathrm{in}) = \big(\mathcal{P}^{n-1}_{\tau,\disc} (\mathcal{P}_{\tau,\disc}f)\big)(v^\mathrm{in});
    \end{align*}
    \item for all $f \in C_b\big(\mathbb{L}^2(\mathcal{O})\big)$ and $v^\mathrm{in} \in \Pi_\disc E_{\disc, 0}$: 
    \begin{align*}
         (\mathcal{P}^0_{\tau,\disc} f)(v^\mathrm{in}) =  f(v^\mathrm{in}).
    \end{align*}
\end{itemize}
\end{lemma}  

\begin{remark}
For initial velocity that doesn't belong to $\Pi_\disc E_{\disc, 0}$, the semigroup evaluated at~$0$ doesn't act as an identity. Instead, one obtains
\begin{align*}
     (\mathcal{P}^0_{\tau,\disc} f)(v^\mathrm{in}) = f(\Pi_\disc \mathcal{K}_\disc v^\mathrm{in}),
\end{align*}
where $\mathcal{K}_\disc$ is the $L^2$-projection onto~$E_{\disc,0}$.
\end{remark}

At this point, we have defined the discrete transition semigroup. The next natural questions are: 
\begin{itemize}
    \item For fixed time step size and GD, does there exist an invariant measure~$\mu_{\tau,\disc}$ with respect to~$\mathcal{P}_{\tau,\disc}$?
    \item If yes, can we construct it? 
\end{itemize}
Before we discuss first steps towards the construction of invariant measures, let us recall what it means to be invariant.
\begin{definition}[Invariant measure]
Let $\mathcal{P}$ be a discrete semigroup. A measure $\mu$ on $\mathbb{L}^2(\mathcal{O})$ is called \emph{'invariant with respect to $\mathcal{P}$'} if for all $n \in \mathbb{N}$ and $f \in C_b\big( \mathbb{L}^2(\mathcal{O})\big)$:
\begin{align*}
  \int_{\mathbb{L}^2(\mathcal{O})} (\mathcal{P}^n f)(v) \mu ( \mathrm{d} v) = \int_{\mathbb{L}^2(\mathcal{O})} f(v) \mu( \mathrm{d} v).
\end{align*}
\end{definition}
If the invariant measure associated to the discrete semigroup exists, then it is unique and hence even ergodic. 
\begin{theorem}[Uniqueness] \label{thm:unique}
There exists at most one invariant measure with respect to $\mathcal{P}_{\tau,\disc}$. 
\end{theorem}

\subsection{Constructing the invariant measure}
Related to the two time scales of our algorithm, we choose the following two sequences of input-velocity-dependent measures as our ansatz:
\begin{subequations}
\begin{align} \label{def:aprox-measure}
    \mu_{\tau,\disc}^{N}( v^\mathrm{in}; A) &:= \frac{1}{N} \sum_{n = 0}^{N-1} \mathbb{P}\left(\Pi_\disc \mathfrak{S}(n,v^\mathrm{in}) \in A \right), \\ \label{def:aprox-measure-shift}
     \mu_{\tau,\disc}^{1/2,N}( v^\mathrm{in}; A) &:= \frac{1}{N} \sum_{n = 0}^{N-1} \mathbb{P}\left(\Pi_\disc \mathfrak{S}^{1/2}(n,v^\mathrm{in}) \in A \right),
\end{align}
\end{subequations}
 for $N \in \mathbb{N}$, $v^\mathrm{in} \in \mathbb{L}^2(\mathcal{O})$, and $A \in \mathcal{B}\big( \mathbb{L}^2(\mathcal{O})\big)$.

 While the first sequence of measures is asymptotically (for large $N$) invariant with respect to the discrete semigroup, it is unclear if the sequence has an accumulation point (with respect to the weak convergence of measures); in contrast, the second sequence of measures has an accumulation point, which gives rise to a limit measure. However, for this limit measure it remains open if it is invariant with respect to the semigroup. 

The asymptotic invariance of the first sequence of measures with respect to the discrete semigroup is the content of the following lemma. 
\begin{lemma}[Asymptotic invariance of 1st sequence] \label{lem:rel-1st-measure-semigroup}
For all $f \in C_b\big( \mathbb{L}^2(\mathcal{O}) \big)$ there exists a constant $C_f$ such that for all $n$ and $N \in \mathbb{N}$:
\begin{align*}
 \abs{\int_{\mathbb{L}^2(\mathcal{O})} (P_{\tau,\disc}^n f)(v) \mu_{\tau,\disc}^N (v^\mathrm{in}; \mathrm{d} v)  - \int_{\mathbb{L}^2(\mathcal{O})} f(v) \mu_{\tau,\disc}^{N} (v^\mathrm{in}; \mathrm{d} v)} \leq C_f \frac{n}{N}.
\end{align*}
\end{lemma}

\begin{remark}
    Thanks to Lemma~\ref{lem:rel-1st-measure-semigroup}, any accumulation point of the sequence of measures~$\big( \mu_{\tau,\disc}^{N}( v^\mathrm{in}; \cdot) \big)_{N \in \mathbb{N}}$ will be invariant with respect to~$P_{\tau,\disc}$. 

    Lemma~\ref{lem:rel-1st-measure-semigroup} also applies to joint variations in the time horizon and time step size as long as the analytic time horizon reaches infinity. Specifically, taking a sequence of time step refinements $(\tau_k)_{k\in\mathbb{N}}$ and fixing a discrete time $t=n_k\tau_k$, \eqref{lem:rel-1st-measure-semigroup} applied to $n=n_k$ shows the asymptotic invariance of the transition semigroup at time $t$ provided that $n_k/N_k=t/(N_k\tau_k)\to 0$, which means that the analytic time horizon $N_k\tau_k$ must tend to $+\infty$.
\end{remark}

 The next lemma guarantees the existence of a limit measure for the second sequences of measures. It relies on the time-global stability of the reconstructed symmetric gradient, which is only available at shifted integer times; see~\eqref{eq:apriori-vel}. 
\begin{lemma}[Existence of limit measure for 2nd sequence] \label{lem:limit-measure}
There exist a subsequence $(N_k)_{k \in \mathbb{N}}$ and a probability measure~$\mu_{\tau,\disc}^{1/2}(v^\mathrm{in};\cdot)$ on $\mathbb{L}^2(\mathcal{O})$ such that 
\begin{align*}
    \forall f \in C_b\big(\mathbb{L}^2(\mathcal{O}) \big): \quad \int_{\mathbb{L}^2(\mathcal{O})} f( v) \mu_{\tau,\disc}^{1/2,N_k}(v^\mathrm{in}; \mathrm{d} v) \overset{k \to \infty}{\rightarrow} \int_{\mathbb{L}^2(\mathcal{O})} f( v) \mu_{\tau,\disc}^{1/2}(v^\mathrm{in}; \mathrm{d} v).
\end{align*}
\end{lemma}
\begin{remark}
Lemma~\ref{lem:limit-measure} remains true for joint variations in the time horizon and time step size as long as the analytic time horizon is uniformly non-degenerate, i.e., $\tau_k N_k \gtrsim 1$ for all $k\in\mathbb{N}$. Moreover, if the considered GD is compact (see \cite[Definition 2.8]{Droniou2018}), then Lemma~\ref{lem:limit-measure} provides the existence of a limit measure for the following cases: 
\begin{itemize}
    \item fully discrete case ($\tau$ and $\disc$ fixed, $N\to \infty$); 
    \item time discrete case ($\tau$ fixed, $\disc \rightarrow 0$ and $N\to \infty$);
    \item space discrete case ($\disc$ fixed, $\tau \to 0$ and $N\to \infty$ such that $\tau N \gtrsim 1$);
    \item continuous case ($\disc \rightarrow 0$, $\tau \to 0$ and $N\to \infty$ such that $\tau N \gtrsim 1$).
\end{itemize}
Above, $\disc\to 0$ means that we have sent the spatial discretisations parameters to the limit and that the considered model is therefore the continuous-in-space model.
\end{remark}
One might hope that the discrete limit measure constructed above is invariant with respect to the discrete semigroup. However, a verification of its invariance is non-trivial, since the limit measure and the semigroup depend on the shifted integer solution operator and integer solution operator, respectively. At the moment, we are unable to close this gap in full satisfaction. But we obtain a partial answer: a sufficient condition that guarantees the invariance of the limit measure is a consequence of the next lemma. 
\begin{lemma}[Mismatch of invariance for 2nd sequence] \label{lem:mismatch-invariance}
For all Lipschitz-continuous $f \in C_b\big(\mathbb{L}^2(\mathcal{O}) \big)$ there exists a constant $C_{f}$ such that for all $n$ and $N \in \mathbb{N}$:
    \begin{align*}
         &\abs{\int_{\mathbb{L}^2(\mathcal{O})} \mathcal{P}_{\tau,\disc}^n f( v) \mu_{\tau,\disc}^{1/2,N}(v^\mathrm{in}; \mathrm{d} v) - \int_{\mathbb{L}^2(\mathcal{O})} f( v) \mu_{\tau,\disc}^{1/2,N}(v^\mathrm{in}; \mathrm{d} v)} \\
         &\hspace{6em} \leq C_{f} \left( \frac{n}{N} + \sqrt{\mathbb{E}\left[ \frac{1}{N} \sum_{\ell=0}^{N-1} \norm{\Pi_\disc v^{\ell+1/2}_\disc - \Pi_\disc v^\ell_\disc }_{L^2(\mathcal{O})}^2  \right]} \right).
    \end{align*}
\end{lemma}

\begin{remark}
    Notice that differences of integer velocity and shifted integer velocity are proportional to differences of integer velocities, since~$v^{\ell+1}_\disc = 2 v^{\ell + 1/2}_\disc - v^{\ell}_\disc$ for all $\ell \in \mathbb{N}_0$. Therefore,
    \begin{align} \label{ass:accum-dif}
     \lim_{N\to \infty} \mathbb{E}\left[ \frac{1}{N} \sum_{\ell=0}^{N-1} \norm{\Pi_\disc v^{\ell+1}_\disc - \Pi_\disc v^\ell_\disc }_{L^2(\mathcal{O})}^2  \right] =0
    \end{align}
    is a sufficient condition for the invariance of the limit measure.
\end{remark}
 In most situations, we don't expect~\eqref{ass:accum-dif} to be satisfied; instead, motivated by our numerical experiments (see Figure~\ref{fig:incEvo} and Table~\ref{tab:eoc}), we expect that the mean time-averaged difference of velocity converges to a time step size-dependent constant:
\begin{align} \label{eq:lim-difference-tau}
    \lim_{N\to \infty} \mathbb{E}\left[ \frac{1}{N} \sum_{\ell=0}^{N-1} \norm{\Pi_\disc v^{\ell+1}_\disc - \Pi_\disc v^\ell_\disc }_{L^2(\mathcal{O})}^2  \right] = \mathfrak{C}_\disc(\tau).
\end{align}
While~\eqref{eq:lim-difference-tau} is insufficient for the verification of the $\mathcal{P}_{\tau,\disc}$-invariance of the limit measure of the second sequence of measures, it will eventually enable the construction of time-continuous velocity and a corresponding time-continuous semigroup. In the time-continuous framework, the two time scales of our algorithm will be merged, so that both -- existence of the limit measure and its invariance -- will hold simultaneously. The analysis of the limit as the time step size vanishes is an open question for future work.

The next example is a special case, in which invariance and existence hold simultaneously.  
\begin{example}[Trivial boundary conditions] \label{ex:trivial-BC}
 Let $g = 0$. Then the following statements are true:
\begin{enumerate}
    \item (characterising the invariant measure) the $\mathcal{P}_{\tau,\disc}$-invariant measure is the Dirac-measure~$\delta_0$ supported at zero velocity; 
    \item (convergence of velocity at integer times) $\lim_{n\to \infty}  \Pi_\disc v^{n}_\disc = 0$;
    \item (convergence of velocity at shifted integer times) $\lim_{n\to \infty}  \Pi_\disc v^{n+1/2}_\disc = 0$;
    \item (convergence of 1st sequence of measures) $\lim_{N\to \infty}\mu_{\tau,\disc}^{N}( v^\mathrm{in}; \cdot) = \delta_{0}$;
    \item (convergence of 2nd sequence of measures) $\lim_{N\to \infty}\mu_{\tau,\disc}^{1/2,N}( v^\mathrm{in}; \cdot) = \delta_{0}$.
\end{enumerate}
\end{example}

\begin{figure}
    \centering
    \includegraphics[width=0.7\linewidth]{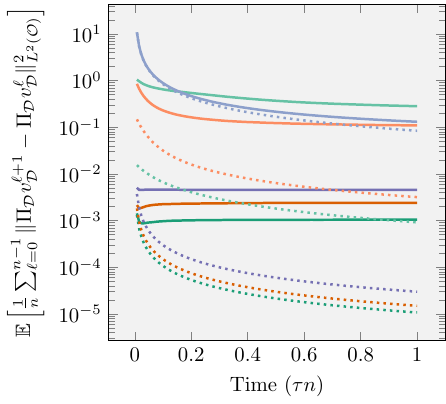}
    \caption{Time evolution of mean (based on 100 trajectories) time-averaged difference of velocity for EXP-1: $p=1.5$~({\protect\tikz \protect\draw[color=intro_color1, line width=2] (0,0) -- (0.5,0);}), $p=2$~({\protect\tikz \protect\draw[color=intro_color2, line width=2] (0,0) -- (0.5,0);}), and $p=3$~({\protect\tikz \protect\draw[color=intro_color3, line width=2] (0,0) -- (0.5,0);}); and EXP-2: $p=1.5$~({\protect\tikz \protect\draw[color=orange_light, line width=2] (0,0) -- (0.5,0);}), $p=2$~({\protect\tikz \protect\draw[color=green_light, line width=2] (0,0) -- (0.5,0);}), and $p=3$~({\protect\tikz \protect\draw[color=purple_light, line width=2] (0,0) -- (0.5,0);}). Thick lines and dotted lines denote the stochastic and deterministic evolutions, respectively. For further details on the numerical simulations, see Section~\ref{sec:num-sim}. }
    \label{fig:incEvo}
\end{figure}

\begin{table}
    \centering
    \begin{subtable}{0.5\textwidth}
    \centering
        \includegraphics[width=0.8\linewidth]{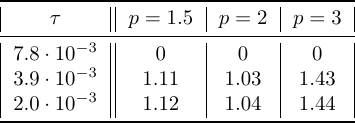}
        \caption{EOC of~$\mathfrak{C}_\disc$ for EXP-1.}
        \label{tab:eoc-var}
    \end{subtable}%
    \begin{subtable}{0.5\textwidth}
    \centering
        \includegraphics[width=0.8\linewidth]{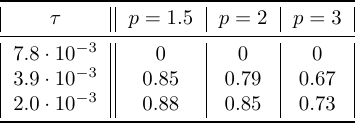}
        \caption{EOC of~$\mathfrak{C}_\disc$ for EXP-2.}
        \label{tab:eoc-lid}
    \end{subtable}%
    \caption{Experimental order of convergence (EOC) of~$\mathfrak{C}_\disc$; i.e., $ \mathfrak{C}_\disc(\tau) \eqsim \tau^{\mathrm{EOC}}$, defined in~\eqref{eq:lim-difference-tau} for EXP-1 and EXP-2, and varying viscous growth rates. For further details on the numerical simulations, see Section~\ref{sec:num-sim}.}
    \label{tab:eoc}
\end{table}

\section{Proofs} \label{sec:proofs}
In this section, we verify the claims presented in Sections~\ref{sec:disc-tools} and~\ref{sec:long-term}. We start by deriving the long-term stability of our proposed algorithm. Afterwards, we investigate the continuity of the solution operator generated by the algorithm, as well as the uniqueness of the discrete invariant measure. The continuity is derived in two steps, corresponding to the two time scales of the Crank--Nicolson algorithm. Once the continuity is established, we use it to verify the semigroup property as well as the asymptotic invariance of the first sequence of measures. Building on the long-term stability of the algorithm, we show that the second sequence of measures has an accumulation point. Lastly, we verify the sufficiency of Condition~\eqref{ass:accum-dif} to merge the desired properties: invariance and existence, and the consequences of trivial boundary conditions. 

\subsection{Stability of algorithm}  \label{sec:stability-of-algo}
Before we present the proof of Theorem~\ref{thm:apriori-vel}, we recall some classical results on the relation between the tensors~$S$ and~$V$. Details can be found, for example, in~\cite{DieE08,MR4091593}.

\begin{lemma}[Relation of tensors] \label{lem:relation-of-tensors}
The following equivalence is true, uniformly for $A, B \in \mathbb{R}^{n\times n}$:
\begin{align}\label{eq:RelTensors}
\big( S(A) - S(B)\big): (A-B) &\eqsim \abs{V(A) - V(B)}^2.
\end{align}
\end{lemma}

\begin{lemma}[Generalised Young's inequality] 
Let $\delta > 0$. Then there exists $c_\delta \geq 1$ such that for all $A,B,C \in \mathbb{R}^{n\times n}$:
\begin{align}\label{eq:GenYoung}
\big( S(A) - S(B)\big): (C-B) \leq \delta \abs{V(A) - V(B)}^2 + c_\delta \abs{V(C) - V(B)}^2.
\end{align}
\end{lemma}

\begin{proof}[Proof of Theorem~\ref{thm:apriori-vel}]
Let $n\in \mathbb{N}_0$. The test function~$\xi \in X_{\disc,0}$ in~\eqref{eq:Evolution-01} is to our disposal. Thus, we may choose $\xi = v^{n+1/2}_\disc$, which is discretely divergence free thanks to Equation~\eqref{eq:Evolution-02}. Consequently, the pressure contribution vanishes and we find the local energy identity:
\begin{align}\label{eq:square-expansion-one-step}
\begin{aligned}
\frac{1}{2}\norm{\Pi_\disc v^{n+1}_\disc}_{L^2}^2 &-\frac{1}{2}\norm{\Pi_\disc v^n_\disc}_{L^2}^2 + \tau \left(S( \varepsilon_\disc v^{n+1/2}_\disc + \varepsilon g ), \varepsilon_\disc v^{n+1/2}_\disc  \right) \\
&= \left[ B_\disc^\sigma( v^{n+1/2}_\disc, v^{n+1/2}_\disc ) + \left( (\sigma \cdot \nabla)g, \Pi_\disc v^{n+1/2}_\disc \right) \right] \Delta_{n+1} W.
\end{aligned}
\end{align}
Eventually, we will sum up the local energy identities to a prescribed time threshold~$\tau N$, $n \leq N \in \mathbb{N}$, to obtain the global energy identity. Before we do so, we analyse each local energy contribution.

\underline{Estimating the shifted diffusion~$S$.}~Using the relation of tensors and the generalised Young's inequality (see~\eqref{eq:RelTensors}--\eqref{eq:GenYoung}), we infer (for arbitrary $\delta >0$)
\begin{align} \label{eq:shifted-S-estimate}
\begin{aligned}
&\left(S( \varepsilon_\disc v^{n+1/2}_\disc + \varepsilon g ), \varepsilon_\disc v^{n+1/2}_\disc  \right)  \\
&\quad = \left(S( \varepsilon_\disc v^{n+1/2}_\disc + \varepsilon g ), \varepsilon_\disc v^{n+1/2}_\disc + \varepsilon g  \right) - \left(S( \varepsilon_\disc v^{n+1/2}_\disc + \varepsilon g ) , \varepsilon g  \right) \\
&\quad \geq (1-\delta) \norm{V(\varepsilon_\disc v^{n+1/2}_\disc + \varepsilon g)}_{L^2(\mathcal{O})}^2 - c_\delta \norm{V(\varepsilon g)}_{L^2(\mathcal{O})}^2.
\end{aligned}
\end{align}

\underline{Estimating the noise coefficient.}~Since the bi-linear form $B_\disc^\sigma$ is anti-symmetric, it vanishes if the arguments coincide. Thus, the noise coefficient reduces to
\begin{align*}
 \left( (\sigma \cdot \nabla)g , \Pi_{\disc} v^{n+1/2}_\disc \right) =\frac{1}{2} \left( (\sigma \cdot \nabla)g , \Pi_{\disc} (v^{n+1}_\disc - v^n_\disc) \right) + \left( (\sigma \cdot \nabla)g , \Pi_{\disc} v^n_\disc \right).
\end{align*}
The expectation of the second term vanishes when multiplied by the random increment~$\Delta_{n+1}W$, by independence with $\Pi_{\disc} v^n_\disc$. Next, we expand the increment $v^{n+1}_\disc - v^n_\disc$ to obtain the It\^o-Stratonovich corrector: Using the $L^2$-projection and~\eqref{eq:Evolution-01} with $\xi = \mathcal{K}_\disc [(\sigma \cdot \nabla)g] \in E_{\disc,0}$, we arrive at
\begin{align*}
&\frac{1}{2} \left( (\sigma \cdot \nabla)g , \Pi_{\disc} (v^{n+1}_\disc - v^n_\disc) \right) = \frac{1}{2} \left( \Pi_\disc \mathcal{K}_\disc [(\sigma \cdot \nabla)g] , \Pi_{\disc} (v^{n+1}_\disc - v^n_\disc) \right)\\
&\quad = -  \frac{\tau}{2} \underbrace{\left( S( \varepsilon_\disc v^{n+1/2}_\disc+ \varepsilon g ), \varepsilon_\disc \mathcal{K}_\disc [(\sigma \cdot \nabla)g] \right)}_{\mathrm{I}}  \\
&\quad \quad + \frac{\Delta_{n+1} W}{2} \left[ \underbrace{B_\disc^{\sigma}( v^{n+1/2}_\disc, \mathcal{K}_\disc [(\sigma \cdot \nabla)g] )}_{\mathrm{II}}+ \underbrace{\left( (\sigma \cdot \nabla)g , \Pi_{\disc} \mathcal{K}_\disc [(\sigma \cdot \nabla)g] \right)}_{\mathrm{III}} \right].
\end{align*}
Without loss of generality, we assume that $\abs{\Delta_{n+1} W} > 0$. If $\Delta_{n+1}W = 0$, then the right-hand-side of Equation~\eqref{eq:square-expansion-one-step} vanishes; i.e., it doesn't need to be estimated.

Estimating~$\mathrm{I}$: Writing $ \varepsilon_\disc \mathcal{K}_\disc [(\sigma \cdot \nabla)g] =\abs{\Delta_{n+1} W}^{-1}\abs{\Delta_{n+1} W} \varepsilon_\disc \mathcal{K}_\disc [(\sigma \cdot \nabla)g] $, an application of the generalised Young inequality~\eqref{eq:GenYoung} shows (for arbitrary $\delta > 0$)
\begin{align} 
\begin{aligned}\label{eq:I-final}
\mathrm{I} &\leq \frac{1}{\abs{\Delta_{n+1}W}} \delta \norm{V(\varepsilon_\disc v^{n+1/2}_\disc +\varepsilon g) }_{L^2(\mathcal{O})}^2 \\
&\qquad + \frac{1}{\abs{\Delta_{n+1}W}} c_\delta \norm{V(\abs{\Delta_{n+1} W} \varepsilon_\disc \mathcal{K}_\disc [(\sigma \cdot \nabla)g]) }_{L^2(\mathcal{O})}^2 .
\end{aligned}
\end{align}

Estimating~$\mathrm{II}$: Using~\eqref{def:discrete-noise-coefficient} and  the H\"older inequality, we obtain
\begin{align*} 
\mathrm{II} &= \frac{1}{2} \left( \sigma^T\nabla_\disc  v^{n+1/2}_\disc, \Pi_\disc \mathcal{K}_\disc [(\sigma \cdot \nabla)g] \right) - \frac{1}{2}  \left(  \Pi_\disc  v^{n+1/2}_\disc, \sigma^T\nabla_\disc \mathcal{K}_\disc [(\sigma \cdot \nabla)g]  \right) \\
&\leq \frac{\norm{\sigma}_{L^\infty(\mathcal{O})}}{2}  \norm{\nabla_\disc v^{n+1/2}_\disc}_{L^p(\mathcal{O})} \norm{\Pi_\disc \mathcal{K}_\disc[(\sigma \cdot \nabla)g]}_{L^{p'}(\mathcal{O})} \\
&\qquad + \frac{\norm{\sigma}_{L^\infty(\mathcal{O})}}{2}  \norm{\Pi_\disc  v^{n+1/2}_\disc}_{L^p(\mathcal{O})} \norm{ \nabla_\disc \mathcal{K}_\disc [(\sigma \cdot \nabla)g]}_{L^{p'}(\mathcal{O})}.
\end{align*} 
Applying~\eqref{def:discrete-poincare} and~\eqref{eq:l2-proj-sobolev}, together with the $\Gamma$-stability of the GD, and a weighted Young inequality, we derive 
\begin{align} \label{eq:II-final}
\begin{aligned}
\mathrm{II}  &\leq \norm{\sigma}_{L^\infty(\mathcal{O})} C_\disc(p) C_\disc(p') C^{\mathrm{Sob}}_\disc(p')  \norm{ (\sigma \cdot \nabla)g}_{W^{1,p'}(\mathcal{O})} \norm{\varepsilon_\disc v^{n+1/2}_\disc}_{L^p(\mathcal{O})}\\
&\leq \frac{\tau}{\abs{\Delta_{n+1} W}^2} \left( \delta  \norm{\varepsilon_\disc v^{n+1/2}_\disc}_{L^p(\mathcal{O})}^p +  c_\delta \left(\frac{\abs{\Delta_{n+1} W}^2}{\tau} \right)^{p'}  \mathrm{K}^{p'} \right),
\end{aligned}
\end{align}
where $\mathrm{K} :=  \Gamma^3 \norm{\sigma}_{L^\infty(\mathcal{O})} \norm{ (\sigma \cdot \nabla)g}_{W^{1,p'}(\mathcal{O})}$.

Estimating~$\mathrm{III}$:~By definition of the $L^2$-projection, 
\begin{align} \label{eq:III-final}
\mathrm{III} = \norm{ \Pi_{\disc} \mathcal{K}_\disc [(\sigma \cdot \nabla)g]}_{L^2(\mathcal{O})}^2 \leq \norm{ (\sigma \cdot \nabla)g }_{L^2(\mathcal{O})}^2.
\end{align}

We are ready to estimate the right-hand-side of Equation~\eqref{eq:square-expansion-one-step}. Utilising~\eqref{eq:I-final},~\eqref{eq:II-final} and~\eqref{eq:III-final}, it holds:
\begin{align}\label{eq:B-all-together}
\begin{aligned}
&\left[ B_\disc^\sigma( v^{n+1/2}_\disc, v^{n+1/2}_\disc ) + \left( (\sigma \cdot \nabla)g, \Pi_\disc v^{n+1/2}_\disc \right) \right] \Delta_{n+1} W \\
&\quad = \left[ \frac{1}{2} \left( (\sigma \cdot \nabla)g , \Pi_{\disc} (v^{n+1}_\disc - v^n_\disc) \right) + \left( (\sigma \cdot \nabla)g , \Pi_{\disc} v^n_\disc \right) \right] \Delta_{n+1} W \\
&\quad\leq  \tau \delta \left(\norm{V(\varepsilon_\disc v^{n+1/2}_\disc +\varepsilon g) }_{L^2(\mathcal{O})}^2 +  \norm{\varepsilon_\disc v^{n+1/2}_\disc}_{L^p(\mathcal{O})}^p \right) \\
&\quad\quad + \tau c_\delta  \left( \norm{V(\abs{\Delta_{n+1} W} \varepsilon_\disc \mathcal{K}_\disc [(\sigma \cdot \nabla)g]) }_{L^2(\mathcal{O})}^2 +\left(\frac{\abs{\Delta_{n+1} W}^2}{\tau} \right)^{p'} \mathrm{K}^{p'} \right) \\
&\quad\quad + \abs{\Delta_{n+1} W}^2  \norm{ (\sigma \cdot \nabla)g}_{L^2(\mathcal{O})}^2 + \left( (\sigma \cdot \nabla)g , \Pi_{\disc} v^n_\disc\right) \Delta_{n+1} W.
\end{aligned}
\end{align}
Notice that $c_\delta$ in~\eqref{eq:I-final} and~\eqref{eq:II-final} don't coincide necessarily. However, we can choose a common $c_\delta$ by taking their maximum.

\underline{Connecting~$V$ and $L^p$-norm.}
Due to the structure of $V$, there exists $\lambda > 1$ such that for all $A \in \big(\mathbb{L}^p(\mathcal{O})\big)^n$:
\begin{align*}
\lambda^{-1} \left( \norm{A}_{L^p(\mathcal{O})}^p -  \kappa^{p/2} \abs{\mathcal{O}} \right) \leq \norm{V(A)}_{L^2(\mathcal{O})}^2 \leq \lambda \left( \norm{A}_{L^p(\mathcal{O})}^p +  \kappa^{p/2} \abs{\mathcal{O}} \right).
\end{align*}
Since $\abs{a+b}^p \geq 2^{1-p} \abs{a}^p -  \abs{b}^p$, this allows us to conclude,
\begin{align} \nonumber
 \norm{V(\varepsilon_\disc v^{n+1/2}_{\disc} +\varepsilon g) }_{L^2(\mathcal{O})}^2 &\geq \lambda^{-1} \left( \norm{\varepsilon_\disc v^{n+1/2}_{\disc} + \varepsilon g}_{L^p(\mathcal{O})}^p -  \kappa^{p/2} \abs{\mathcal{O}} \right)  \\ \label{eq:V1-temp}
 &\geq \lambda^{-1} \left( 2^{1-p} \norm{\varepsilon_\disc v^{n+1/2}_{\disc}}_{L^p(\mathcal{O})}^p - \norm{\varepsilon g}_{L^p(\mathcal{O})}^p  -  \kappa^{p/2} \abs{\mathcal{O}} \right).
\end{align}
In a similar way, additionally using~\eqref{eq:l2-proj-sobolev} and the $\Gamma$-stability of the GD, we conclude that
\begin{align} \label{eq:V2-temp}
\begin{aligned}
&\norm{V(\abs{\Delta_{n+1} W} \varepsilon_\disc \mathcal{K}_\disc [(\sigma \cdot \nabla)g]) }_{L^2(\mathcal{O})}^2  \\
&\hspace{3em} \leq \lambda \left( \abs{\Delta_{n+1} W}^p \Gamma^p \norm{(\sigma \cdot \nabla)g}_{W^{1,p}(\mathcal{O})}^p +  \kappa^{p/2} \abs{\mathcal{O}} \right) 
\end{aligned}
\end{align}
and
\begin{align} \label{eq:V3-temp}
\norm{V(\varepsilon g)}_{L^2(\mathcal{O})}^2 \leq  \lambda \left( \norm{\varepsilon g}_{L^p(\mathcal{O})}^p +  \kappa^{p/2} \abs{\mathcal{O}} \right).
\end{align}

\underline{Time-local a priori estimate.}
Applying Inequalities~\eqref{eq:shifted-S-estimate} and~\eqref{eq:B-all-together} in Identity~\eqref{eq:square-expansion-one-step} and slightly reorganising yield
\begin{align*}
&\frac12\norm{\Pi_\disc v^{n+1}_\disc}_{L^2(\mathcal{O})}^2 -\frac12 \norm{\Pi_\disc v^n_\disc}_{L^2(\mathcal{O})}^2 \\
&\quad \quad +\tau (1-2\delta) \norm{V(\varepsilon_\disc v^{n+1/2}_\disc +\varepsilon g) }_{L^2(\mathcal{O})}^2 - \tau \delta  \norm{\varepsilon_\disc v^{n+1/2}_\disc}_{L^p(\mathcal{O})}^p  \\
&\leq \tau c_\delta \left( \norm{V(\abs{\Delta_{n+1} W}  \varepsilon_\disc \mathcal{K}_\disc [(\sigma \cdot \nabla)g]) }_{L^2(\mathcal{O})}^2 +  \norm{V(\varepsilon g)}_{L^2(\mathcal{O})}^2 + \left(\frac{\abs{\Delta_{n+1} W}^2}{\tau} \right)^{p'} \mathrm{K}^{p'}  \right)  \\
&\quad\quad +\abs{\Delta_{n+1} W}^2  \norm{(\sigma \cdot \nabla)g}_{L^2(\mathcal{O})}^2 + \left( (\sigma \cdot \nabla)g , \Pi_{\disc} v^n_\disc \right) \Delta_{n+1} W.
\end{align*}
Next, we use the relation of $V$ and $L^p$-norms (see Inequalities~\eqref{eq:V1-temp},~\eqref{eq:V2-temp} and~\eqref{eq:V3-temp}) to obtain
\begin{align} \label{eq:first-apriori-bound-local}
\begin{aligned}
\norm{\Pi_\disc v^{n+1}_\disc}_{L^2(\mathcal{O})}^2 &-\norm{\Pi_\disc v^n_\disc}_{L^2(\mathcal{O})}^2 +\tau (1-2\delta) \left( \frac{ 2^{1-p}}{\lambda}  -\frac{\delta}{1-2\delta}  \right) \norm{\varepsilon_\disc v^{n+1/2}_\disc}_{L^p(\mathcal{O})}^p \\
& \leq \tau  C^{\mathrm{tot}}_{\delta}(n)  + \left( (\sigma \cdot \nabla)g , \Pi_{\disc} v^n_\disc \right) \Delta_{n+1} W.
\end{aligned}
\end{align}
where
\begin{align*}
    C^{\mathrm{tot}}_{\delta}(n) &= \left( \lambda^{-1} (1-2\delta)  + 2\lambda c_\delta \right) \left( \norm{\varepsilon g}_{L^p(\mathcal{O})}^p  + \kappa^{p/2} \abs{\mathcal{O}} \right) \\
    &\quad + c_\delta \left(\frac{\abs{\Delta_{n+1} W}^2}{\tau} \right)^{p'} \left( \Gamma^3 \norm{\sigma}_{L^\infty(\mathcal{O})} \norm{ (\sigma \cdot \nabla)g}_{W^{1,p'}(\mathcal{O})} \right)^{p'} \\
    &\quad + c_\delta \lambda \tau^{p/2} \left(  \frac{\abs{\Delta_{n+1} W}^2}{\tau} \right)^{p/2} \Gamma^p \norm{(\sigma \cdot \nabla)g}_{W^{1,p}(\mathcal{O})}^p  \\
    &\quad + \frac{\abs{\Delta_{n+1} W}^2}{\tau}  \norm{ (\sigma \cdot \nabla)g}_{L^2(\mathcal{O})}^2.
\end{align*}
At this point, we derived a time-local energy inequality suitable for time-global estimates. But first, we need to fix the value of~$\delta >0$. Notice that $f: \delta \mapsto (1-2\delta) \left( \frac{ 2^{1-p}}{\lambda}  -\frac{\delta}{1-2\delta}  \right)$ is continuous in $[0,1/2)$ and satisfies: $f(0) = \frac{ 2^{1-p}}{\lambda} > 0$. Thus, there exists $\delta^* \in (0,1/2)$ such that $f(\delta^*) > 0$. We now fix $\delta = \delta^*$.

\underline{Time-global a priori estimate.} Let $M\leq N \in \mathbb{N}$ and $q\in (0,\infty)$. First, we sum up Inequality~\eqref{eq:first-apriori-bound-local} for $n \leq M-1$. Afterwards, we take the $q$-th power, apply the maximum for $M \leq N$ and expectations to find:
\begin{align} \label{eq:almost-done-apriori}
\begin{aligned}
&\mathbb{E}\left[ \max_{M \leq N} \left( \norm{\Pi_\disc v^{M}_\disc}_{L^2(\mathcal{O})}^{2} + \tau f(\delta^*) \sum_{n=0}^{M-1} \norm{\varepsilon_\disc v^{n+1/2}_\disc}_{L^p(\mathcal{O})}^p \right)^q \right] \\
&\qquad \qquad \lesssim_q  \norm{\Pi_\disc v^0}_{L^2(\mathcal{O})}^{2q}  + \mathrm{R}_2 + \mathrm{R}_3,
\end{aligned}
\end{align}
where
\begin{align*}
   \mathrm{R}_2 &:=   \mathbb{E}\left[ \left( \sum_{n=0}^{N-1}\tau C^{\mathrm{tot}}_{\delta^*}(n)  \right)^q \right], \\
   \mathrm{R}_3 &:= \mathbb{E}\left[ \left( \max_{M \leq N} \sum_{n=0}^{M-1} \left( (\sigma \cdot \nabla)g , \Pi_{\disc} v^n_\disc \right) \Delta_{n+1} W \right)^q  \right].
\end{align*}
Due to the stability of the discrete Helmholtz projection, it holds:
\begin{align*}
    \norm{\Pi_\disc v^0}_{L^2(\mathcal{O})} \leq \norm{v^\mathrm{in}}_{L^2(\mathcal{O})}.
\end{align*}
The proof is complete, once we have estimated the Terms~$\mathrm{R}_2$ and $\mathrm{R}_3$. 

Estimating~$\mathrm{R}_2$:~Notice that, for any $\alpha \in (0,\infty)$,
\begin{align} \label{eq:estimate-corrector}
\begin{aligned}
&\mathbb{E} \left[ \left( C^{\mathrm{tot}}_{\delta^*}(n) \right)^\alpha  \right] \lesssim 1,
\end{aligned}
\end{align}
since the standard normal distribution has moments of arbitrary high order. Thus,~$\mathrm{R}_2$ has the right time-scaling, provided we can interchange expectation and summation. First, let us assume that $q \in (0,1)$. Using H\"older's inequality and Inequality~\eqref{eq:estimate-corrector} with $\alpha = 1$,
\begin{align*}
\mathrm{R}_2 \leq \left( \mathbb{E}\left[ \sum_{n=0}^{N-1}\tau C^{\mathrm{tot}}_{\delta^*}(n) \right] \right)^q \lesssim  \left(\tau N \right)^q.
\end{align*} 
Next, let $q \in [1,\infty)$. Similarly, using H\"older's inequality and Inequality~\eqref{eq:estimate-corrector} with $\alpha = q$,
\begin{align*}
\mathrm{R}_2 \leq \sum_{n=0}^{N-1}\tau \mathbb{E} \left[ \big(C^{\mathrm{tot}}_{\delta^*}(n)\big)^q \right] (N \tau)^{q-1} \lesssim \left(\tau N \right)^q.
\end{align*}
This concludes the estimate of~$\mathrm{R}_2$.

Estimating~$\mathrm{R}_3$:~ Notice that
\begin{align*}
\mathbb{N}_0 \ni M \mapsto \sum_{n=0}^{M-1} \left( (\sigma \cdot \nabla)g , \Pi_{\disc} v^n_\disc \right) \Delta_{n+1} W
\end{align*}
is a centred time-discrete $(\mathcal{F}_{n\tau})_{n \in \mathbb{N}_0}$-martingale. Invoking the Burkholder-Davis-Gundy inequality, the H\"older inequality and the Young inequality with weight~$\eta>0$,
\begin{align*}
\mathrm{R}_3 &\lesssim \mathbb{E}\left[ \left( \sum_{n=0}^{N-1} \abs{ \left( (\sigma \cdot \nabla)g , \Pi_{\disc} v^n_\disc \right) }^2 \tau \right)^{q/2}  \right] \\
&\leq \mathbb{E}\left[ \left( \max_{M \leq N} \norm{ \Pi_{\disc} v^M_\disc}_{L^2(\mathcal{O})}^2 \norm{(\sigma \cdot \nabla)g}_{L^2(\mathcal{O})}^2  N\tau \right)^{q/2}  \right] \\
&\leq \eta \mathbb{E}\left[  \max_{M \leq N} \norm{ \Pi_{\disc} v^M_\disc}_{L^2(\mathcal{O})}^{2q}  \right] + c_\eta  \left( \norm{(\sigma \cdot \nabla)g}_{L^2(\mathcal{O})}^{2}  N\tau \right)^{q}.
\end{align*}
Since $\eta>0$ is to our disposal, we can absorb the first term into the left-hand-side of Inequality~\eqref{eq:almost-done-apriori}. 

In total, we have verified Inequality~\eqref{eq:apriori-vel} and the proof is complete. 
\end{proof}

\subsection{Long-term dynamics of algorithm}

\subsubsection{Continuity of time-stepping and semigroup property}
Before we start the verification of the semigroup property for the discrete semigroup, we derive preparatory results that address continuity with respect to the input velocity and random update of a single time-step of the time-stepping scheme. First, we derive a result for the half-step solution operator, followed by one for the full-step solution operator. We close the section with the proof of the semigroup property.  

For $q \in (0,\infty)$, we denote the set of reconstructed discretely divergence free velocity with finite $q$-th moment measured on $\mathbb{L}^2(\mathcal{O})$ by
\begin{align*}
    L^q\big(\Omega; (E_{\disc,0}, \norm{\Pi_\disc \cdot}_{L^2(\mathcal{O})}) \big) := \left\{ u: \Omega \to E_{\disc,0}\Big|~ \mathbb{E}\left[ \norm{\Pi_\disc u}_{L^2(\mathcal{O})}^q \right] < \infty \right\}.
\end{align*}

\underline{Continuity of the half-step solution operator.}
Let $\mathscr{S}^{1/2}_{\mathrm{imp}}$ be the implicit half-step solution operator that maps input velocity~$v^\mathrm{in}_\disc \in E_{\disc,0}$ and random update~$\Delta W$ to the unique solution~$\mathscr{S}^{1/2}_{\mathrm{imp}}[v^\mathrm{in}_\disc,\Delta W] \in E_{\disc,0}$ which satisfies, for all $\xi \in E_{\disc,0}$ and $\mathbb{P}$-a.s.:
\begin{align} 
\nonumber 
 &\left( \Pi_\disc \mathscr{S}^{1/2}_{\mathrm{imp}}[v^\mathrm{in}_\disc,\Delta W] -\Pi_\disc v^\mathrm{in}_\disc, \Pi_\disc \xi \right) + \frac{\tau}{2} \left( S\left( \varepsilon_\disc\mathscr{S}^{1/2}_{\mathrm{imp}}[v^\mathrm{in}_\disc,\Delta W] + \varepsilon g \right), \varepsilon_\disc \xi \right) \\  \label{eq:half-step-solution}
    &\hspace{2em} = \left[ B_\disc^{\sigma}\left( \mathscr{S}^{1/2}_{\mathrm{imp}}[v^\mathrm{in}_\disc,\Delta W], \xi \right)+ \left( (\sigma \cdot \nabla)g , \Pi_{\disc} \xi \right) \right] \frac{\Delta W}{2}.
\end{align}
Its stability is presented in the next lemma.
\begin{lemma}
For any $q \in (0,\infty)$,
\begin{align*}
   \mathscr{S}^{1/2}_{\mathrm{imp}}:~L^q\big(\Omega; (E_{\disc,0}, \norm{\Pi_\disc \cdot}_{L^2(\mathcal{O})}) \big) \times L^q(\Omega) \to  L^q\big(\Omega; (E_{\disc,0}, \norm{\Pi_\disc \cdot}_{L^2(\mathcal{O})}) \big).
\end{align*}
\end{lemma}
\begin{proof}
The existence of~$\mathscr{S}^{1/2}_{\mathrm{imp}}[v^\mathrm{in}_\disc,\Delta W]$ follows from monotone operator theory. Moreover, the following a priori estimate is satisfied $\mathbb{P}$-a.s.: 
\begin{align} 
\begin{aligned} \label{eq:estimate-data-half}
    &\norm{\Pi_{\disc}\mathscr{S}^{1/2}_{\mathrm{imp}}[v^\mathrm{in}_\disc,\Delta W]  }_{L^2(\mathcal{O})}^2 + \tau \norm{\varepsilon_\disc \mathscr{S}^{1/2}_{\mathrm{imp}}[v^\mathrm{in}_\disc,\Delta W] }_{L^p(\mathcal{O})}^p \\
 &\qquad \lesssim \norm{\Pi_{\disc}v^\mathrm{in}_\disc }_{L^2(\mathcal{O})}^2 + \tau \big( \norm{\varepsilon g }_{L^p(\mathcal{O}) }^p + 1\big) +  \norm{(\sigma \cdot \nabla) g}_{L^2(\mathcal{O})}^2 \abs{\Delta W}^2.
 \end{aligned}
\end{align}
Inequality~\eqref{eq:estimate-data-half} implies the claimed stability.
\end{proof}
Next, we investigate the sensitivity of the solution operator with respect to the input velocity and random update. First, we present the discussion on the input velocity, followed by that on the random update. 

\begin{lemma}
For any $q \in (0,\infty)$ and fixed random update~$\Delta W \in L^q(\Omega)$, the map 
\begin{align*}
   \mathscr{S}^{1/2}_{\mathrm{imp}}[\cdot,\Delta W]: ~L^q\big(\Omega; (E_{\disc,0}, \norm{\Pi_\disc \cdot}_{L^2(\mathcal{O})}) \big) \to  L^q\big(\Omega; (E_{\disc,0}, \norm{\Pi_\disc \cdot}_{L^2(\mathcal{O})}) \big)
\end{align*}
is Lipschitz-continuous with unit Lipschitz-constant.
\end{lemma}
\begin{proof}
    Let $v^a_\disc,\, v^b_\disc \in L^q\big(\Omega; (E_{\disc,0}, \norm{\Pi_\disc \cdot}_{L^2(\mathcal{O})}) \big)$ and $\Delta W \in L^q(\Omega)$. We denote differences of input and output velocity by
    \begin{align*}
        \delta v^\mathrm{in}_\disc := v^a_\disc - v^b_\disc \quad \text{ and } \quad \delta v^\mathrm{out}_\disc:= \mathscr{S}^{1/2}_{\mathrm{imp}}[v^a_\disc,\Delta W] - \mathscr{S}^{1/2}_{\mathrm{imp}}[v^b_\disc,\Delta W],
    \end{align*}
     respectively. Subtracting Identity~\eqref{eq:half-step-solution} for~$v^\mathrm{in}_\disc \in \{ v^a_\disc, v^b_\disc \}$ yields, for all $\xi \in E_{\disc,0}$ and $\mathbb{P}$-a.s.:
\begin{align*}
    &\left( \Pi_\disc \delta v^\mathrm{out}_\disc -\Pi_\disc \delta v^\mathrm{in}_\disc, \Pi_\disc \xi \right) \\
    &\qquad \qquad + \frac{\tau}{2} \left( S\left( \varepsilon_\disc\mathscr{S}^{1/2}_{\mathrm{imp}}[v^a_\disc,\Delta W] + \varepsilon g \right) - S\left( \varepsilon_\disc\mathscr{S}^{1/2}_{\mathrm{imp}}[v^b_\disc,\Delta W] + \varepsilon g \right), \varepsilon_\disc \xi \right) \\ 
    &\qquad = B_\disc^{\sigma}\left(  \delta v^\mathrm{out}_\disc, \xi \right)\frac{\Delta W}{2}.
\end{align*}
Choosing $\xi = \delta v^\mathrm{out}_\disc \in E_{\disc,0}$, in combination with the monotonicity of $S$, the fact that~$B^\sigma_\disc$ vanishes on its diagonal, and the standard vector identity: $2 a \cdot (a-b) = \abs{a}^2 - \abs{b}^2 + \abs{a-b}^2$; we conclude that~$\mathbb{P}$-a.s.:
\begin{align*}
    \norm{\Pi_\disc \delta v^\mathrm{out}_\disc}_{L^2(\mathcal{O})}^2 + \norm{\Pi_\disc \delta v^\mathrm{out}_\disc - \Pi_\disc \delta v^\mathrm{in}_\disc}_{L^2(\mathcal{O})}^2 \leq \norm{\Pi_\disc \delta v^\mathrm{in}_\disc}_{L^2(\mathcal{O})}^2.
\end{align*}
Thus, for fixed random update~$\Delta W \in L^q(\Omega)$, the map 
\begin{align*}
   \mathscr{S}^{1/2}_{\mathrm{imp}}[\cdot,\Delta W]: ~L^q\big(\Omega; (E_{\disc,0}, \norm{\Pi_\disc \cdot}_{L^2(\mathcal{O})}) \big) \to  L^q\big(\Omega; (E_{\disc,0}, \norm{\Pi_\disc \cdot}_{L^2(\mathcal{O})}) \big)
\end{align*}
is Lipschitz-continuous with unit Lipschitz-constant.
\end{proof}

\begin{lemma}
 Let $q \in (0,\infty)$ and $\overline{q}=\max\{ 4/p, 1\} q$. Moreover, let the input velocity $v^\mathrm{in}_\disc \in L^{\overline{q}}\big(\Omega; (E_{\disc,0}, \norm{\Pi_\disc \cdot}_{L^2(\mathcal{O})}) \big)$ be fixed. Then, the map
\begin{align*}
    \mathscr{S}^{1/2}_{\mathrm{imp}}[v^\mathrm{in}_\disc,\cdot]:~L^{\overline{q}}(\Omega) \to L^{q}\big(\Omega; (E_{\disc,0}, \norm{\Pi_\disc \cdot}_{L^2(\mathcal{O})}) \big)
\end{align*}
is continuous. 
\end{lemma}
\begin{proof}
Let $v^\mathrm{in}_\disc \in L^q\big(\Omega; (E_{\disc,0}, \norm{\Pi_\disc \cdot}_{L^2(\mathcal{O})}) \big)$ and $\Delta^a W,\, \Delta^b W \in L^q(\Omega)$. We denote differences of random updates and output velocity by 
\begin{align*}
    \delta W:= \Delta^a W - \Delta^b W \quad \text{ and } \quad \delta w^\mathrm{out}_\disc:= \mathscr{S}^{1/2}_{\mathrm{imp}}[v^\mathrm{in}_\disc,\Delta^a W] - \mathscr{S}^{1/2}_{\mathrm{imp}}[v^\mathrm{in}_\disc,\Delta^b W],
\end{align*}
 respectively. Subtracting Identity~\eqref{eq:half-step-solution} for~$\Delta W \in \{ \Delta^a W, \Delta^b W \}$ yields, for all $\xi \in E_{\disc,0}$ and $\mathbb{P}$-a.s.:
\begin{align*}
    &\left( \Pi_\disc \delta w^\mathrm{out}_\disc, \Pi_\disc \xi \right) \\
    &\qquad \qquad + \frac{\tau}{2} \left( S\left( \varepsilon_\disc\mathscr{S}^{1/2}_{\mathrm{imp}}[v^\mathrm{in}_\disc,\Delta^a W] + \varepsilon g \right) - S\left( \varepsilon_\disc \mathscr{S}^{1/2}_{\mathrm{imp}}[v^\mathrm{in}_\disc,\Delta^b W] + \varepsilon g \right), \varepsilon_\disc \xi \right) \\ 
    &\qquad = \left[ B_\disc^{\sigma}\left( \mathscr{S}^{1/2}_{\mathrm{imp}}[v^\mathrm{in}_\disc,\Delta^a W], \xi \right)+ \left( (\sigma \cdot \nabla)g , \Pi_{\disc} \xi \right) \right] \frac{\Delta^a W}{2} \\
    &\qquad \qquad - \left[ B_\disc^{\sigma}\left( \mathscr{S}^{1/2}_{\mathrm{imp}}[v^\mathrm{in}_\disc,\Delta^b W], \xi \right)+ \left( (\sigma \cdot \nabla)g , \Pi_{\disc} \xi \right) \right] \frac{\Delta^b W}{2}.
\end{align*}
Similar as before, we choose~$\xi = \delta w^\mathrm{out}_\disc \in E_{\disc,0}$ in above identity. Applying the monotonicity of $S$, slightly rewriting the right-hand-side and using the fact that~$B^\sigma_\disc$ vanishes on its diagonal, we find that~$\mathbb{P}$-a.s.
\begin{align*}
    &\norm{\Pi_\disc \delta w^\mathrm{out}_\disc}_{L^2(\mathcal{O})}^2 \leq \left[ B_\disc^{\sigma}\left( \mathscr{S}^{1/2}_{\mathrm{imp}}[v^\mathrm{in}_\disc,\Delta^b W],  \delta w^\mathrm{out}_\disc \right)+ \left( (\sigma \cdot \nabla)g , \Pi_{\disc}  \delta w^\mathrm{out}_\disc \right) \right] \frac{\delta W}{2}.
\end{align*}
It remains to estimate the coefficient of~$\delta W$. The coefficient consists of two terms: a linear and a non-linear one. The linear term can be handled by the Cauchy--Schwarz and Young inequalities:
\begin{align*}
    \left( (\sigma \cdot \nabla)g , \Pi_{\disc}  \delta w^\mathrm{out}_\disc \right)\frac{\delta W}{2} \leq \frac{1}{8} \abs{\delta W}^2 \norm{(\sigma \cdot \nabla)g}_{L^2(\mathcal{O})}^2 + \frac{1}{2} \norm{\Pi_{\disc}  \delta w^\mathrm{out}_\disc}_{L^2(\mathcal{O})}^2.
\end{align*}
The non-linear term is estimated as follows. Using that~$B^\sigma_\disc$ vanishes on its diagonal, the H\"older inequality and the dominance of the reconstructed symmetric gradient (see Equation~\eqref{def:discrete-poincare}) we verify that
\begin{align*}
    &B_\disc^{\sigma}\left( \mathscr{S}^{1/2}_{\mathrm{imp}}[v^\mathrm{in}_\disc,\Delta^b W],  \delta w^\mathrm{out}_\disc \right) = B_\disc^{\sigma}\left( \mathscr{S}^{1/2}_{\mathrm{imp}}[v^\mathrm{in}_\disc,\Delta^b W],  \mathscr{S}^{1/2}_{\mathrm{imp}}[v^\mathrm{in}_\disc,\Delta^a W] \right) \\
    &\qquad \leq \frac{1}{2} \norm{\sigma}_{L^\infty(\mathcal{O})} \norm{\nabla_\disc \mathscr{S}^{1/2}_{\mathrm{imp}}[v^\mathrm{in}_\disc,\Delta^b W] }_{L^2(\mathcal{O})} \norm{\Pi_\disc \mathscr{S}^{1/2}_{\mathrm{imp}}[v^\mathrm{in}_\disc,\Delta^a W] }_{L^2(\mathcal{O})} \\
    &\qquad \qquad  +  \frac{1}{2} \norm{\sigma}_{L^\infty(\mathcal{O})} \norm{\nabla_\disc \mathscr{S}^{1/2}_{\mathrm{imp}}[v^\mathrm{in}_\disc,\Delta^a W] }_{L^2(\mathcal{O})} \norm{\Pi_\disc \mathscr{S}^{1/2}_{\mathrm{imp}}[v^\mathrm{in}_\disc,\Delta^b W] }_{L^2(\mathcal{O})}\\
    &\qquad \leq \norm{\sigma}_{L^\infty(\mathcal{O})} \big( C_\disc(2) \big)^2 \norm{\varepsilon_\disc \mathscr{S}^{1/2}_{\mathrm{imp}}[v^\mathrm{in}_\disc,\Delta^b W] }_{L^2(\mathcal{O})} \norm{\varepsilon_\disc \mathscr{S}^{1/2}_{\mathrm{imp}}[v^\mathrm{in}_\disc,\Delta^a W] }_{L^2(\mathcal{O})}.
\end{align*}
Combining the estimates for the linear and the non-linear terms, we arrive at 
\begin{align*}
    &\norm{\Pi_\disc \delta w^\mathrm{out}_\disc}_{L^2(\mathcal{O})}^2 \\
    &\qquad \lesssim \abs{\delta W}^2 +  \norm{\varepsilon_\disc \mathscr{S}^{1/2}_{\mathrm{imp}}[v^\mathrm{in}_\disc,\Delta^b W] }_{L^2(\mathcal{O})} \norm{\varepsilon_\disc \mathscr{S}^{1/2}_{\mathrm{imp}}[v^\mathrm{in}_\disc,\Delta^a W] }_{L^2(\mathcal{O})} \abs{\delta W}.
\end{align*}
Taking the $q/2$-th power, expectations and the Cauchy--Schwarz inequality shows
\begin{align*}
    &\mathbb{E}\left[ \norm{\Pi_\disc \delta w^\mathrm{out}_\disc}_{L^2(\mathcal{O})}^q \right] \lesssim  \mathbb{E}\left[ \abs{\delta W}^q \right] \\
    &\qquad + \left( \mathbb{E}\left[ \norm{\varepsilon_\disc \mathscr{S}^{1/2}_{\mathrm{imp}}[v^\mathrm{in}_\disc,\Delta^b W] }_{L^2(\mathcal{O})}^q \norm{\varepsilon_\disc \mathscr{S}^{1/2}_{\mathrm{imp}}[v^\mathrm{in}_\disc,\Delta^a W] }_{L^2(\mathcal{O})}^q \right] \mathbb{E}\left[ \abs{\delta W}^q \right] \right)^{1/2}.
\end{align*}
It remains to argue why 
\begin{align*}
 \mathbb{E}\left[ \norm{\varepsilon_\disc \mathscr{S}^{1/2}_{\mathrm{imp}}[v^\mathrm{in}_\disc,\Delta^b W] }_{L^2(\mathcal{O})}^q \norm{\varepsilon_\disc \mathscr{S}^{1/2}_{\mathrm{imp}}[v^\mathrm{in}_\disc,\Delta^a W] }_{L^2(\mathcal{O})}^q \right]   
\end{align*}
is finite. We distinguish two cases: $p\geq 2$ and $p \in (1,2)$.

First, let $p \geq 2$. Invoking Young's inequality and $\mathbb{L}^p(\mathcal{O}) \hookrightarrow \mathbb{L}^2(\mathcal{O})$ since $p \geq 2$, we find 
\begin{align*}
  &\mathbb{E}\left[ \norm{\varepsilon_\disc \mathscr{S}^{1/2}_{\mathrm{imp}}[v^\mathrm{in}_\disc,\Delta^b W] }_{L^2(\mathcal{O})}^q \norm{\varepsilon_\disc \mathscr{S}^{1/2}_{\mathrm{imp}}[v^\mathrm{in}_\disc,\Delta^a W] }_{L^2(\mathcal{O})}^q \right] \\
  &\qquad \lesssim \mathbb{E}\left[ \norm{\varepsilon_\disc \mathscr{S}^{1/2}_{\mathrm{imp}}[v^\mathrm{in}_\disc,\Delta^b W] }_{L^p(\mathcal{O})}^{2q} \right] + \mathbb{E}\left[ \norm{\varepsilon_\disc \mathscr{S}^{1/2}_{\mathrm{imp}}[v^\mathrm{in}_\disc,\Delta^a W] }_{L^p(\mathcal{O})}^{2q} \right].
\end{align*}
Thanks to Inequality~\eqref{eq:estimate-data-half} we derive that, for $\Delta W \in \{ \Delta^a W, \Delta^b W\}$,
\begin{align*}
  &\mathbb{E}\left[ \norm{\varepsilon_\disc \mathscr{S}^{1/2}_{\mathrm{imp}}[v^\mathrm{in}_\disc,\Delta W] }_{L^p(\mathcal{O})}^{2q} \right] = \tau^{-\frac{2}{p}q} \mathbb{E}\left[ \left( \tau \norm{\varepsilon_\disc \mathscr{S}^{1/2}_{\mathrm{imp}}[v^\mathrm{in}_\disc,\Delta W] }_{L^p(\mathcal{O})}^{p} \right)^{\frac{2}{p}q} \right] \\
  &\qquad \lesssim \tau^{-\frac{2}{p}q} \mathbb{E}\left[ \left( \norm{\Pi_{\disc}v^\mathrm{in}_\disc }_{L^2(\mathcal{O})}^2 + \tau \big( \norm{\varepsilon g }_{L^p(\mathcal{O}) }^p + 1\big) +  \norm{(\sigma \cdot \nabla) g}_{L^2(\mathcal{O})}^2 \abs{\Delta W}^2 \right)^{\frac{2}{p}q} \right],
\end{align*}
where the right-hand-side is finite if input velocity and random updates have finite $4q/p$-th moment.

Now, let $p \in (1,2)$. Using the inverse estimate \eqref{eq:inverse-first} and the Young inequality, it holds:
\begin{align*}
  &\mathbb{E}\left[ \norm{\varepsilon_\disc \mathscr{S}^{1/2}_{\mathrm{imp}}[v^\mathrm{in}_\disc,\Delta^b W] }_{L^2(\mathcal{O})}^q \norm{\varepsilon_\disc \mathscr{S}^{1/2}_{\mathrm{imp}}[v^\mathrm{in}_\disc,\Delta^a W] }_{L^2(\mathcal{O})}^q \right] \\
  &\qquad  \leq \big(\mathfrak{B}_\disc(2)\big)^2\mathbb{E}\left[ \norm{\Pi_\disc \mathscr{S}^{1/2}_{\mathrm{imp}}[v^\mathrm{in}_\disc,\Delta^b W] }_{L^2(\mathcal{O})}^{2q} +  \norm{\Pi_\disc \mathscr{S}^{1/2}_{\mathrm{imp}}[v^\mathrm{in}_\disc,\Delta^a W] }_{L^2(\mathcal{O})}^{2q} \right] .
\end{align*}
Again, applying Inequality~\eqref{eq:estimate-data-half} for $\Delta W \in \{ \Delta^a W, \Delta^b W\}$, we deduce:
\begin{align*}
  &\mathbb{E}\left[ \norm{\Pi_\disc \mathscr{S}^{1/2}_{\mathrm{imp}}[v^\mathrm{in}_\disc,\Delta W] }_{L^2(\mathcal{O})}^{2q} \right] \\
  &\qquad \lesssim \mathbb{E}\left[\left( \norm{\Pi_{\disc}v^\mathrm{in}_\disc }_{L^2(\mathcal{O})}^2 + \tau \big( \norm{\varepsilon g }_{L^p(\mathcal{O}) }^p + 1\big) +  \norm{(\sigma \cdot \nabla) g}_{L^2(\mathcal{O})}^2 \abs{\Delta W}^2 \right)^q \right],
\end{align*}
which is finite for input velocity and random update with finite $2q$-th moment. Notice that $4/p > 2$ since $p \in (1,2)$. Thus, let $\overline{q} := \max\{ 4/p, 1\} q \geq q$.

For fixed $v^\mathrm{in}_\disc \in L^{\overline{q}}\big(\Omega; (E_{\disc,0}, \norm{\Pi_\disc \cdot}_{L^2(\mathcal{O})}) \big)$, we have shown that the map
\begin{align*}
    \mathscr{S}^{1/2}_{\mathrm{imp}}[v^\mathrm{in}_\disc,\cdot]:~L^{\overline{q}}(\Omega) \to L^{q}\big(\Omega; (E_{\disc,0}, \norm{\Pi_\disc \cdot}_{L^2(\mathcal{O})}) \big)
\end{align*}
is continuous. 
\end{proof}

\underline{Continuity of the full-step solution operator.}~Let $\mathscr{S}_{\mathrm{CN}}$ be the full-step solution operator that maps input velocity~$v^\mathrm{in}_\disc \in E_{\disc,0}$ and random update~$\Delta W$ to the unique solution~$\mathscr{S}_{\mathrm{CN}}[v^\mathrm{in}_\disc,\Delta W] \in E_{\disc,0}$ which satisfies, for all $\xi \in E_{\disc,0}$ and $\mathbb{P}$-a.s.:
\begin{align*}
     &\left( \Pi_\disc \mathscr{S}_{\mathrm{CN}}[v^\mathrm{in}_\disc,\Delta W] -\Pi_\disc v^\mathrm{in}_\disc, \Pi_\disc \xi \right) \\
     &\qquad \qquad + \tau \left( S\left( \varepsilon_\disc \frac{\mathscr{S}_{\mathrm{CN}}[v^\mathrm{in}_\disc,\Delta W] + v^\mathrm{in}_\disc}{2} + \varepsilon g \right), \varepsilon_\disc \xi \right) \\ 
    &\hspace{2em} = \left[ B_\disc^{\sigma}\left(\frac{\mathscr{S}_{\mathrm{CN}}[v^\mathrm{in}_\disc,\Delta W] + v^\mathrm{in}_\disc}{2} , \xi \right)+ \left( (\sigma \cdot \nabla)g , \Pi_{\disc} \xi \right) \right] \Delta W.
\end{align*}
Due to the unique solvability of Equation~\eqref{eq:half-step-solution}, we can identify the arithmetic mean of the input velocity and output velocity, generated by the full-step solution operator, as the output of the half-step solution operator, i.e., 
\begin{align*}
  \frac{\mathscr{S}_{\mathrm{CN}}[v^\mathrm{in}_\disc,\Delta W] + v^\mathrm{in}_\disc}{2}  & = \mathscr{S}_{\mathrm{imp}}^{1/2}[v_\disc^\mathrm{in}, \Delta W]. 
\end{align*}
Thus, the full-step solution operator~$\mathscr{S}_{\mathrm{CN}}$ has the following alternative representation:
\begin{align} \label{def:full-step-operator}
    \mathscr{S}_{\mathrm{CN}}[v_\disc^\mathrm{in}, \Delta W] =  2 \mathscr{S}_{\mathrm{imp}}^{1/2}[v_\disc^\mathrm{in}, \Delta W] - v_\disc^\mathrm{in}.
\end{align}
Consequently, the full-step solution operator inherits the following properties from the half-step solution operator:
\begin{lemma}
Let $q \in (0,\infty)$ and $\overline{q}=\max\{ 4/p, 1\} q$. The following statements are true:
\begin{itemize}
    \item (boundedness) the full-step solution operator is bounded, i.e.,
    \begin{align*}
        \mathscr{S}_{\mathrm{CN}}:~L^q\big(\Omega; (E_{\disc,0}, \norm{\Pi_\disc \cdot}_{L^2(\mathcal{O})}) \big) \times L^q(\Omega) \to  L^q\big(\Omega; (E_{\disc,0}, \norm{\Pi_\disc \cdot}_{L^2(\mathcal{O})}) \big);
    \end{align*}
    \item (continuity w.r.t initial velocity) for any random update~$\Delta W \in L^q(\Omega)$, the map 
\begin{align} \label{eq:Lip-cont-wrt-vel}
   \mathscr{S}_{\mathrm{CN}}[\cdot,\Delta W]: ~L^q\big(\Omega; (E_{\disc,0}, \norm{\Pi_\disc \cdot}_{L^2(\mathcal{O})}) \big) \to  L^q\big(\Omega; (E_{\disc,0}, \norm{\Pi_\disc \cdot}_{L^2(\mathcal{O})}) \big)
\end{align}
is Lipschitz-continuous;
    \item (continuity w.r.t random update) for any $v^\mathrm{in}_\disc \in L^{\overline{q}}\big(\Omega; (E_{\disc,0}, \norm{\Pi_\disc \cdot}_{L^2(\mathcal{O})}) \big)$, the map
\begin{align} \label{eq:cont-wrt-random}
    \mathscr{S}_{\mathrm{CN}}[v^\mathrm{in}_\disc,\cdot]:~L^{\overline{q}}(\Omega) \to L^{q}\big(\Omega; (E_{\disc,0}, \norm{\Pi_\disc \cdot}_{L^2(\mathcal{O})}) \big)
\end{align}
is continuous. 
\end{itemize}
\end{lemma}

We are ready to verify Lemma~\ref{lem:semigroup}.
\begin{proof}[Proof of Lemma~\ref{lem:semigroup}]
We need to check the following two claims: 
\begin{enumerate}
\item[(a)] \label{it:semigroup-a} for all $n \in \mathbb{N}$, $f \in C_b\big(\mathbb{L}^2(\mathcal{O})\big)$ and $v^\mathrm{in} \in \mathbb{L}^2(\mathcal{O})$: 
    \begin{align*}
        (\mathcal{P}^n_{\tau,\disc} f)(v^\mathrm{in}) = \big(\mathcal{P}^{n-1}_{\tau,\disc} (\mathcal{P}_{\tau,\disc}f)\big)(v^\mathrm{in});
    \end{align*}
    \item[(b)] \label{it:semigroup-b} for all $f \in C_b\big(\mathbb{L}^2(\mathcal{O})\big)$ and $v^\mathrm{in} \in \Pi_\disc E_{\disc, 0}$: 
    \begin{align*}
         (\mathcal{P}^0_{\tau,\disc} f)(v^\mathrm{in}) =  f(v^\mathrm{in}).
    \end{align*}
\end{enumerate}

\underline{Addressing~\ref{it:semigroup-b}.}~If $v^\mathrm{in}=\Pi_\disc v_\disc$ for some $v_\disc\in E_{\disc,0}$, then it follows immediately that $(v_\disc,0)$ solves \eqref{eq:disc-Helmholtz}, so that $v_\disc=v^0_\disc$ by well-posedness of this system. Consequently,  $\mathfrak{S}(0,v^\mathrm{in}) = v_\disc$ and, therefore, 
\begin{align*}
    (\mathcal{P}^0_{\tau,\disc} f)(v^\mathrm{in}) = \mathbb{E} \left[ f\big(\Pi_\disc \mathfrak{S}(0,v^\mathrm{in}) \big) \right]  =  f\big(\Pi_\disc  v_\disc \big) = f(v^\mathrm{in}),
\end{align*}
which completes the verification of~\ref{it:semigroup-b}.

\underline{Addressing~\ref{it:semigroup-a}.} We provide an informal argument only. It can be made rigorous through an approximation of the appearing integrals by finite sums and a limit passage. 

Let $n \in \mathbb{N}$, $f \in C_b\big(\mathbb{L}^2(\mathcal{O})\big)$ and $v^\mathrm{in} \in \mathbb{L}^2(\mathcal{O})$. Most importantly, the time-stepping algorithm uses a sequence of independent and identically distributed random variables. Moreover, since the velocity at the time index~$n$ is generated by applying the single full-step algorithm $n$-times, we can represent the velocity as follows:  
\begin{align} \label{eq:rep-velocity}
    \mathfrak{S}(n,v^\mathrm{in}) = \mathscr{S}_{\mathrm{CN}}^{\diamond n}[\mathcal{K}_\disc v^\mathrm{in}, \Delta_1 W, \ldots, \Delta_n W],
\end{align}
where $\mathcal{K}_\disc$ is the $L^2$-projection onto $E_{\disc,0}$, and $\mathscr{S}_{\mathrm{CN}}^{\diamond n}$ is recursively defined by:
\begin{align}
\begin{aligned} \label{eq:iter-full-step}
    \mathscr{S}_{\mathrm{CN}}^{\diamond 1}[v,Z_1] &:= \mathscr{S}_{\mathrm{CN}}[v,Z_1], \\
    \mathscr{S}_{\mathrm{CN}}^{\diamond n}[v,Z_1,\ldots,Z_n] &:= \mathscr{S}_{\mathrm{CN}} \left[ \mathscr{S}_{\mathrm{CN}}^{\diamond (n-1)}[v,Z_1,\ldots,Z_{n-1}], Z_n \right],
\end{aligned}
\end{align}
for $v \in E_{\disc,0}$ and~$Z_1, \ldots, Z_n \in \mathbb{R}$.

Using the definition of $P_{\tau,\disc}$ and the representation of the approximate velocity (see Definition~\ref{def:discrete-trans-group} and Equation~\eqref{eq:rep-velocity}, respectively), we find 
\begin{align*}
     &\big(P^{n-1}_{\tau,\disc} (Pf)\big)(v^\mathrm{in}) \\
     &\quad =  \int_{\mathbb{L}^2(\mathcal{O})} \left[ \int_{\mathbb{L}^2(\mathcal{O})} f(\overline{z}) \mathbb{P} \left( \Pi_\disc \mathfrak{S}(1,z) \in \dd \overline{z} \right)   \right] \mathbb{P} \left( \Pi_\disc \mathfrak{S}(n-1,v^\mathrm{in}) \in \dd z \right) \\
     &\quad =  \int_{\mathbb{L}^2(\mathcal{O})} \left[ \int_{\mathbb{L}^2(\mathcal{O})} f(\overline{z}) \mathbb{P} \left( \Pi_\disc \mathscr{S}_{\mathrm{CN}}[\mathcal{K}_\disc z,\Delta_1 W] \in \dd \overline{z} \right)   \right] \mathbb{P} \left( \Pi_\disc \mathfrak{S}(n-1,v^\mathrm{in}) \in \dd z \right).
\end{align*}
In~\eqref{eq:cont-wrt-random}, we established the continuity of the full-step solution operator with respect to the random update. This, and the fact that $\Delta_1 W$ and $\Delta_n W$ have the same law imply, for any~$A \in \mathcal{B}\big( \mathbb{L}^2(\mathcal{O}) \big)$,
\begin{align*}
    \mathbb{P} \left( \Pi_\disc \mathscr{S}_{\mathrm{CN}}[\mathcal{K}_\disc z,\Delta_1 W] \in A \right) &= \mathbb{P} \left( \Delta_1 W \in \left(\Pi_\disc \mathscr{S}_{\mathrm{CN}}[\mathcal{K}_\disc z,\cdot ] \right)^{-1} (A) \right) \\
    &= \mathbb{P} \left( \Delta_n W \in \left(\Pi_\disc \mathscr{S}_{\mathrm{CN}}[\mathcal{K}_\disc z,\cdot ] \right)^{-1} (A) \right) \\
    &= \mathbb{P} \left( \Pi_\disc \mathscr{S}_{\mathrm{CN}}[\mathcal{K}_\disc z,\Delta_n W] \in A \right).
\end{align*}
Thus, 
\begin{align*}
     &\big(P^{n-1}_{\tau,\disc} (Pf)\big)(v^\mathrm{in}) \\
     &\quad =  \int_{\mathbb{L}^2(\mathcal{O})} \left[ \int_{\mathbb{L}^2(\mathcal{O})} f(\overline{z}) \mathbb{P} \left( \Pi_\disc \mathscr{S}_{\mathrm{CN}}[\mathcal{K}_\disc z,\Delta_n W] \in \dd \overline{z} \right)   \right] \mathbb{P} \left( \Pi_\disc \mathfrak{S}(n-1,v^\mathrm{in}) \in \dd z \right)\\
     &\quad =  \int_{\mathbb{L}^2(\mathcal{O})} f(\overline{z}) \nu(\mathrm{d} \overline{z}), 
\end{align*}
where the last equality follows from the Fubini theorem, and the probability measure~$\nu$ is given by: for $A \in \mathcal{B}\big( \mathbb{L}^2(\mathcal{O}) \big)$,
\begin{align*}
    \nu(A) = \int_{\mathbb{L}^2(\mathcal{O})}  \mathbb{P} \left( \Pi_\disc \mathscr{S}_{\mathrm{CN}}[\mathcal{K}_\disc z,\Delta_n W] \in  A \right)   \mathbb{P} \left( \Pi_\disc \mathfrak{S}(n-1,v^\mathrm{in}) \in \dd z \right).
\end{align*}
The independence of the family $(\Delta_i W)_{i=1}^n$, and the fact that~$\mathcal{K}_\disc \Pi_\disc \mathfrak{S}(n-1,v^\mathrm{in}) = \mathfrak{S}(n-1,v^\mathrm{in})$ since $\mathfrak{S}(n-1,v^\mathrm{in}) \in E_{\disc,0}$, ensure
\begin{align*}
    \nu(A) &= \int_{\mathbb{L}^2(\mathcal{O})}  \mathbb{P} \left( \Pi_\disc \mathscr{S}_{\mathrm{CN}}[\mathcal{K}_\disc z,\Delta_n W] \in  A \big| z = \Pi_\disc \mathfrak{S}(n-1,v^\mathrm{in}) \right)   \\
    &\qquad \qquad \times \mathbb{P} \left( \Pi_\disc \mathfrak{S}(n-1,v^\mathrm{in}) \in \dd z \right) \\
    &= \mathbb{P} \left( \Pi_\disc \mathscr{S}_{\mathrm{CN}}[\mathfrak{S}(n-1,v^\mathrm{in}),\Delta_n W] \in  A  \right).
\end{align*}
It remains to observe that $\mathscr{S}_{\mathrm{CN}}[\mathfrak{S}(n-1,v^\mathrm{in}),\Delta_n W] = \mathfrak{S}(n,v^\mathrm{in})$, which follows from~\eqref{eq:rep-velocity} and~\eqref{eq:iter-full-step}. Therefore, we conclude 
\begin{align*}
 \int_{\mathbb{L}^2(\mathcal{O})} f(\overline{z}) \nu(\mathrm{d} \overline{z}) = \int_{\mathbb{L}^2(\mathcal{O})} f(\overline{z}) \mathbb{P} \left( \Pi_\disc \mathfrak{S}(n,v^\mathrm{in}) \in  \dd \overline{z} \right)  =  \big(P^{n}_{\tau,\disc}f \big)(v^\mathrm{in}).
\end{align*}
This completes the proof of Lemma~\ref{lem:semigroup}. 
\end{proof}

\subsubsection{Uniqueness of the discrete semigroup}
Before we start the uniqueness proof, we state another result that relates increments of the tensors~$S$ and $V$. 

\begin{lemma} \label{lem:Rel-Convergence-Tensors}
The following estimates hold, uniformly for $A, B \in \big(L^p(\mathcal{O})\big)^{n\times n}$:
\begin{itemize}
\item if $ p \in [2,\infty)$, then
\begin{subequations}
\begin{align} \label{eq:rel-conv-pBig-01}
\norm{A - B}_{L^p(\mathcal{O})}^p \lesssim{}& \norm{V(A) - V(B)}_{L^2(\mathcal{O})}^2,\\ \nonumber
\norm{S(A) - S(B)}_{L^{p'}(\mathcal{O})} \lesssim{}& \norm{V(A) - V(B)}_{L^2(\mathcal{O})}\\  \label{eq:rel-conv-pBig-02}
&\times\left( \kappa^{p/2} \abs{\mathcal{O}} + \norm{A}_{L^p(\mathcal{O})}^p + \norm{B}_{L^p(\mathcal{O})}^p \right)^{\frac{2-p'}{2p'}};
\end{align}
\end{subequations}
\item if $p \in (1,2]$, then 
\begin{subequations}
\begin{align} \label{eq:rel-conv-pSmall-01}
\norm{S(A) - S(B)}_{L^{p'}(\mathcal{O})}^{p'} \lesssim{}& \norm{V(A) - V(B)}_{L^2(\mathcal{O})}^2,\\ \nonumber
\norm{A-B}_{L^p(\mathcal{O})} \lesssim{}& \norm{V(A) - V(B)}_{L^2(\mathcal{O})} \\ \label{eq:rel-conv-pSmall-02}
&\times\left( \kappa^{p/2} \abs{\mathcal{O}} + \norm{A}_{L^p(\mathcal{O})}^p + \norm{B}_{L^p(\mathcal{O})}^p \right)^{\frac{2-p}{2p}}.
\end{align}
\end{subequations}
\end{itemize}
\end{lemma}
\begin{proof}
We start by observing that 
\begin{align} \label{rep:V}
    \abs{V(A) - V(B)}^2 & \eqsim (\sqrt{\kappa} + \abs{A} + \abs{A-B})^{p-2} \abs{A-B}^2, \\ \label{rep:S}
    \abs{S(A) - S(B)} & \eqsim (\sqrt{\kappa} + \abs{A} + \abs{A-B})^{p-2} \abs{A-B},
\end{align}
which both follow from a similar argumentation as for the verification of~\cite[Lemma~3.1]{Berselli2008}.

Firstly, we discuss the pure increment inequalities: Inequalities~\eqref{eq:rel-conv-pBig-01} and~\eqref{eq:rel-conv-pSmall-01}. If $p \geq 2$, then clearly
\begin{align*}
    \abs{A-B}^p \leq (\sqrt{\kappa} + \abs{A} + \abs{A-B})^{p-2} \abs{A-B}^2 \eqsim \abs{V(A) - V(B)}^2.
\end{align*}
If $p \in (1,2)$, then, using $(p-2) + (2-p')= (p-2)p'$,
\begin{align*}
    \abs{S(A) - S(B)}^{p'} &\eqsim   (\sqrt{\kappa} + \abs{A} + \abs{A-B})^{(p-2)p'} \abs{A-B}^{p'} \\
    &=   (\sqrt{\kappa} + \abs{A} + \abs{A-B})^{p-2} \abs{A-B}^{2} \left( \frac{\abs{A-B}}{ \sqrt{\kappa} + \abs{A} + \abs{A-B}} \right)^{p'-2} \\
    &\leq (\sqrt{\kappa} + \abs{A} + \abs{A-B})^{p-2} \abs{A-B}^{2} \eqsim \abs{V(A) - V(B)}^2.
\end{align*}
Integration in space and an application of the above inequalities establish~\eqref{eq:rel-conv-pBig-01} and~\eqref{eq:rel-conv-pSmall-01}.

Secondly, we address the impure increment inequalities: Inequalities~\eqref{eq:rel-conv-pBig-02} and~\eqref{eq:rel-conv-pSmall-02}. Both follow from H\"older's inequality and the representations for $V$ and $S$ (see~\eqref{rep:V} and~\eqref{rep:S}, respectively). If $p \in (1,2)$, then Inequality~\eqref{eq:rel-conv-pSmall-02} is derived in~\cite[Lemma~4.1]{Berselli2008}. If $p \geq 2$, then, using $(p-2)p' = (p'-2)p$ and H\"older's inequality,
\begin{align*}
    &\norm{S(A) - S(B)}_{L^{p'}}^{p'} \eqsim \int \left( \sqrt{\kappa} + \abs{A} + \abs{A-B} \right)^{(p-2) p'} \abs{A-B}^{p'} \dd x \\
    &\qquad =  \int \left( \sqrt{\kappa} + \abs{A} + \abs{A-B} \right)^{\frac{(p-2) p'}{2}} \abs{A-B}^{p'} \left( \sqrt{\kappa} + \abs{A} + \abs{A-B} \right)^{\frac{(p'-2) p}{2}} \dd x \\
    &\qquad\leq \left( \int \left( \sqrt{\kappa} + \abs{A} + \abs{A-B} \right)^{p-2} \abs{A-B}^2 \dd x \right)^{\frac{p'}{2}} \\
    &\hspace{8em} \times \left( \int \left( \sqrt{\kappa} + \abs{A} + \abs{A-B} \right)^{p} \dd x\right)^{\frac{2-p'}{2}}.
\end{align*}
Taking the $p'$-th root, estimating $( \sqrt{\kappa} + \abs{A} + \abs{A-B})^{p} \lesssim \kappa^{p/2}+ \abs{A}^p + \abs{B}^p$, and applying the representation of $V$ establish the claim.  
\end{proof}

\begin{proof}[Proof of Theorem~\ref{thm:unique}]
    Let us assume that there are two invariant measures~$\mu$ and $\overline{\mu}$ with respect to the semigroup~$\mathcal{P}_{\tau,\disc}$. We will show that, for any Lipschitz-continuous $f \in C_b\big( \mathbb{L}^2(\mathcal{O}) \big)$:
\begin{align*}
    \abs{\int_{\mathbb{L}^2(\mathcal{O})} f(x)\mu(\mathrm{d}x) - \int_{\mathbb{L}^2(\mathcal{O})} f(x)\overline{\mu}(\mathrm{d}x) } = 0,
\end{align*}
which will imply $\mu = \overline{\mu}$.

Let $f \in C_b\big( \mathbb{L}^2(\mathcal{O}) \big)$ be Lipschitz-continuous and $n \in \mathbb{N}$. Since $\mu$ and $\overline{\mu}$ are $\mathcal{P}_{\tau,\disc}$-invariant, we find that
\begin{align*}
    &\int_{\mathbb{L}^2(\mathcal{O})} f(u)\mu(\mathrm{d}u) - \int_{\mathbb{L}^2(\mathcal{O})} f(v)\overline{\mu}(\mathrm{d}v) \\
    &\qquad = \frac{1}{n} \sum_{\ell=1}^{n}  \left[ \int_{\mathbb{L}^2(\mathcal{O})} \mathcal{P}_{\tau,\disc}^\ell f(u)\mu(\mathrm{d}u) - \int_{\mathbb{L}^2(\mathcal{O})}  
 \mathcal{P}_{\tau,\disc}^\ell f(v)\overline{\mu}(\mathrm{d}v) \right] \\
    &\qquad = \int_{\mathbb{L}^2(\mathcal{O})} \int_{\mathbb{L}^2(\mathcal{O})} \frac{1}{n} \sum_{\ell=1}^{n} \mathbb{E} \left[ f(\Pi_\disc \mathfrak{S}(n,u) ) - f(\Pi_\disc \mathfrak{S}(n,v) ) \right] \mu(\mathrm{d}u) \overline{\mu}(\mathrm{d}v).
\end{align*}
Recalling that $f$ is bounded, by dominated convergence theorem the right-hand side above will converge to $0$ provided we show that, for any $u,v \in L^2(\mathcal{O})$:
\begin{align*}
    \lim_{n\to \infty} \frac{1}{n} \sum_{\ell=1}^{n}  \mathbb{E} \left[ f(\Pi_\disc \mathfrak{S}(\ell,u) ) - f(\Pi_\disc \mathfrak{S}(\ell,v) ) \right] = 0.
\end{align*}
Since $f$ is Lipschitz-continuous, it is sufficient to verify:
\begin{align} \label{eq:toShow}
    \lim_{n\to \infty}  \frac{1}{n} \sum_{\ell=1}^{n} \mathbb{E}  \left[ \norm{\Pi_\disc \mathfrak{S}(\ell,u) - \Pi_\disc \mathfrak{S}(\ell,v) }_{L^2(\mathcal{O})} \right] = 0.
\end{align}
 Notice that the difference of velocity satisfies~$\mathbb{P}$-a.s.:
\begin{align}
\begin{aligned} \label{eq:sol-dif-unique}
    &\norm{\Pi_\disc \mathfrak{S}(n,u) -\Pi_\disc \mathfrak{S}(n,v)}_{L^2(\mathcal{O})}^2 \\
    &\qquad + \sum_{\ell = 0}^{n-1} \tau \norm{ V(\varepsilon_\disc \mathfrak{S}^{1/2}(\ell,u) + \varepsilon g) - V(\varepsilon_\disc \mathfrak{S}^{1/2}(\ell,v) + \varepsilon g)}_{L^2(\mathcal{O})}^2 \\
    &\eqsim  \norm{\Pi_\disc \mathcal{K}_\disc u -\Pi_\disc \mathcal{K}_\disc v }_{L^2(\mathcal{O})}^2.
\end{aligned}
\end{align}
This equivalence is derived from subtracting~\eqref{eq:Evolution-01} for $ \mathfrak{S}(\ell,u)$ and $ \mathfrak{S}(\ell,v)$, respectively, choosing $\xi = \mathfrak{S}^{1/2}(\ell,u) - \mathfrak{S}^{1/2}(\ell,u)$ as the test function, and a similar reasoning as in the proof of Theorem~\ref{thm:apriori-vel} to cancel the noise term. 

The right-hand-side of~\eqref{eq:sol-dif-unique} is independent of the time horizon~$n$. Therefore, after taking expectations and normalisation, we infer that
\begin{align} \label{eq:conv-V}
   \lim_{n \to \infty} \frac{1}{n} \sum_{\ell=0}^{n-1} \mathbb{E}\left[ \norm{ V(\varepsilon_\disc \mathfrak{S}^{1/2}(\ell,u) + \varepsilon g) - V(\varepsilon_\disc \mathfrak{S}^{1/2}(\ell,v) + \varepsilon g)}_{L^2(\mathcal{O})}^2 \right] = 0.
\end{align}
By invoking Lemma~\ref{lem:Rel-Convergence-Tensors}, we can transfer $L^2$-convergence of $V$ to $L^p$-convergence of the identity. To do this, we need to distinguish the cases: $p \geq 2$ and $p \in (1,2)$.

Firstly, let $p \geq 2$. Introducing $\pm\varepsilon g$ and utilising Inequality~\eqref{eq:rel-conv-pBig-01}, we find that
\begin{align*}
   &\frac{1}{n} \sum_{\ell=0}^{n-1}\mathbb{E}\left[ \norm{ \varepsilon_\disc \mathfrak{S}^{1/2}(\ell,u) - \varepsilon_\disc \mathfrak{S}^{1/2}(\ell,v)}_{L^p(\mathcal{O})}^p \right]  \\
   &\qquad \lesssim \frac{1}{n} \sum_{\ell=0}^{n-1} \mathbb{E}\left[ \norm{ V(\varepsilon_\disc \mathfrak{S}^{1/2}(\ell,u) + \varepsilon g) - V(\varepsilon_\disc \mathfrak{S}^{1/2}(\ell,v) + \varepsilon g)}_{L^2(\mathcal{O})}^2 \right],
\end{align*}
where the right-hand-side vanishes asymptotically thanks to~\eqref{eq:conv-V}.

Secondly, let $p \in (1,2)$. Here, we introduce $\pm\varepsilon g$ and use Inequality~\eqref{eq:rel-conv-pSmall-02} together with H\"older's inequality to deduce
\begin{align*}
    &\frac{1}{n} \sum_{\ell=0}^{n-1} \mathbb{E}\left[ \norm{ \varepsilon_\disc \mathfrak{S}^{1/2}(\ell,u) - \varepsilon_\disc \mathfrak{S}^{1/2}(\ell,v)}_{L^p(\mathcal{O})}^p \right] \\
    &\quad \lesssim \left(\frac{1}{n} \sum_{\ell=0}^{n-1} \mathbb{E}\left[ \norm{ V(\varepsilon_\disc \mathfrak{S}^{1/2}(\ell,u) + \varepsilon g) - V(\varepsilon_\disc \mathfrak{S}^{1/2}(\ell,v) + \varepsilon g)}_{L^2(\mathcal{O})}^2 \right] \right)^{\frac{p}{2}} \\
    &\qquad \times
    \left( \frac{1}{n} \sum_{\ell=0}^{n-1} \mathbb{E}\left[ \norm{\varepsilon_\disc \mathfrak{S}^{1/2}(\ell,u) +\varepsilon g}_{L^p(\mathcal{O})}^p + \norm{\varepsilon_\disc \mathfrak{S}^{1/2}(\ell,v) +\varepsilon g}_{L^p(\mathcal{O})}^p  + 1 \right] \right)^{\frac{2-p}{2}}.
\end{align*}
 While~\eqref{eq:conv-V} ensures the convergence to $0$ of the first factor, Inequality~\eqref{eq:apriori-vel} guarantees that the second factor is uniformly bounded. This implies the convergence to $0$ of the right-hand-side. 

Combining both cases, we have shown that
\begin{align} \label{eq:sym-limit-half}
    \lim_{n \to \infty}  \frac{1}{n} \sum_{\ell=0}^{n-1} \mathbb{E}\left[ \norm{ \varepsilon_\disc \mathfrak{S}^{1/2}(\ell,u) - \varepsilon_\disc \mathfrak{S}^{1/2}(\ell,v)}_{L^p(\mathcal{O})}^p \right] = 0,
\end{align}
which, together with an inverse estimate (see Equation~\eqref{eq:inverse-first}), implies:
\begin{align} \label{eq:limit-half}
    \lim_{n \to \infty} \frac{1}{n} \sum_{\ell=0}^{n-1} \mathbb{E}\left[ \norm{ \Pi_\disc \mathfrak{S}^{1/2}(\ell,u) - \Pi_\disc \mathfrak{S}^{1/2}(\ell,v)}_{L^2(\mathcal{O})}^p \right] = 0.
\end{align}
This establishes the convergence of the time-averaged arithmetic mean-values.

Next, we verify the convergence of the time-averaged difference at integer times. Let us denote the differences at integer and shifted integer times by: 
\begin{align*}
    \delta^{n} := \mathfrak{S}(n,u) - \mathfrak{S}(n,v) \qquad \text{ and }  \qquad \delta^{n+1/2} := \mathfrak{S}^{1/2}(n,u) - \mathfrak{S}^{1/2}(n,v).
\end{align*} 

Notice that the difference at integer times satisfies the equation, $\mathbb{P}$-a.s.~for all $\xi \in E_{\disc,0}$:
\begin{align*}
    &\left( \Pi_\disc \delta^{n+1} - \Pi_\disc \delta^{n}, \Pi_\disc \xi \right)\\
    &\qquad = -\tau \left(S(\varepsilon_\disc \mathfrak{S}^{1/2}(n,u) + \varepsilon g) - S(\varepsilon_\disc \mathfrak{S}^{1/2}(n,v) + \varepsilon g), \varepsilon_\disc \xi  \right) \\
    &\qquad \qquad + B_\disc^\sigma\left(\delta^{n+1/2},\xi \right) \Delta_{n+1} W.  
\end{align*}
Choosing $\xi = \delta^{n+1} - \delta^n \in E_{\disc,0}$ yields
\begin{align}
\begin{aligned} \label{eq:pre-closing}
    &\norm{ \Pi_\disc \delta^{n+1} - \Pi_\disc \delta^{n}}_{L^2(\mathcal{O})}^2 \\
    &\qquad = -\tau \left(S(\varepsilon_\disc \mathfrak{S}^{1/2}(n,u) + \varepsilon g) - S(\varepsilon_\disc \mathfrak{S}^{1/2}(n,v) + \varepsilon g), \varepsilon_\disc ( \delta^{n+1} - \delta^n)  \right) \\
    &\qquad \qquad + B_\disc^\sigma\left(\delta^{n+1/2},\delta^{n+1} - \delta^n\right) \Delta_{n+1} W.
\end{aligned}
\end{align}
We will discuss the terms on the right-hand-side separately. 

An application of H\"older's inequality and weighted Young's inequality, together with an inverse estimate (see Equation~\eqref{eq:inverse-first}) show, for arbitrary $\delta >0$, 
\begin{align*}
     &\tau \left(S(\varepsilon_\disc \mathfrak{S}^{1/2}(n,u) + \varepsilon g) - S(\varepsilon_\disc \mathfrak{S}^{1/2}(n,v) + \varepsilon g), \varepsilon_\disc ( \delta^{n+1} - \delta^n)  \right) \\
&\qquad \leq c_\delta \tau^2 \norm{S(\varepsilon_\disc \mathfrak{S}^{1/2}(n,u) + \varepsilon g) - S(\varepsilon_\disc \mathfrak{S}^{1/2}(n,v) + \varepsilon g)}_{L^{p'}(\mathcal{O})}^{2} \\
&\qquad \qquad+ \delta  \big(\mathfrak{B}_\disc(p) \big)^2  \norm{\Pi_\disc ( \delta^{n+1} - \delta^n) }_{L^{2}(\mathcal{O})}^2. 
\end{align*}
The first term will vanish in expectation asymptotically; the second term can be absorbed by choosing~$\delta $ sufficiently small. 

Next, we study the stochastic term. Recalling the definition of the noise coefficient (see~\eqref{def:discrete-noise-coefficient}) and applying H\"older's inequality and weighted Young's inequality we find, for arbitrary $\delta > 0$, 
\begin{align*}
    &B_\disc^\sigma\left(\delta^{n+1/2},\delta^{n+1} - \delta^n\right) \Delta_{n+1} W \\
    &\quad \leq c_\delta \norm{\sigma}_{L^\infty(\mathcal{O})}^2 \left( \norm{\nabla_\disc \delta^{n+1/2}}_{L^2(\mathcal{O})}^2 + \norm{\Pi_\disc \delta^{n+1/2} }_{L^2(\mathcal{O})}^2 \right)\abs{\Delta_{n+1} W}^2 \\
    &\qquad+ \delta \left( \norm{\Pi_\disc (\delta^{n+1} - \delta^n)  }_{L^2(\mathcal{O})}^2  + \norm{\nabla_\disc (\delta^{n+1} - \delta^n)  }_{L^2(\mathcal{O})}^2 \right). 
\end{align*}
Utilising the coercivity constant (see~\eqref{def:discrete-poincare}) and an inverse estimate, we further estimate
\begin{align*}
    &B_\disc^\sigma\left(\delta^{n+1/2},\delta^{n+1} - \delta^n\right) \Delta_{n+1} W \\
    &\quad \leq c_\delta \norm{\sigma}_{L^\infty(\mathcal{O})}^2 \big( C_\disc(2)\big)^2 \norm{\varepsilon_\disc \delta^{n+1/2}}_{L^2(\mathcal{O})}^2 \abs{\Delta_{n+1} W}^2 \\
    &\qquad + \delta \left( 1 + \big( C_\disc(2) \mathfrak{B}_\disc(2) \big)^2 \right)  \norm{\Pi_\disc (\delta^{n+1} - \delta^n)  }_{L^2(\mathcal{O})}^2.
\end{align*}
Again, by choosing~$\delta$ sufficiently small, we can absorb the second term, while the first term vanishes in expectation asymptotically.

Now, we are ready to use the above estimates in~\eqref{eq:pre-closing}. But first, we fix~$\delta >0$ sufficiently small. Then we apply them to find
\begin{align*}
   &\norm{ \Pi_\disc \delta^{n+1} - \Pi_\disc \delta^{n}}_{L^2(\mathcal{O})}^2 \\
   &\qquad \lesssim \norm{S(\varepsilon_\disc \mathfrak{S}^{1/2}(n,u) + \varepsilon g) - S(\varepsilon_\disc \mathfrak{S}^{1/2}(n,v) + \varepsilon g)}_{L^{p'}(\mathcal{O})}^{2} \\
   &\qquad \qquad +  \norm{\varepsilon_\disc \delta^{n+1/2}}_{L^2(\mathcal{O})}^2 \abs{\Delta_{n+1} W}^2.
\end{align*}
Taking the square root, expectation and time-average, followed by an invokation of H\"older's inequality and the equivalence of norms in finite dimensions,
\begin{align*}
    & \frac{1}{n} \sum_{\ell=0}^{n-1} \mathbb{E}\left[ \norm{ \Pi_\disc \delta^{\ell+1} - \Pi_\disc \delta^{\ell}}_{L^2(\mathcal{O})} \right] \\
    &\qquad \lesssim \left(  \frac{1}{n} \sum_{\ell=0}^{n-1} \mathbb{E}\left[\norm{S(\varepsilon_\disc \mathfrak{S}^{1/2}(\ell,u) + \varepsilon g) - S(\varepsilon_\disc \mathfrak{S}^{1/2}(\ell,v) + \varepsilon g)}_{L^{p'}(\mathcal{O})}^{p'} \right] \right)^{1/p'} \\
    &\qquad \qquad +\left( \frac{1}{n} \sum_{\ell=0}^{n-1} \mathbb{E}\left[\norm{\varepsilon_\disc \delta^{\ell+1/2}}_{L^p(\mathcal{O})}^p  \right] \right)^{1/p} \left( \frac{1}{n} \sum_{\ell=0}^{n-1} \mathbb{E}\left[ \abs{\Delta_{\ell+1} W}^{p'} \right] \right)^{1/p'}.
\end{align*}
Notice that the first term vanishes asymptotically, which follows from an analogous argumentation as for the Identity~\eqref{eq:sym-limit-half}. Additionally, Identity~\eqref{eq:sym-limit-half} also implies the convergence to $0$ of the second term since Wiener increments have arbitrarily high moments, i.e., 
\begin{align*}
   \sup_{n \in \mathbb{N}}  \left( \frac{1}{n} \sum_{\ell=0}^{n-1} \mathbb{E}\left[ \abs{\Delta_{\ell+1} W}^{p'} \right] \right)^{1/p'} \lesssim 1.
\end{align*}
Thus, we have established that
\begin{align} \label{eq:dif-vanish}
   \lim_{n\to \infty} \frac{1}{n} \sum_{\ell=0}^{n-1} \mathbb{E}\left[ \norm{ \Pi_\disc \delta^{\ell+1} - \Pi_\disc \delta^{\ell}}_{L^2(\mathcal{O})} \right] = 0.
\end{align}

The convergence of the time-averaged arithmetic mean-values of velocity and velocity increments is enough to prove the convergence of velocity at integer times since: $\delta^{n+1} = \delta^{n+1/2}  + \tfrac{\delta^{n+1} - \delta^{n}}{2}$. Indeed, using~\eqref{eq:limit-half} and~\eqref{eq:dif-vanish}, we conclude that
\begin{align*}
    \lim_{n\to \infty} \frac{1}{n} \sum_{\ell=1}^{n}  \mathbb{E}\left[ \norm{\Pi_\disc \delta^{\ell}}_{L^2(\mathcal{O})} \right]  &\leq \lim_{n\to \infty}\frac{1}{n} \sum_{\ell=0}^{n-1}  \mathbb{E}\left[ \norm{\Pi_\disc \delta^{\ell+1/2}}_{L^2(\mathcal{O})} \right] \\
    &\qquad + \lim_{n\to \infty} \frac{1}{n} \sum_{\ell=0}^{n-1} 
 \mathbb{E}\left[ \norm{\Pi_\disc (\delta^{\ell+1} - \delta^{\ell})}_{L^2(\mathcal{O})} \right] = 0.
\end{align*}
Recalling that $\delta^{\ell} := \mathfrak{S}(\ell,u) - \mathfrak{S}(\ell,v)$, this shows \eqref{eq:toShow} and completes the proof.
\end{proof}

\subsubsection{Asymptotic invariance of the first sequence of measures}
\begin{proof}[Proof of Lemma~\ref{lem:rel-1st-measure-semigroup}]
As in the proof of Lemma~\ref{lem:semigroup}, one can show that  
\begin{align*}
   \mathbb{P}\left(\Pi_\disc \mathfrak{S}(k + m, v^\mathrm{in}) \in  A \right) =   \int_{\mathbb{L}^2(\mathcal{O})}  \mathbb{P}\left(\Pi_\disc \mathfrak{S}(k, v) \in  A \right) \mathbb{P} \left( \Pi_\disc \mathfrak{S}(m, v^\mathrm{in}) \in \mathrm{d} v \right),
\end{align*}
for any $A \in \mathcal{B}\big( \mathbb{L}^2(\mathcal{O}) \big)$, $k$ and $m \in \mathbb{N}$. 

Now, let $n$ and $N \in \mathbb{N}$ be the time index of the semigroup and the sequence index of the measure, respectively. Using the definition of the semigroup and the measure (see Definition~\ref{def:discrete-trans-group} and~\eqref{def:aprox-measure}, respectively), together with the above identity, we find
\begin{align*}
    &\int_{\mathbb{L}^2(\mathcal{O})} (P_{\tau,\disc}^n f)(v) \mu_{\tau,\disc}^N (v^\mathrm{in}; \mathrm{d} v) \\
    &\hspace{2em} = \frac{1}{N} \sum_{\ell=0}^{N-1}  \int_{\mathbb{L}^2(\mathcal{O})} \left[ \int_{\mathbb{L}^2(\mathcal{O})} f(z) \mathbb{P}\left(\Pi_\disc \mathfrak{S}(n, v) \in  \mathrm{d} z \right) \right] \mathbb{P} \left( \Pi_\disc \mathfrak{S}(\ell, v^\mathrm{in}) \in \mathrm{d} v \right) \\
    &\hspace{2em} = \frac{1}{N} \sum_{\ell=0}^{N-1}  \int_{\mathbb{L}^2(\mathcal{O})} f(v) \mathbb{P} \left( \Pi_\disc \mathfrak{S}(n+\ell, v^\mathrm{in}) \in \mathrm{d} v \right).
\end{align*}
An index shift shows
\begin{align*}
    &\frac{1}{N} \sum_{\ell=0}^{N-1}  \int_{\mathbb{L}^2(\mathcal{O})} f(v) \mathbb{P} \left( \Pi_\disc \mathfrak{S}(n+\ell, v^\mathrm{in}) \in \mathrm{d} v \right) \\
    & \hspace{2em}=  \frac{1}{N}  \sum_{\ell=0}^{N-1}  \int_{\mathbb{L}^2(\mathcal{O})} f(v) \mathbb{P} \left( \Pi_\disc \mathfrak{S}(\ell, v^\mathrm{in}) \in \mathrm{d} v \right) \\
    &\hspace{4em}  +  \frac{1}{N} \sum_{\ell=0}^{n-1}  \int_{\mathbb{L}^2(\mathcal{O})} f(v) \mathbb{P} \left( \Pi_\disc \mathfrak{S}(N+\ell, v^\mathrm{in}) \in \mathrm{d} v \right)\\
     &\hspace{4em}-  \frac{1}{N} \sum_{\ell=0}^{n-1} \int_{\mathbb{L}^2(\mathcal{O})} f(v) \mathbb{P} \left( \Pi_\disc \mathfrak{S}(\ell, v^\mathrm{in}) \in \mathrm{d} v \right).
\end{align*}
Thus, 
\begin{align}
    &\Bigg|\int_{\mathbb{L}^2(\mathcal{O})} (P_{\tau,\disc}^n f)(v) \mu_{\tau,\disc}^N (v^\mathrm{in}; \mathrm{d} v)  - \int_{\mathbb{L}^2(\mathcal{O})} f(v) \mu_{\tau,\disc}^{N} (v^\mathrm{in}; \mathrm{d} v) \Bigg|\nonumber\\
    &\hspace{2em}= \Bigg|\frac{1}{N} \sum_{\ell=0}^{n-1}  \int_{\mathbb{L}^2(\mathcal{O})} f(v) \mathbb{P} \left( \Pi_\disc \mathfrak{S}(N+\ell, v^\mathrm{in}) \in \mathrm{d} v \right) \nonumber\\
     &\hspace{4em}-  \frac{1}{N} \sum_{\ell=0}^{n-1} \int_{\mathbb{L}^2(\mathcal{O})} f(v) \mathbb{P} \left( \Pi_\disc \mathfrak{S}(\ell, v^\mathrm{in}) \in \mathrm{d} v \right) \Bigg|\nonumber\\
     &\hspace{2em} \leq  2 \sup_{v \in \mathbb{L}^2(\mathcal{O})}\abs{f(v)}~ \frac{n}{N}.
\label{eq:est.Pf.minus.f}
\end{align}
This finishes the proof of Lemma~\ref{lem:rel-1st-measure-semigroup}.
\end{proof}

\subsubsection{Existence of a limit measure for the second sequence of measures}
\begin{proof}[Proof of Lemma~\ref{lem:limit-measure}]
By the Prokhorov theorem, it is sufficient to verify that the sequence of probability measures $\big( \mu_{\tau,\disc}^{1/2,N}( v^\mathrm{in}; \cdot) \big)_{N \in \mathbb{N}}$, defined in~\eqref{def:aprox-measure-shift} is tight. Thus, let $\eta > 0$. We need to find a compact set~$\mathscr{K}_\eta \subset \mathbb{L}^2(\mathcal{O})$ such that 
\begin{align}\label{eq:tight-show}
    \forall N \in \mathbb{N}: \quad \mu_{\tau}^{1/2,N}( v^\mathrm{in};  \mathscr{K}_\eta) > 1- \eta.
\end{align}

Let $R \in [0,\infty)$ and define 
\begin{align*}
    D_R := \big\{ u \in \mathbb{L}^2(\mathcal{O}) \big|\, \exists v_\disc \in X_{\disc,0}:~u = \Pi_\disc v_\disc~\text{ and } \norm{\varepsilon_\disc v_\disc}_{L^p(\mathcal{O})} \leq R \big\}.
\end{align*}
The set~$D_R$ contains all vector fields that can be represented by a reconstructed discrete velocity with an additional, quantified control over the reconstructed symmetric gradient. Since $D_R\subset\Pi_\disc X_{\disc,0}$ which is finite-dimensional, the bounded set $D_R$ is compact.

The space~$\mathbb{L}^2(\mathcal{O})$ splits into three disjoint subsets: 
\begin{align*}
    \mathbb{L}^2(\mathcal{O}) = \mathbb{L}^2(\mathcal{O}) \backslash \Pi_\disc X_{\disc,0}~ \cup ~ \Pi_\disc X_{\disc,0} \backslash D_R ~\cup~  D_R,
\end{align*}
where the first, second, and third set correspond to vector fields that cannot be represented within the GD, vector fields that can be represented within the GD but lack control over the reconstructed symmetric gradient, and vector fields that can be represented within the GD with control over the reconstructed symmetric gradient, respectively. 

Next, we will show that, by adjusting~$R$ as a function of~$\eta$, the following inequality can be guaranteed:
\begin{align*}
   \forall N \in \mathbb{N}:\quad  \mu_{\tau,\disc}^{1/2,N}\big(v^\mathrm{in}; D_{R(\eta)} \big) > 1 - \eta. 
\end{align*}
Thus, Inequality~\eqref{eq:tight-show} will follow with $\mathscr{K}_\eta = D_{R(\eta)}$ and the proof will be complete.

The definition \eqref{def:aprox-measure} of $\mu_{\tau,\disc}^{1/2,N}$ shows that this measure is supported in $\Pi_\disc X_{\disc,0}$. Hence, it is sufficient to show that 
there exists $R(\eta)$ such that for all $N \in \mathbb{N}$,
\begin{align} \label{eq:tight-converse}
    \mu_{\tau,\disc}^{1/2,N}(v^\mathrm{in}; \Pi_\disc X_{\disc,0} \backslash D_{R(\eta)}  \big) < \eta.
\end{align}

Due to the Tschebycheff inequality and Theorem~\ref{thm:apriori-vel} (more precisely Inequality~\eqref{eq:apriori-vel} with~$q=1$, and recalling $\mathfrak{S}^{1/2}(n,v^\mathrm{in}) = v^{n+1/2}_\disc$ ) we find that
\begin{align*}
\mu_{\tau,\disc}^{1/2,N}(v^\mathrm{in}; \Pi_\disc X_{\disc,0} \backslash D_{R}  ) &=  \frac{1}{N} \sum_{n = 0}^{N-1} \mathbb{P}\left(\Pi_\disc \mathfrak{S}^{1/2}(n,v^\mathrm{in}) \in  \Pi_\disc X_{\disc,0} \backslash D_{R} \right)\\
&\leq  R^{-p} \frac{1}{\tau N} \mathbb{E}\left[ \sum_{n = 0}^{N-1}  \tau \norm{\varepsilon_\disc \mathfrak{S}^{1/2}(n,v^\mathrm{in}) }_{L^p(\mathcal{O})}^p \right]\\
&\leq  R^{-p} \frac{1}{\tau N} C \left( \norm{v^\mathrm{in}}_{L^2(\mathcal{O})}^{2} + \tau N \right).
\end{align*}
Since
\begin{align*}
    \sup_{N \in \mathbb{N}} \frac{1}{\tau N} C \left( \norm{v^\mathrm{in}}_{L^2(\mathcal{O})}^{2} + \tau N \right) \leq C \left( \frac{\norm{v^\mathrm{in}}_{L^2(\mathcal{O})}^{2}}{\tau} + 1 \right) < \infty, 
\end{align*}
there exists~$R(\eta)$ (independent of $N$) such that Inequality~\eqref{eq:tight-converse} holds uniformly in $N$. This finishes the proof of Lemma~\ref{lem:limit-measure}. 
\end{proof}

\subsubsection{Mismatch of invariance for the second sequence of measures}
\begin{proof}[Proof of Lemma~\ref{lem:mismatch-invariance}]
Let $f \in C_b\big(\mathbb{L}^2(\mathcal{O}) \big)$ be Lipschitz-continuous. Moreover, let $n$ and $N \in \mathbb{N}$ be the time index of the semigroup and the sequence index of the measure, respectively. Our proof strategy bases on the following decomposition:
\begin{align*}
    &\langle \mathcal{P}^n_{\tau,\disc} f - f, \mu^{1/2,N}_{\tau,\disc} \rangle =  \underbrace{\langle \mathcal{P}^n_{\tau,\disc} f - f,  \mu^{N}_{\tau,\disc}  \rangle}_{=:~\mathrm{I}} +  \underbrace{\langle  \mathcal{P}^n_{\tau,\disc} f - f,  \mu^{1/2,N}_{\tau,\disc}  -  \mu^{N}_{\tau,\disc}   \rangle}_{=:~\mathrm{II}},
\end{align*}
where we neglected the dependence of the measures on the input-velocity, and $\langle f, \mu \rangle = \int_{\mathbb{L}^2(\mathcal{O})} f(v) \mu(\mathrm{d} v)$ is an abbreviation for the pairing of a function and a measure. 

Before we address each term separately, we will derive a preparatory result that concerns propagation of Lipschitz-continuity for the semigroup. We will show the following inequality: 
\begin{align}  \label{eq:lip-propagation-1}
     \seminorm{  \mathcal{P}_{\tau,\disc}^n f}_{\mathrm{Lip}} &\leq \seminorm{ f}_{\mathrm{Lip}},
\end{align}
where $\seminorm{ f}_{\mathrm{Lip}} := \sup_{u\neq v \in \mathbb{L}^2(\mathcal{O})} \frac{\abs{f(u) - f(v)}}{\norm{u-v}_{\mathbb{L}^2(\mathcal{O})}}$, which immediately implies 
\begin{align}
    \label{eq:lip-propagation-2}
     \seminorm{  \mathcal{P}^n_{\tau,\disc} f - f}_{\mathrm{Lip}} &\leq 2 \seminorm{ f}_{\mathrm{Lip}}. 
\end{align}

\underline{Addressing~\eqref{eq:lip-propagation-1}.} Let $u,\, v \in \mathbb{L}^2(\mathcal{O})$. Using the definition of the semigroup (see Definition~\ref{def:discrete-trans-group}), and Identities~\eqref{eq:rep-velocity} and~\eqref{eq:iter-full-step}, we find
\begin{align*}
    \mathcal{P}_{\tau,\disc} f (u) -  \mathcal{P}_{\tau,\disc} f (v) &= \mathbb{E} \left[f\big(\Pi_\disc \mathfrak{S}(1,u) \big) - f\big(\Pi_\disc \mathfrak{S}(1,v) \big) \right] \\
    &= \mathbb{E} \left[f\big(\Pi_\disc \mathscr{S}_{\mathrm{CN}}[u,\Delta_1 W]\big) - f\big(\Pi_\disc \mathscr{S}_{\mathrm{CN}}[v,\Delta_1 W] \big) \right].
\end{align*}
Recall that~$f$ is assumed to be Lipschitz-continuous. Moreover, the full-step solution operator~$\mathscr{S}_{\mathrm{CN}}[\cdot,\Delta_1 W]$ is Lipschitz-continuous as we have verified in~\eqref{eq:Lip-cont-wrt-vel}. The Lipschitz-constant of the full-step solution operator is bounded by one, which follows from similar arguments as for the half-step solution operator~$\mathscr{S}_{\mathrm{imp}}^{1/2}[\cdot,\Delta_1 W]$. For more details, we refer to the proof of Lemma~\ref{lem:semigroup}. Therefore, we obtain
\begin{align*}
    &\mathbb{E} \left[f\big(\Pi_\disc \mathscr{S}_{\mathrm{CN}}[u,\Delta_1 W]\big) - f\big(\Pi_\disc \mathscr{S}_{\mathrm{CN}}[v,\Delta_1 W] \big) \right] \\
    &\qquad \leq \seminorm{f}_{\mathrm{Lip}} \mathbb{E} \left[\norm{\Pi_\disc \mathscr{S}_{\mathrm{CN}}[u,\Delta_1 W]- \Pi_\disc \mathscr{S}_{\mathrm{CN}}[v,\Delta_1 W] }_{L^2(\mathcal{O})} \right] \\
    &\qquad  \leq \seminorm{f}_{\mathrm{Lip}} \norm{u-v}_{L^2(\mathcal{O})},
\end{align*}
which implies~\eqref{eq:lip-propagation-1} for $n=1$. 

For general $n$, the assertion follows by the semigroup property and recursive application of the first case: 
\begin{align*}
    \seminorm{  \mathcal{P}^n_{\tau,\disc} f }_{\mathrm{Lip}}
=\seminorm{  \mathcal{P}_{\tau,\disc} (\mathcal P^{n-1}_{\tau,\disc}f)}_{\mathrm{Lip}}\le \seminorm{  \mathcal P^{n-1}_{\tau,\disc}f}_{\mathrm{Lip}} \leq \cdots \leq \seminorm{  f }_{\mathrm{Lip}}.
\end{align*}

Now, we estimate the terms~$\mathrm{I}$ and~$\mathrm{II}$.

\underline{Addressing~$\mathrm{I}$.} This term has already been estimated in Lemma~\ref{lem:rel-1st-measure-semigroup} (see \eqref{eq:est.Pf.minus.f}). Thus, it remains to observe that
\begin{align*}
    \sup_{u \in \mathbb{L}^2(\mathcal{O})} \abs{\mathcal{P}^n_{\tau,\disc} f(u) - f(u)} \leq 2 \sup_{u \in \mathbb{L}^2(\mathcal{O})} \abs{f(u)}, 
\end{align*}
since $\mathcal{P}^n_{\tau,\disc} f(u) = \int_{\mathbb{L}^2(\mathcal{O})} f(z) \mathbb{P}( \Pi_\disc \mathfrak{S}(n,u) \in \dd z) \leq \sup_{u \in \mathbb{L}^2(\mathcal{O})} \abs{f(u)}$.

\underline{Addressing~$\mathrm{II}$.} 
 To shorten the notation, we set $g= \mathcal{P}^n_{\tau,\disc} f - f$. Using the definition of the measures (see~\eqref{def:aprox-measure} and~\eqref{def:aprox-measure-shift}), the Lipschitz-property of~$g$, and the H\"older inequality, we derive
\begin{align*}
    &\langle g ,  \mu^{1/2,N}_{\tau,\disc} -\mu^{N}_{\tau,\disc}  \rangle = \frac{1}{N} \sum_{n=0}^{N-1} \mathbb{E}\left[ g\big(\Pi_\disc \mathfrak{S}^{1/2}(n,v^\mathrm{in}) \big) - g\big(\Pi_\disc \mathfrak{S}(n,v^\mathrm{in}) \big) \right] \\
    &\hspace{2em} \leq \seminorm{g}_{\mathrm{Lip}} \frac{1}{N} \sum_{n=0}^{N-1} \mathbb{E}\left[ \norm{\Pi_\disc \mathfrak{S}^{1/2}(n,v^\mathrm{in}) - \Pi_\disc \mathfrak{S}(n,v^\mathrm{in}) }_{L^2(\mathcal{O})} \right] \\
     &\hspace{2em}  \leq \seminorm{g}_{\mathrm{Lip}} \frac{1}{\sqrt{N}} \left( \sum_{n=0}^{N-1} \mathbb{E}\left[ \norm{\Pi_\disc \mathfrak{S}^{1/2}(n,v^\mathrm{in}) - \Pi_\disc \mathfrak{S}(n,v^\mathrm{in}) }_{L^2(\mathcal{O})}^2 \right] \right)^{1/2}.
\end{align*}
Invoking~\eqref{eq:lip-propagation-2} finishes the proof of Lemma~\ref{lem:mismatch-invariance}.
\end{proof}

\subsubsection{Consequences of trivial boundary conditions}
\begin{proof}[Verifying the claims presented in Example~\ref{ex:trivial-BC}]
Let $g = 0$. 

\underline{Characterising the invariant measure.} First, we notice that $\delta_0$ is indeed $\mathcal{P}_{\tau,\disc}$-invariant. To see this, let $v^\mathrm{in} = 0$. Invoking the pathwise energy equality (see Identity~\eqref{eq:path-stab}) yields
\begin{align*}
 \norm{\Pi_\disc \mathfrak{S}(1,0)}_{L^2(\mathcal{O})}^2 \leq   \norm{0}_{L^2(\mathcal{O})}^2 = 0,   
\end{align*}
which ensures: $\mathfrak{S}(1,0) = 0_\disc$, and consequently: $\Pi_\disc \mathfrak{S}(1,0) = 0$. Thus,
\begin{align*}
    \langle P_{\tau,\disc} f , \delta_0 \rangle = P_{\tau,\disc} f(0) = \mathbb{E}\left[ f(\Pi_\disc  0_\disc) \right] = f(0) =   \langle f , \delta_0 \rangle.
\end{align*}
Since there is at most one $\mathcal{P}_{\tau,\disc}$-invariant measure by Theorem~\ref{thm:unique}, it has to be~$\delta_0$.

Now, let $v^\mathrm{in} \in \mathbb{L}^2(\mathcal{O})$.

\underline{Convergence of velocity.} The pathwise energy identity~\eqref{eq:path-stab} ensures the convergence of the accumulated reconstructed symmetric gradients; thus, in particular~$\mathbb{P}$-a.s.: $\lim_{n \to \infty} \varepsilon_\disc \mathfrak{S}^{1/2}(n,v^\mathrm{in}) = 0$, which immediately guarantees~$\mathbb{P}$-a.s.: $\lim_{n \to \infty} \Pi_\disc\mathfrak{S}^{1/2}(n,v^\mathrm{in}) = 0$. The convergence of reconstructed velocity at integer times follows similarly to the proof of Theorem~\ref{thm:unique}; in this way, one obtains~$\mathbb{P}$-a.s.:
\begin{align} \label{eq:conv-integer}
    \lim_{n\to \infty} \norm{\Pi_\disc \mathfrak{S}(n,v^\mathrm{in})}_{L^2(\mathcal{O})} = 0.
\end{align}

\underline{Convergence of measures.}
We only present the details for the first sequence of measures; the second case follows analogously. Let $f \in C_b\big(\mathbb{L}^2(\mathcal{O})\big)$ be Lipschitz-continuous. By definition of the approximate measures and the Lipschitz-continuity of $f$, it holds
\begin{align*}
    \langle f, \mu_{\tau,\disc}^{N}(v^\mathrm{in}; \cdot) - \delta_0 \rangle &= \frac{1}{N} \sum_{n=0}^{N-1} \mathbb{E}\left[ f\big(\Pi_\disc \mathfrak{S}(n,v^\mathrm{in}) \big) - f(0) \right] \\
    &\leq C_f \frac{1}{N} \sum_{n=0}^{N-1} \mathbb{E}\left[ \norm{\Pi_\disc \mathfrak{S}(n,v^\mathrm{in}) - 0}_{L^2(\mathcal{O})} \right].
\end{align*}
Due to~\eqref{eq:conv-integer}, the right-hand-side vanishes when the time horizon~$N$ is sent to infinity, which establishes the assertion. 
\end{proof}

\section{Numerical simulations} \label{sec:num-sim}
In this section, we conduct two experiments: (EXP-1) coincides with the first set of experiments conducted in~\cite{breit2024means} by Breit, Moyo, Prohl, and the last author for investigating the effect of different noises on the stochastic Navier--Stokes equations; (EXP-2) is the classical lid-driven cavity experiment, which is frequently used in fluid dynamics; see, e.g.,~\cite{kuhlmann2019lid} and the references therein.  

Our guiding research questions are: 
\begin{enumerate}
    \item[(Q1)] \label{it:q1} Do we observe convergence of the velocity distribution towards the invariant measure?
    \item[(Q2)] \label{it:q2} If so, how does the invariant measure depend on the growth rate of the viscous stress?
\end{enumerate}
We address these questions by monitoring different solution statistics, such as the system's kinetic energy and the distribution of velocity at a fixed spatial location. If the stochastic system reaches its stationary state, then these statistics necessarily need to be invariant with respect to time. Therefore, if we numerically observe stationarity, then this provides us with numerical evidence for having reached the stationary state; but we cannot infer it with certainty, since we don't access the full velocity distribution. On the other hand, the converse is certain: if we don't observe stationarity, then the system hasn't reached its invariant state yet.  

The implementation of the algorithm as well as the code used for conducting the experiments are available at~\href{https://github.com/joernwichmann/gen-Stokes}{https://github.com/joernwichmann/gen-Stokes}; they use the open-source finite element package \textit{Firedrake}~\cite{FiredrakeUserManual}, which itself heavily relies on \textit{PETSc}~\cite{petsc-web-page}. 

The presentation of this section follows~\cite[Section~7]{Le2024Spacetime}; see also~\cite[Section~4]{breit2024means}. All experiments are conducted on the $2$-dimensional unit square.

\subsection{Choosing the GD}
Even though the GDM encapsulates many discretisation methods in a unified analytic toolbox, for the numerical simulations we need to pick one particular scheme that fits into this framework. As the focus of this work is really on the long-term behaviour of the model, we choose to illustrate that behaviour using a standard method, the Taylor--Hood finite element -- a stable mixed finite element pair; see, e.g.,~\cite{Taylor1973} -- generated by a fixed uniform triangulation of the unit square with~$13\times 13$ vertices; see \cite[Figure~3]{breit2024means}. The pair consists of continuous velocity and continuous pressure approximate spaces, generated by local polynomials of degree $2$ and $1$, respectively. 
In this case, the function reconstruction operators~$\Pi_\disc$ and~~$\chi_\disc$ map the vector of degrees of freedom (with~$\mathrm{DOF}$ components) to the pairing with respect to a fixed basis~$(\phi_j)_{j=1}^{\mathrm{DOF}}$ of their corresponding finite element space; for example, the pairing for velocity is given by: 
\begin{align*}
    X_{\disc,0} \ni u_\disc \mapsto \Pi_{\disc} u_\disc (x) := \sum_{j=1}^{\mathrm{DOF}} u_\disc^{j} \phi_j(x)  \subset W^{1,\infty}_0(\mathcal{O}).
\end{align*}
Due to conformity of the Taylor--Hood element, the reconstructions of differential operators coincide with their continuous counterparts: i.e., $\nabla_\disc = \nabla \Pi_\disc$, $\varepsilon_\disc = \varepsilon \Pi_\disc$, and $\Div_\disc = \Div \Pi_\disc$.  

\subsection{Sampling strategy}
To access the Wiener increments, we employ the Monte-Carlo method. Let $L\in \mathbb{N}$ be the sample size. 
We substitute realizations of the random vector $(\Delta_m W(\omega_\ell) )_{m\in \mathbb{N}}$, $\ell =1,\ldots,L $, by:
\begin{align*} 
Z_\ell^m \approx \Delta_m W(\omega_\ell), \qquad m \in \mathbb{N},\quad  \ell =1, \ldots, L,
\end{align*}
which are independently generated by a pseudo-random number generator.

\subsection{Implemented algorithm}
We implement a slightly modified version of our algorithm, which additionally accounts for an external, time-independent, and deterministic force~$F$ and non-solenoidal boundary conditions. 

\underline{Step 1:~Initialisation.} For all $\ell =1,\ldots,L$, define $\big(u_{\disc,\ell}^{0}, \pi_{\disc,\ell}^{0} \big)$ by solving~\eqref{eq:disc-Helmholtz};

\underline{Step 2:~Time-stepping.} For $n \in \mathbb{N}_0$ and  $\ell =1,\ldots,L$, define $(v_{\disc,\ell}^{n+1}, \pi_{\disc,\ell}^{n+1})$ by solving: for all $(\xi,q) \in X_{\disc,0} \times Y_{\disc,0}$,
\begin{subequations} \label{eq:time-stepping-implemented}
\begin{align} \label{eq:Evolution-01-imp}
&\begin{aligned}
&\left( \Pi_\disc v^{n+1}_{\disc,\ell} -\Pi_\disc v^n_{\disc,\ell}, \Pi_\disc \xi \right)  - \left( \chi_\disc \pi^{n+1}_{\disc,\ell} - \chi_\disc \pi^n_{\disc,\ell}, \Div_\disc \xi \right) \\ 
&\qquad \qquad + \tau \left( S( \varepsilon_\disc v^{n+1/2}_{\disc,\ell} + \varepsilon g ), \varepsilon_\disc \xi \right)  \\
&\qquad= \tau \left( F, \Pi_\disc \xi \right) +  \left[ B_\disc^\sigma( v^{n+1/2}_{\disc,\ell}, \xi)+ \left( (\sigma \cdot \nabla)g , \Pi_{\disc} \xi \right) \right] Z_\ell^{n+1} ,
\end{aligned}\\ \label{eq:Evolution-02-imp}
 & \left(  \Div_\disc v^{n+1/2}_{\disc,\ell},\chi_\disc q \right) = \left(  \Div_\disc g,\chi_\disc q \right) ,
\end{align}
\end{subequations}
where $v^{n+1/2}_{\disc,\ell}:= \frac{v^{n+1}_{\disc,\ell} + v^{n}_{\disc,\ell}}{2}$.

For actually solving the non-linear system of equations, we rely on the \textit{firedrake.solve} function that internally uses the nonlinear solver~\textit{PETSc.SNES} with $10^{-8}$ as absolute and relative tolerance parameters, and the parallel sparse direct solver~\textit{MUMPS}.

To compare the stochastic dynamics with the deterministic one, we always simulate both cases:~$\sigma \neq 0$ and $\sigma = 0$. In all experiments, we fix the following parameters:
\begin{itemize}
\item (viscous stress) $\kappa = 0.1$ and $p  \in \{ 1.5, 2, 3\}$;
\item (time step size) $\tau \in \{ 2^{-7},2^{-8},2^{-9}\}$;
\item (sample size) $L = 1000$;
\item (time horizon) $T = 1$.
\end{itemize}

\subsection{EXP-1: Trivial boundary conditions and deterministic forcing}
As already noted, this experiment enables the comparison of our algorithm to the one presented in~\cite{breit2024means}. The external, deterministic force causes the invariant state to become non-trivial, even for trivial boundary conditions. 

We choose the following data: 
    \begin{align*}
    v^\mathrm{in}(x,y) &= 10^3 \begin{pmatrix}
x^2(1-x)^2(2-6y+4y^2)y \\
-y^2(1-y)^2(2-6x+4x^2)x
\end{pmatrix}, \\
    \sigma(x,y) &= 10^3 \begin{pmatrix}
x^2(1-x)^2(2-6y+4y^2)y \\
-y^2(1-y)^2(2-6x+4x^2)x
\end{pmatrix}, \\
g(x,y) &= \begin{pmatrix}
0 \\
0
\end{pmatrix}, \\
F(x,y) &= 10^2 \begin{pmatrix}
\sin(2 \pi x)\sin(4 \pi y) \\
-\sin(4\pi x) \sin(2 \pi y)
\end{pmatrix}.
\end{align*}

\subsection{EXP-2: Lid-driven cavity}
This experiment repeats the commonly used lid-driven cavity experiment in our particular configuration; that is, for the generalised Stokes equations forced by transport noise as well as our newly proposed algorithm. 

The lid-driven cavity experiment considers a resting fluid in a box, which is influenced by a non-trivial boundary condition acting solely at the container's lid; the boundary conditions push the fluid at a constant rate into one direction. Initially, the fluid starts at rest. Eventually, induced by the boundary conditions, the fluid displays versatile dynamics in the whole container; see, e.g.,~\cite{Zhu2020}.

Resolving the full range of dynamics in numerical simulations is a challenging task. Motivated by this, many authors used the lid-driven cavity experiment as a benchmark for numerical algorithms; see, e.g.,~\cite{Botti2019,Erturk2005,Zhu2020,kuhlmann2019lid} and the references therein. Unfortunately, numerical simulations for stochastic models of the lid-driven cavity experiment are still largely missing. In particular, no simulations have been conducted for stochastic fluid equations that model turbulence. We could only find results addressing parameter uncertainties; see, e.g.,~\cite{Wang2010}, which ambiguously are also called ``stochastic cavity flow''. However, the modeling of parameter uncertainty and turbulence are completely different; our simulations contribute to the latter. 

We choose the following data: 
\begin{align*}
    v^\mathrm{in}(x,y) &= \begin{pmatrix}
0 \\
0
\end{pmatrix}, \\
    \sigma(x,y) &= 10^3 \begin{pmatrix}
x^2(1-x)^2(2-6y+4y^2)y \\
-y^2(1-y)^2(2-6x+4x^2)x
\end{pmatrix}, \\
g(x,y) &= \begin{pmatrix}
1 \\
0
\end{pmatrix} \mathbf{1}_{\{y=1\}}(x,y), \\
F(x,y) &= \begin{pmatrix}
0 \\
0
\end{pmatrix},
\end{align*}
where $\mathbf{1}_{\{y=1\}}$ denotes the indicator function on $\{y=1\}$.

\subsection{Results}

To answer Question~\ref{it:q1}, all numerical experiments indicate that our algorithm gives rise to an invariant measure. They even support the hypothesis that the distribution of velocity converges without considering time-averages; that is,  
\begin{align*}
     \mathbb{P}\left(\Pi_\disc \mathfrak{S}(N,v^\mathrm{in}) \in A \right) \overset{N\to \infty}{\rightarrow} \mu_{\tau,\disc}(A).
\end{align*}
A theoretical explanation for this still needs to be researched.

Next, we address Question~\ref{it:q2}. In all experiments, transport noise enhances the dissipation of kinetic energy; see Figures~\ref{fig:varkin-traj} and~\ref{fig:lidkin-traj}. Depending on the growth rate of the viscous stress, the difference between the stochastic and deterministic energy evolutions can be several orders of magnitude. 

In EXP-1, transport noise causes improved mixing; see Figures~\ref{fig:varStream}: the streamline plots are generated by tracing the same~$400$ particles for each experiment. The particles are initially positioned at four lines: both diagonals, a horizontal line and a vertical line both passing through $(0.5,0.75)$. The randomness enables the particles to reach all locations, at least for shear-thinning fluids ($p=1.5$). For the deterministic dynamics, the shape of the streamlines is mostly unaffected by changing the growth rate; the change influences the velocity magnitude of the particles. The stochastic case is different. Here, the change influences the shape of the streamlines, while only changing the velocity magnitude marginally. This is also indicated by the statistics of the velocity vector at the fixed spatial location~$(0.5,0.75)$; see Figure~\ref{fig:varpoint-hist}. The deterministic velocity vector doesn't change its direction. In contrast, the distribution of the stochastic velocity vector is supported in a neighbourhood of the origin; in particular, the stochastic velocity vector points in all directions with positive probability. The intensity of the random effects increases as the growth rate decreases, which can be seen by, e.g., the magnitudes of the standard deviations and the difference in the energies; see Figures~\ref{fig:varDev} and~\ref{fig:varkin-traj}, respectively. Regions where the velocity trajectories deviate the most from its mean velocity are shown in Figure~\ref{fig:varDev}.

\begin{figure}
    \centering
    \begin{subfigure}{1\textwidth}
    \includegraphics[width=0.5\linewidth]{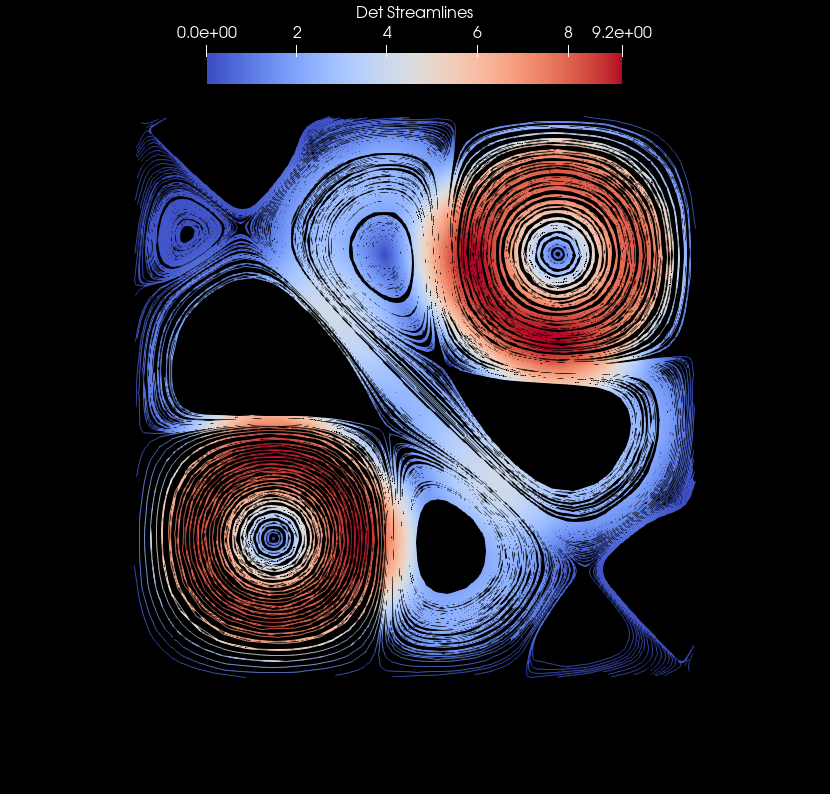}%
    \includegraphics[width=0.5\linewidth]{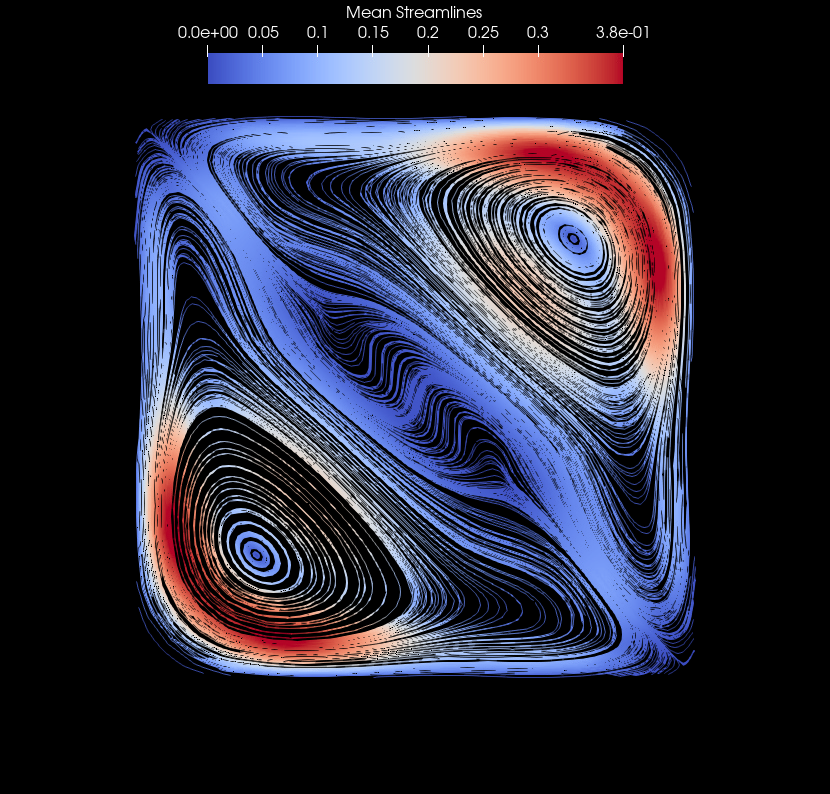}
    \caption{Streamlines for EXP-1 with viscous growth rate $p=1.5$.}
    \label{fig:varStream-smallP}
    \end{subfigure}
    \begin{subfigure}{1\textwidth}
    \includegraphics[width=0.5\linewidth]{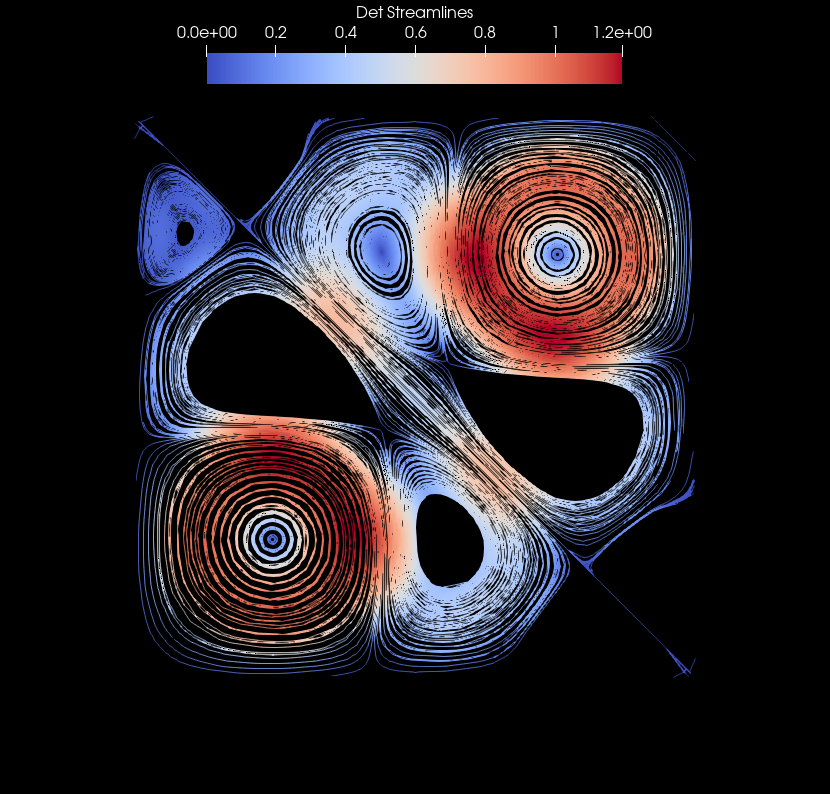}%
    \includegraphics[width=0.5\linewidth]{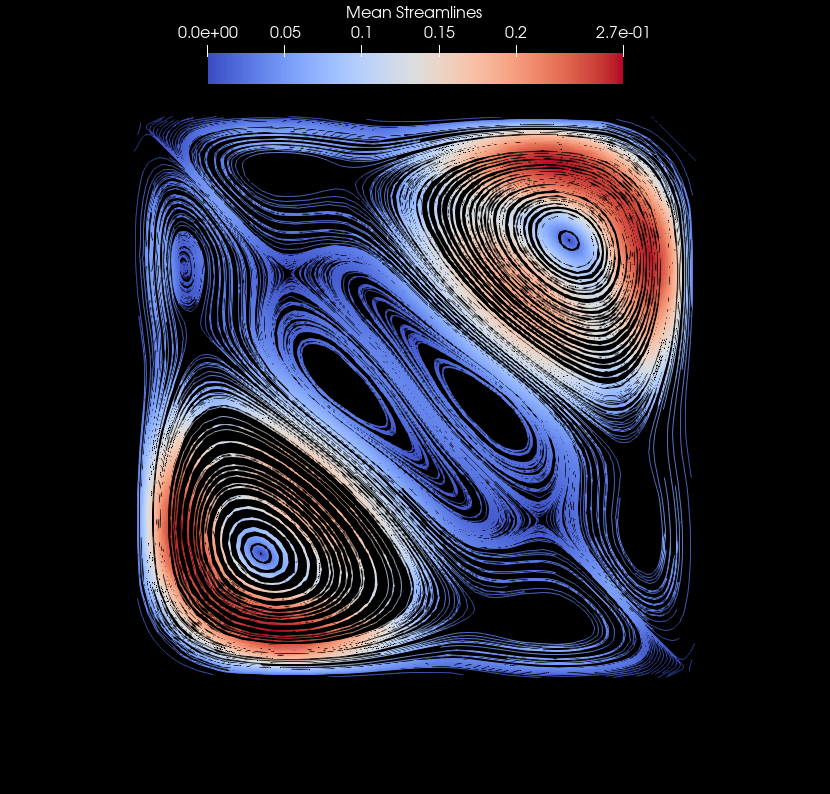}
    \caption{Streamlines for EXP-1 with viscous growth rate $p=2$.}
    \label{fig:varStream-middleP}
    \end{subfigure}
    \begin{subfigure}{1\textwidth}
    \includegraphics[width=0.5\linewidth]{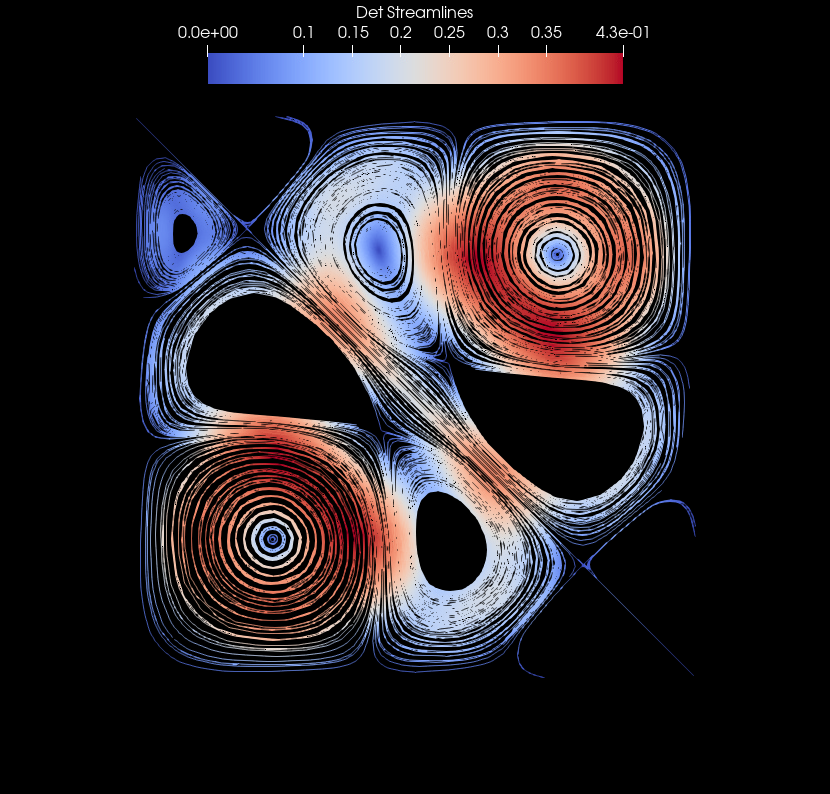}%
    \includegraphics[width=0.5\linewidth]{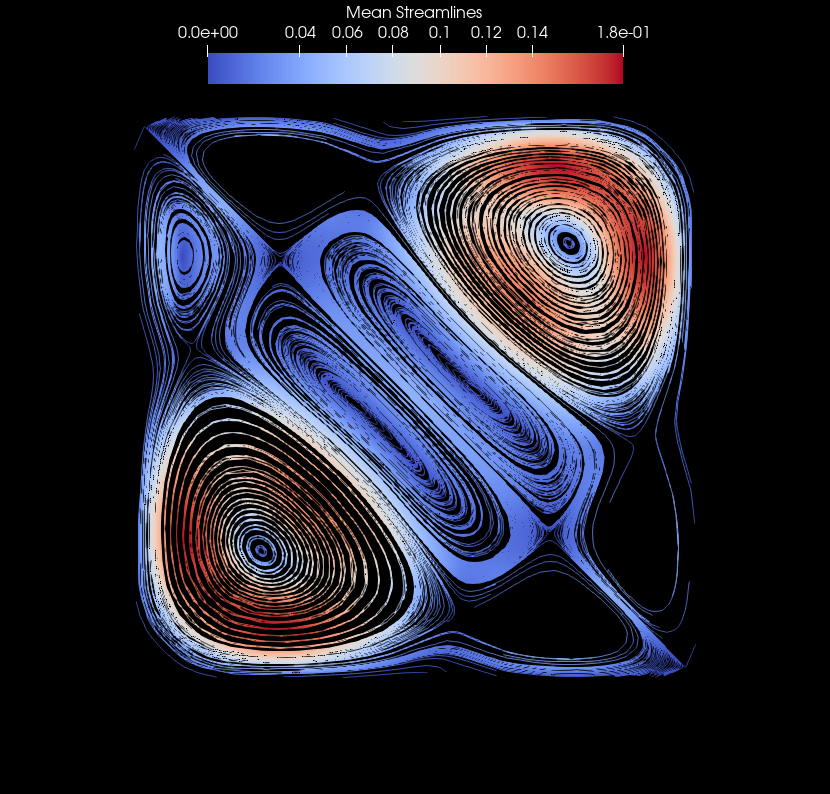}
    \caption{Streamlines for EXP-1 with viscous growth rate $p=3$.}
    \label{fig:varStream-bigP}
    \end{subfigure}
    \caption{Streamlines for EXP-1 with varying viscous growth rates; left: streamlines of deterministic dynamics; right: mean (based on 1,000 trajectories) streamlines of stochastic dynamics. Colour encodes the velocity magnitude; its scaling changes from figure to figure.}
    \label{fig:varStream}
\end{figure}

\begin{figure}
    \centering
    \begin{subfigure}{1\textwidth}
    \includegraphics[width=0.5\linewidth]{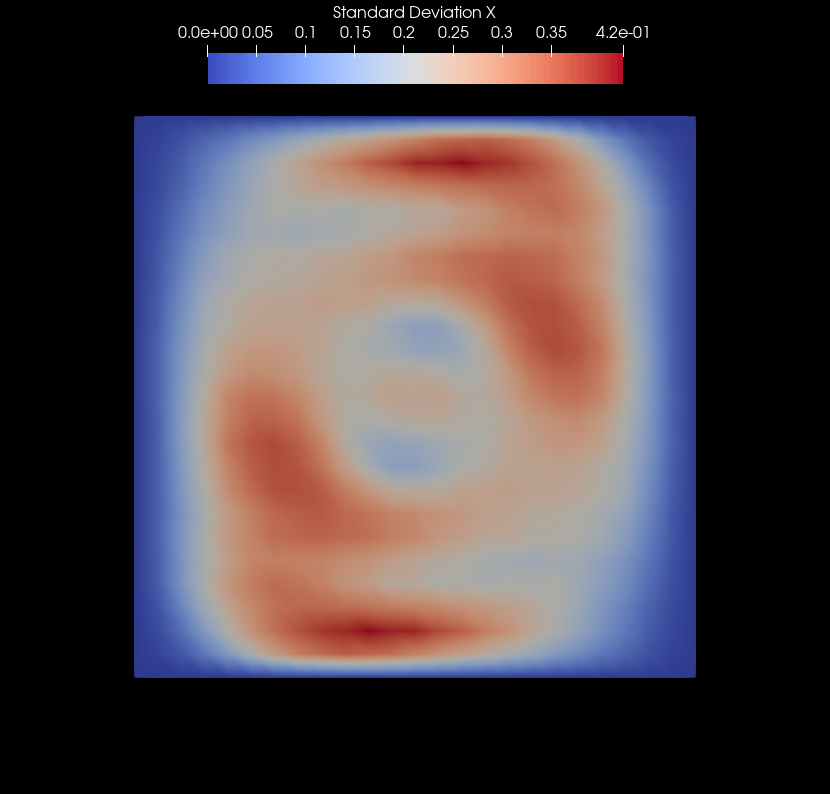}%
    \includegraphics[width=0.5\linewidth]{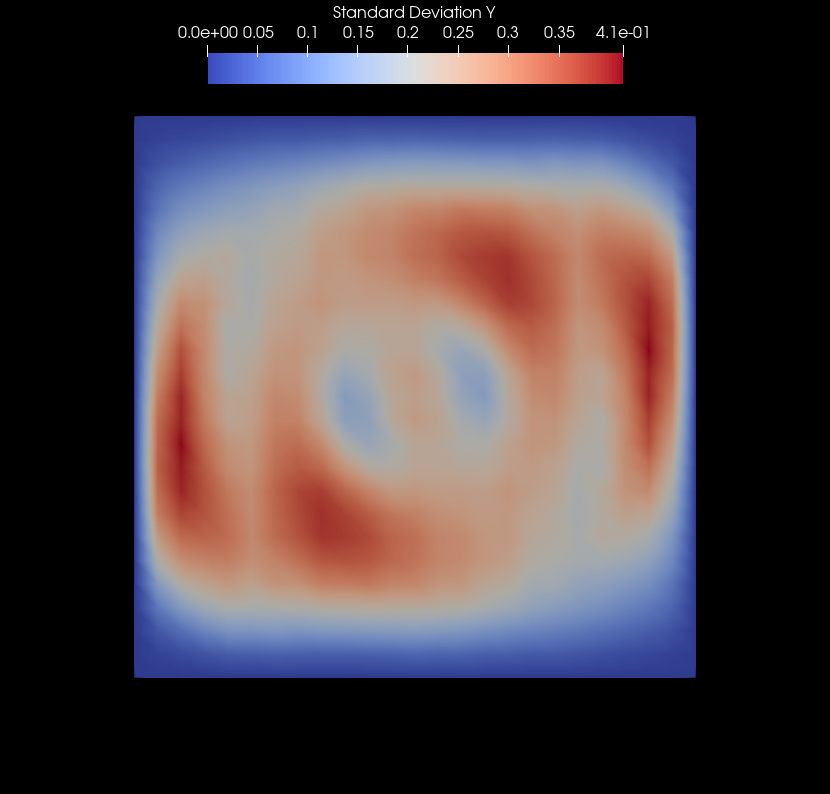}
    \caption{SD of velocity for EXP-1 with viscous growth rate $p=1.5$.}
    \label{fig:varDev-smallP}
    \end{subfigure}
    \begin{subfigure}{1\textwidth}
    \includegraphics[width=0.5\linewidth]{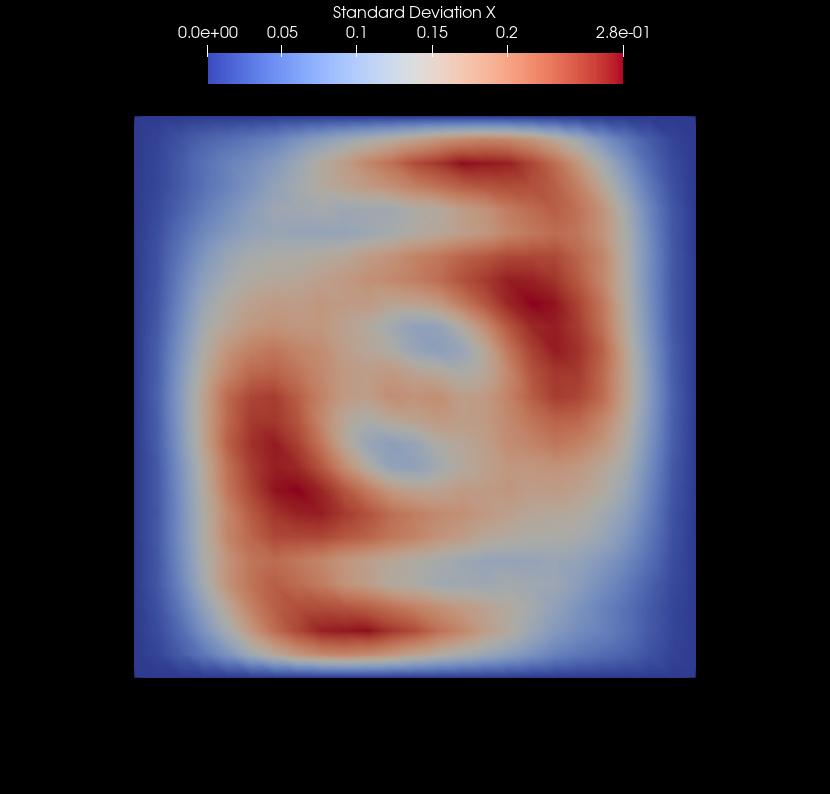}%
    \includegraphics[width=0.5\linewidth]{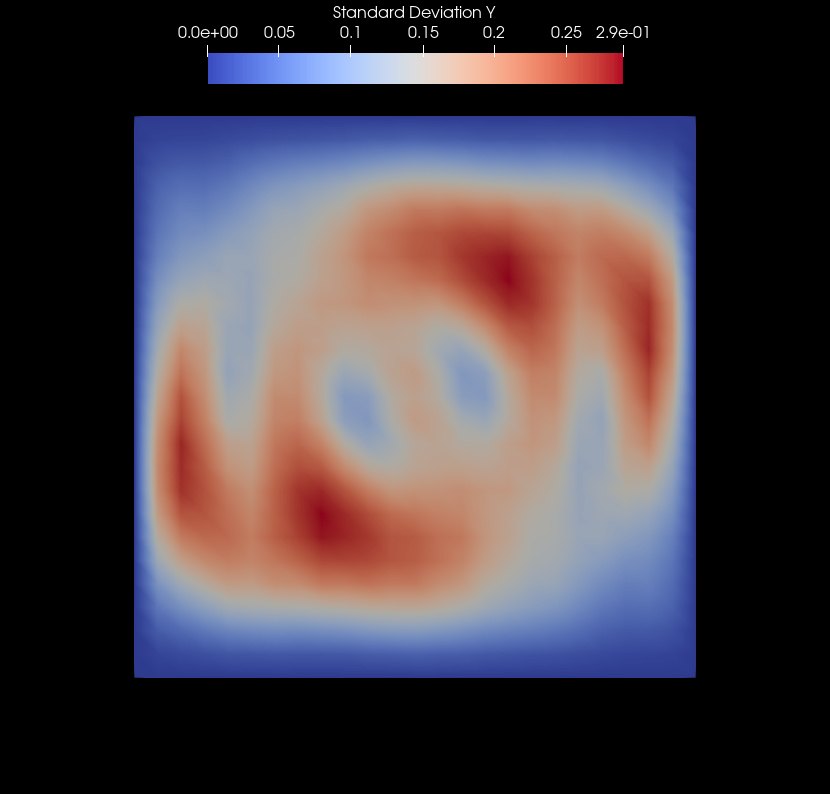}
    \caption{SD of velocity for EXP-1 with viscous growth rate $p=2$.}
    \label{fig:varDev-middleP}
    \end{subfigure}
    \begin{subfigure}{1\textwidth}
    \includegraphics[width=0.5\linewidth]{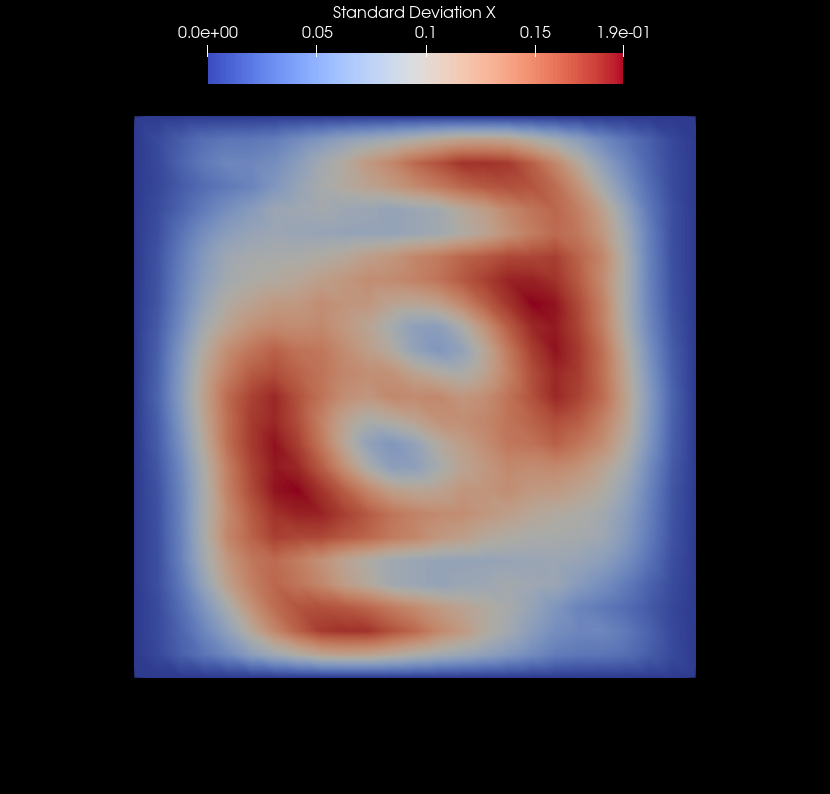}%
    \includegraphics[width=0.5\linewidth]{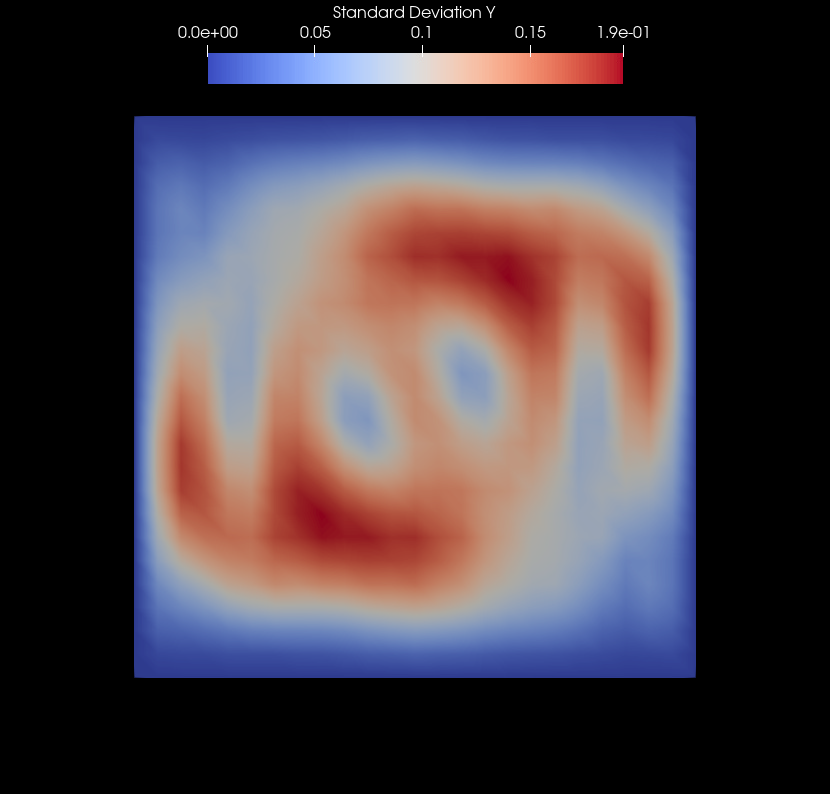}
    \caption{SD of velocity for EXP-1 with viscous growth rate $p=3$.}
    \label{fig:varDev-bigP}
    \end{subfigure}
    \caption{Standard deviation (SD) of velocity for EXP-1 with varying viscous growth rates; left: $x$-component; right: $y$-component. Colour encodes the magnitude; its scaling changes from figure to figure.}
    \label{fig:varDev}
\end{figure}

\begin{figure}
    \centering
    \begin{subfigure}{1\textwidth}
    \includegraphics[width=1\linewidth]{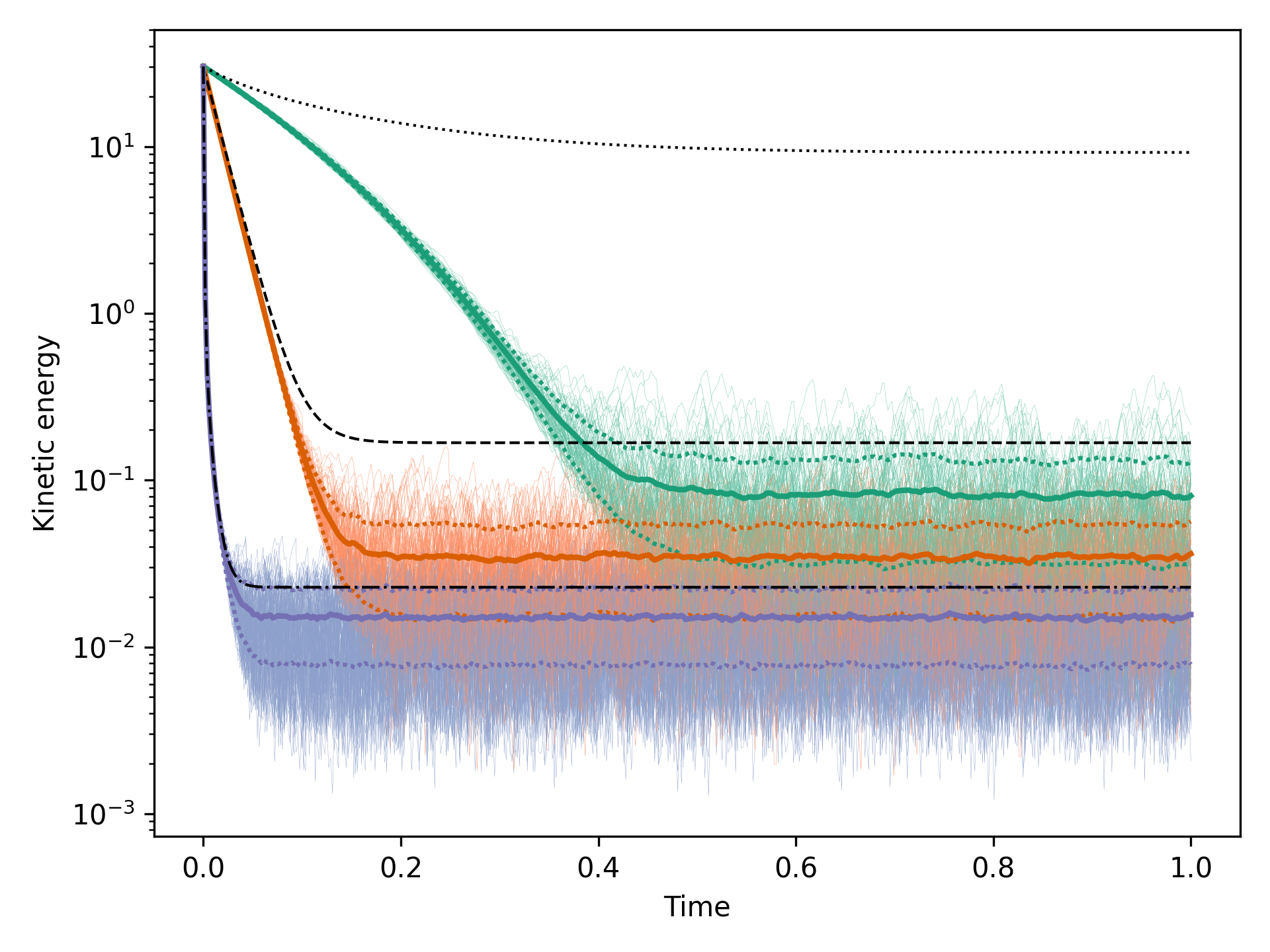}
    \vspace{-2em}
    \caption{Time evolution of kinetic energy: $\tau n \mapsto \tfrac{1}{2}\norm{\Pi_\disc v_\disc^n}_{L^2(\mathcal{O})}^2$, for EXP-1 with viscous growth rates: $p=1.5$~(sto:~{\protect\tikz \protect\draw[color=intro_color1, line width=2] (0,0) -- (0.5,0);}; det:~{\protect\tikz \protect\draw[color=intro_color4, line width=1,dotted] (0,0) -- (0.5,0);}), $p=2$~(sto:~{\protect\tikz \protect\draw[color=intro_color2, line width=2] (0,0) -- (0.5,0);}; det:~{\protect\tikz \protect\draw[color=intro_color4, line width=1,dashed] (0,0) -- (0.5,0);}), and $p=3$~(sto:~{\protect\tikz \protect\draw[color=intro_color3, line width=2] (0,0) -- (0.5,0);}; det:~{\protect\tikz \protect\draw[color=intro_color4, line width=1,dashdotted] (0,0) -- (0.5,0);}). Thick lines and dotted lines show the mean energy and the mean energy plus or minus one standard deviation, respectively. The first 100 (out of 1,000) energy trajectories are shown in pale colours.}
    \label{fig:varkin-traj}
    \end{subfigure}
    \begin{subfigure}{1\textwidth}
    \includegraphics[width=1\linewidth]{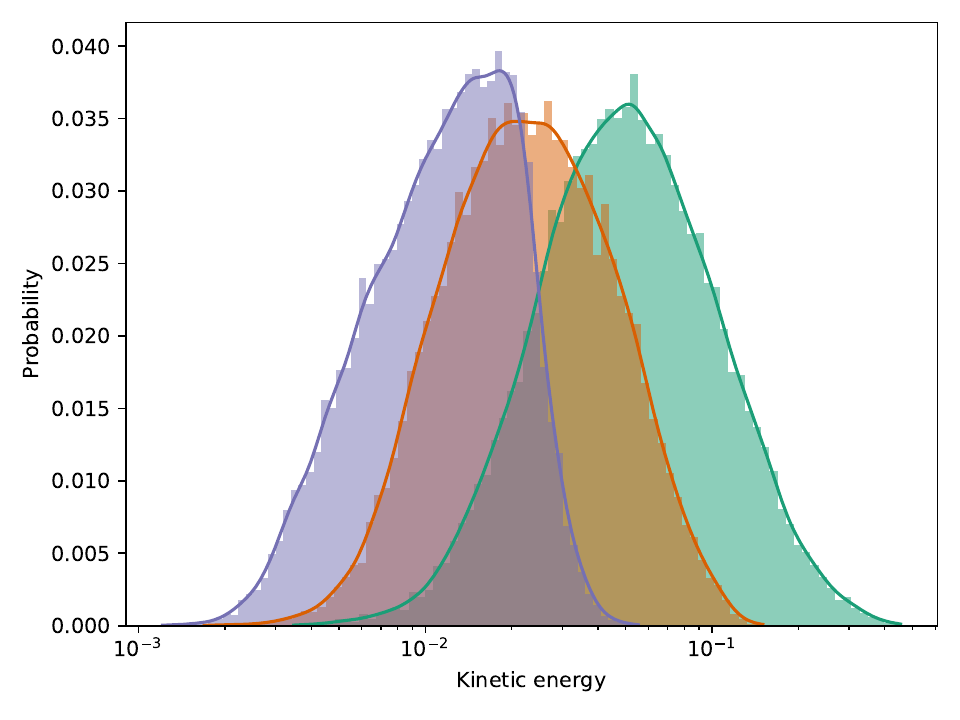}
    \vspace{-2em}
    \caption{Empirical approximation (based on 100 trajectories) of the stationary distributions of the kinetic energy for EXP-1 with viscous growth rates: $p=1.5$~({\protect\tikz \protect\draw[color=intro_color1, line width=2] (0,0) -- (0.5,0);}), $p=2$~({\protect\tikz \protect\draw[color=intro_color2, line width=2] (0,0) -- (0.5,0);}), and $p=3$~({\protect\tikz \protect\draw[color=intro_color3, line width=2] (0,0) -- (0.5,0);}).}
    \label{fig:varkin-hist}
    \end{subfigure}
    \vspace{-2em}
    \caption{Statistics of kinetic energy for EXP-1.}
    \label{fig:varkin-total}
\end{figure}

\begin{figure}
    \centering
    \begin{subfigure}{1\textwidth}
    \centering
    \includegraphics[scale=0.625]{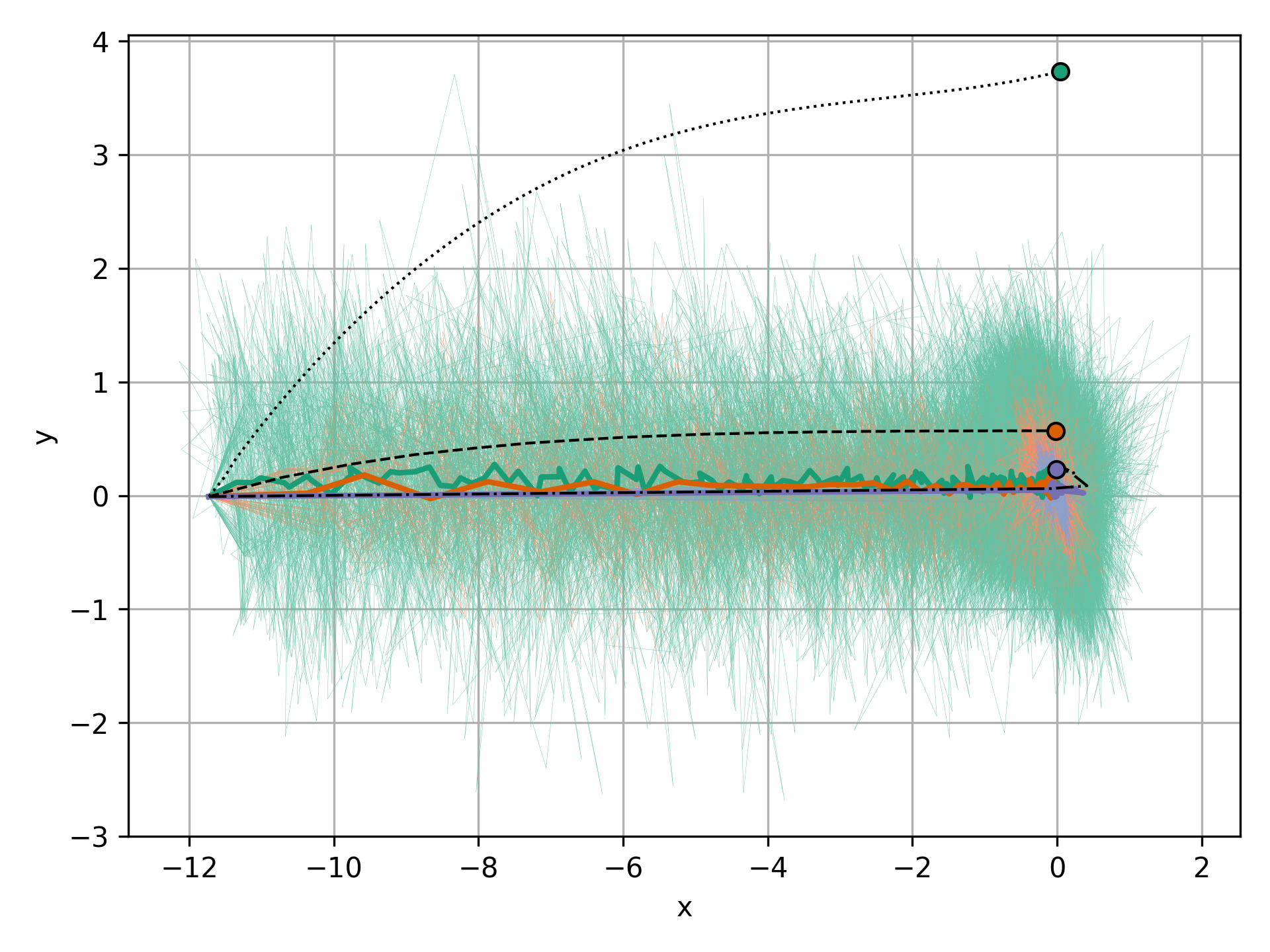}
    \vspace{-1em}
    \caption{Time evolution of the velocity vector at the fixed spatial location $(x^*,y^*) = (0.5,0.75)$ with viscous growth rates: $p=1.5$~(sto:~{\protect\tikz \protect\draw[color=intro_color1, line width=2] (0,0) -- (0.5,0);}; det:~{\protect\tikz \protect\draw[color=intro_color4, line width=1,dotted] (0,0) -- (0.5,0);}), $p=2$~(sto:~{\protect\tikz \protect\draw[color=intro_color2, line width=2] (0,0) -- (0.5,0);}; det:~{\protect\tikz \protect\draw[color=intro_color4, line width=1,dashed] (0,0) -- (0.5,0);}), and $p=3$~(sto:~{\protect\tikz \protect\draw[color=intro_color3, line width=2] (0,0) -- (0.5,0);}; det:~{\protect\tikz \protect\draw[color=intro_color4, line width=1,dashdotted] (0,0) -- (0.5,0);}). Thick lines show the mean velocity vector. Individual trajectories are shown in pale colours. Deterministic stationary velocity vectors are denoted by: {\protect\tikz \protect\draw[color=intro_color4,fill=intro_color1] (0,0) circle (.5ex);}, {\protect\tikz \protect\draw[color=intro_color4,fill=intro_color2] (0,0) circle (.5ex);}, and {\protect\tikz \protect\draw[color=intro_color4,fill=intro_color3] (0,0) circle (.5ex);} for $p=1.5$, $p=2$, and $p=3$, respectively.}
    \label{fig:varpoint-traj}
    \end{subfigure}
    \begin{subfigure}{1\textwidth}
    \centering
    \includegraphics[scale=0.625]{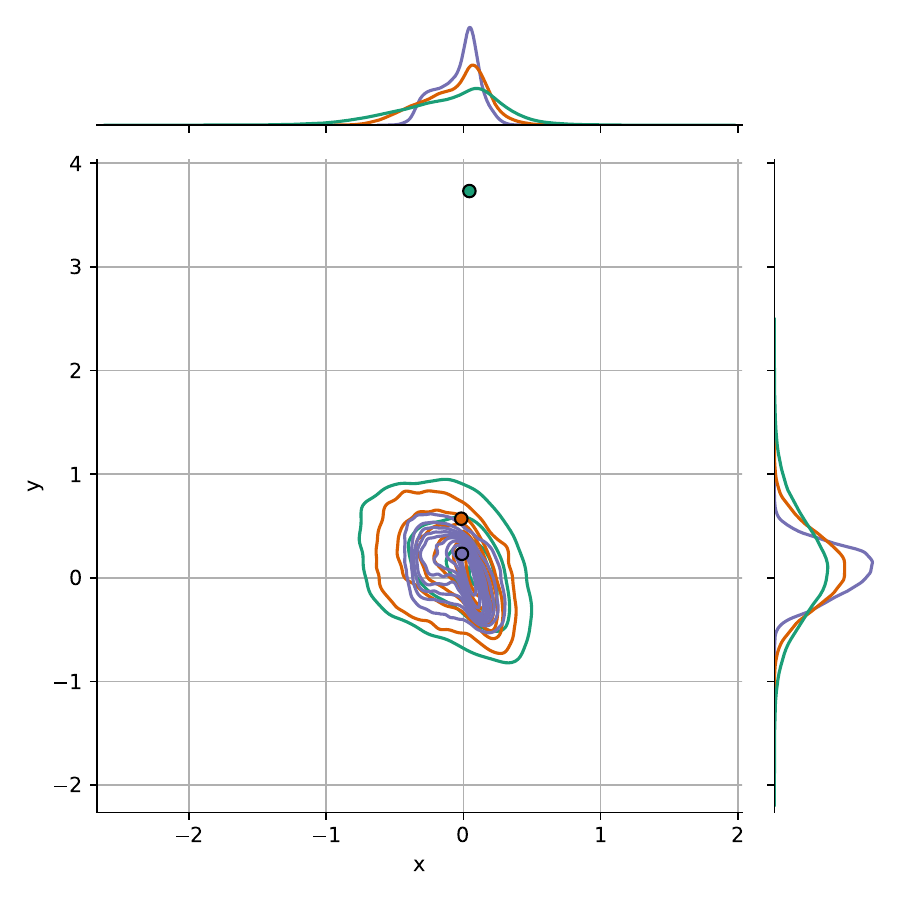}
    \vspace{-1em}
    \caption{Empirical approximation (based on 100 trajectories) of the stationary distributions of the velocity vector at the fixed spatial location $(x^*,y^*) = (0.5,0.75)$ with viscous growth rates: $p=1.5$~({\protect\tikz \protect\draw[color=intro_color1, line width=2] (0,0) -- (0.5,0);}), $p=2$~({\protect\tikz \protect\draw[color=intro_color2, line width=2] (0,0) -- (0.5,0);}), and $p=3$~({\protect\tikz \protect\draw[color=intro_color3, line width=2] (0,0) -- (0.5,0);}). Deterministic stationary velocity vectors are denoted by: {\protect\tikz \protect\draw[color=intro_color4,fill=intro_color1] (0,0) circle (.5ex);}, {\protect\tikz \protect\draw[color=intro_color4,fill=intro_color2] (0,0) circle (.5ex);}, and {\protect\tikz \protect\draw[color=intro_color4,fill=intro_color3] (0,0) circle (.5ex);} for $p=1.5$, $p=2$, and $p=3$, respectively.}
    \label{fig:varpoint-hist}
    \end{subfigure}
    \vspace{-1em}
    \caption{Statistics of velocity vectors at fixed spatial location for EXP-1.}
    \label{fig:varpoint-total}
\end{figure}

In EXP-2, transport noise enlarges the corner vortices to the point where they resemble the corner vortices known for the deterministic Navier-Stokes equations at high Reynolds numbers ($\mathrm{Re}$); compare Figure~\ref{fig:lidStream} with~\cite[{Figure~2 where $\mathrm{Re}= 1,000$}]{Erturk2005}. This is remarkable since, even though the stochastic system is forced by only one randomly-scaled vortex (in classical turbulence models one would have infinitely many), the solutions' streamlines for the (linear if $p=2$) stochastic generalised Stokes and non-linear deterministic generalised Navier--Stokes systems are comparable. While the Stokes system also displays corner vortices (see, e.g.,~\cite[Figures~12]{Carstensen2012}), their size is substantially smaller. In our simulations, the spatial resolution doesn't resolve them. For both stochastic and deterministic dynamics, modification of the growth rate changes the location and shape of the central vortex; the origin of rotation moves farther away from the lid as the growth rate increases. Interestingly, the most uncertain regions of the stochastic velocity field are close to the lid of the container and the container's upper part of the lateral boundary; see Figure~\ref{fig:lidDev}. The intensity of the random effects decrease as the growth rate decreases, which can be seen by, e.g., the size of the supports of the kinetic energies' stationary distributions and the statistics of the velocity vector; see Figures~\ref{fig:lidkin-hist} and~\ref{fig:lidpoint-total}, respectively.

\begin{figure}
    \centering
    \begin{subfigure}{1\textwidth}
    \includegraphics[width=0.5\linewidth]{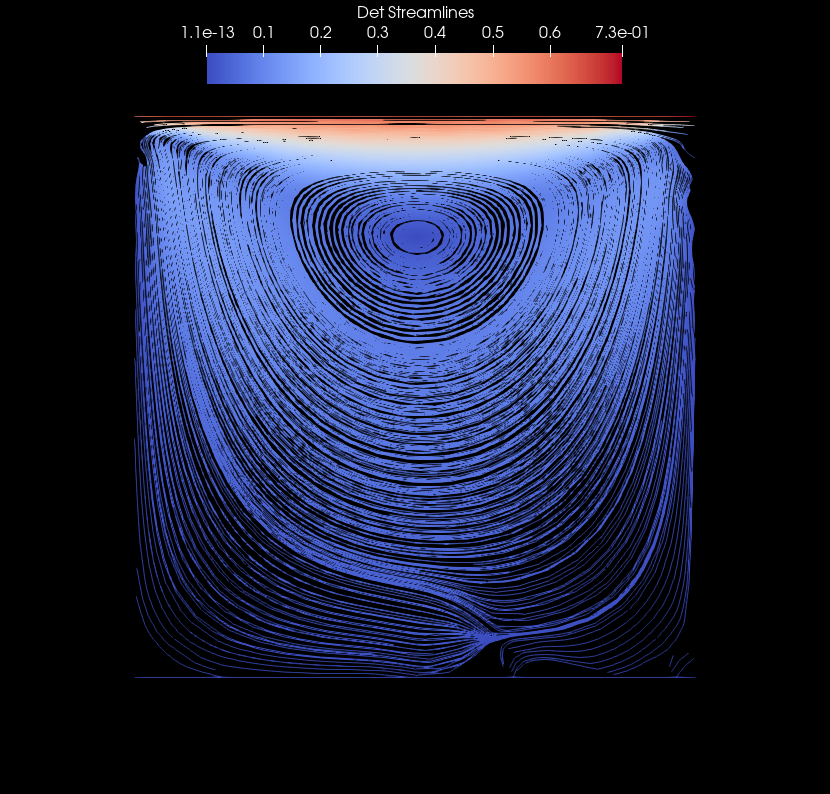}%
    \includegraphics[width=0.5\linewidth]{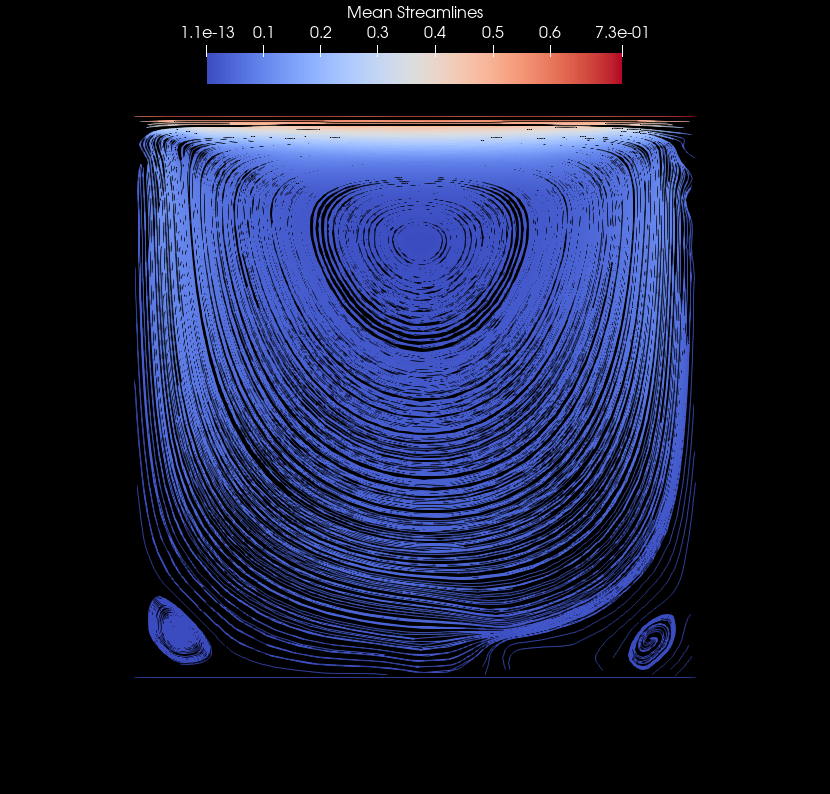}
    \caption{Streamlines for lid-driven cavity experiment with viscous growth rate $p=1.5$.}
    \label{fig:lidStream-smallP}
    \end{subfigure}
    \begin{subfigure}{1\textwidth}
    \includegraphics[width=0.5\linewidth]{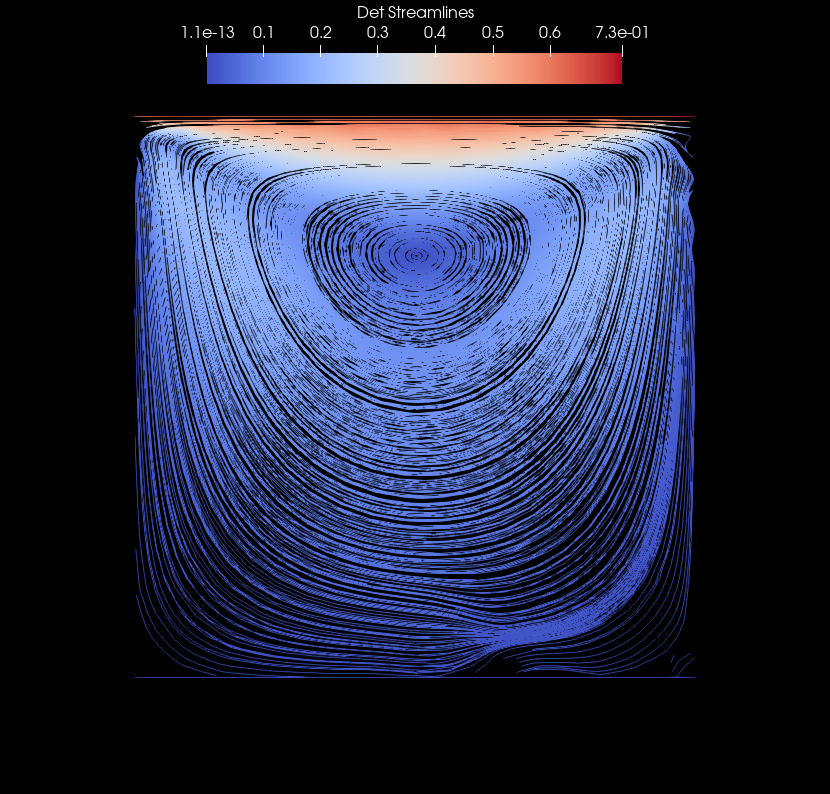}%
    \includegraphics[width=0.5\linewidth]{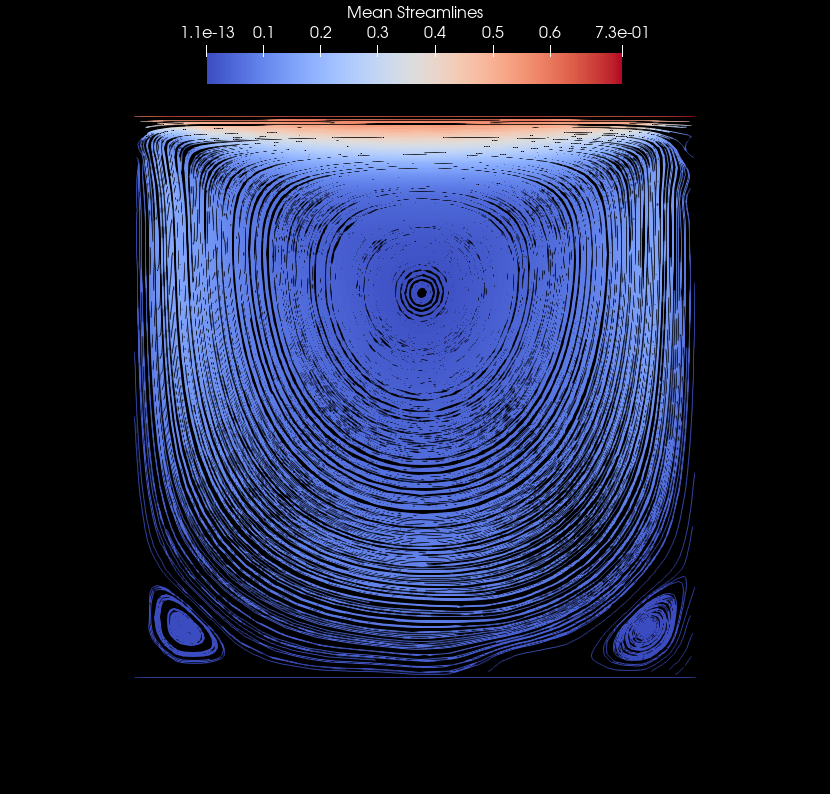}
    \caption{Streamlines for lid-driven cavity experiment with viscous growth rate $p=2$.}
    \label{fig:lidStream-middleP}
    \end{subfigure}
    \begin{subfigure}{1\textwidth}
    \includegraphics[width=0.5\linewidth]{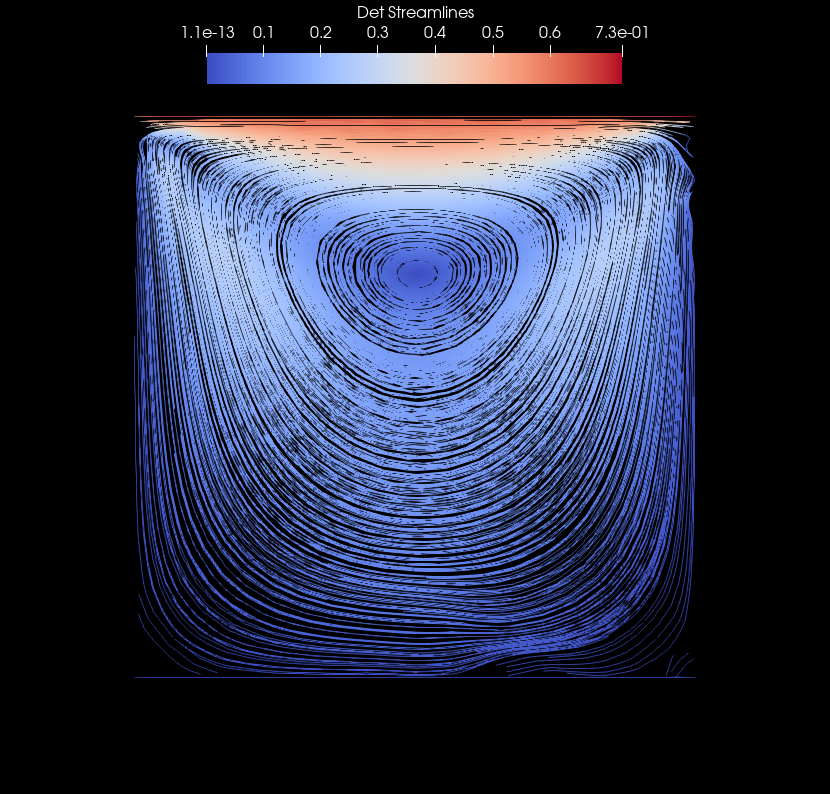}%
    \includegraphics[width=0.5\linewidth]{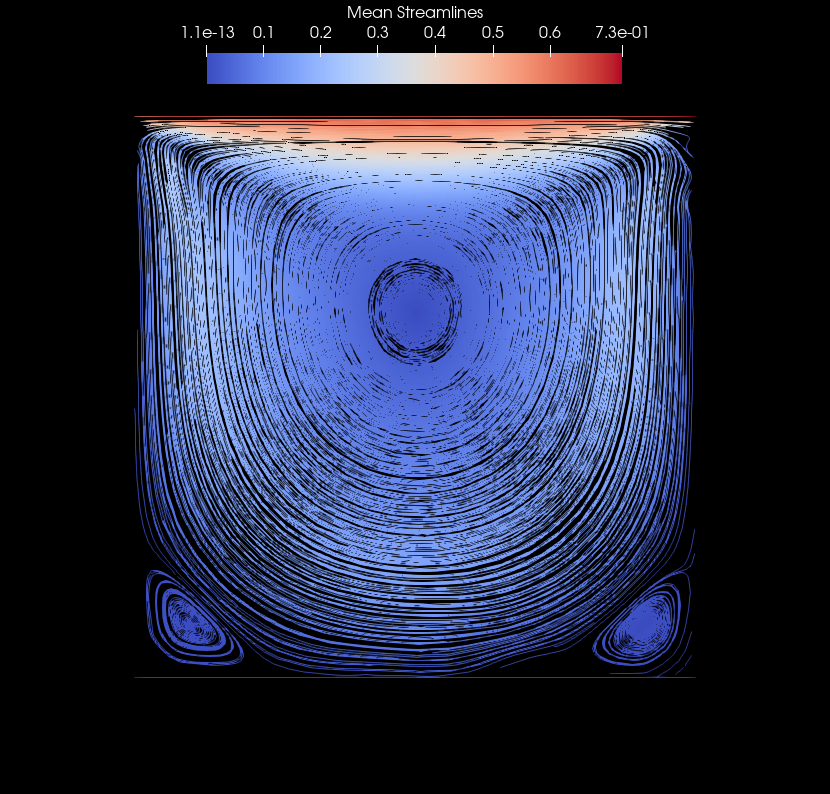}
    \caption{Streamlines for lid-driven cavity experiment with viscous growth rate $p=3$.}
    \label{fig:lidStream-bigP}
    \end{subfigure}
    \caption{Streamlines for lid-driven cavity experiments with varying viscous growth rates; left: streamlines of deterministic dynamics; right: mean (based on 1,000 trajectories) streamlines of stochastic dynamics. Colour encodes the velocity magnitude.}
    \label{fig:lidStream}
\end{figure}

\begin{figure}
    \centering
    \begin{subfigure}{1\textwidth}
    \includegraphics[width=0.5\linewidth]{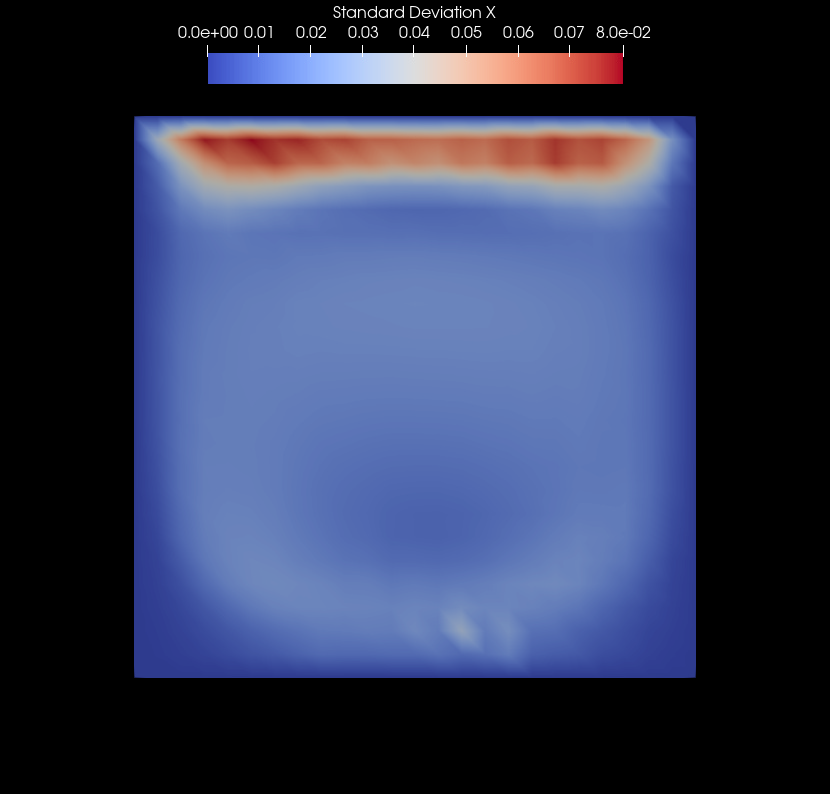}%
    \includegraphics[width=0.5\linewidth]{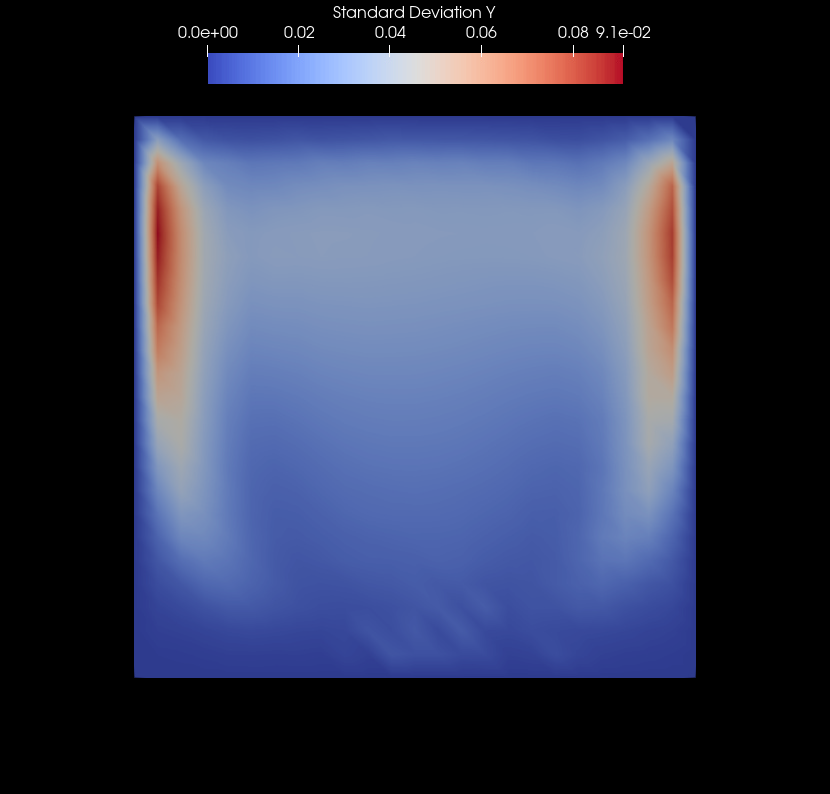}
    \caption{SD of velocity for lid-driven cavity experiment with viscous growth rate $p=1.5$.}
    \label{fig:lidDev-smallP}
    \end{subfigure}
    \begin{subfigure}{1\textwidth}
    \includegraphics[width=0.5\linewidth]{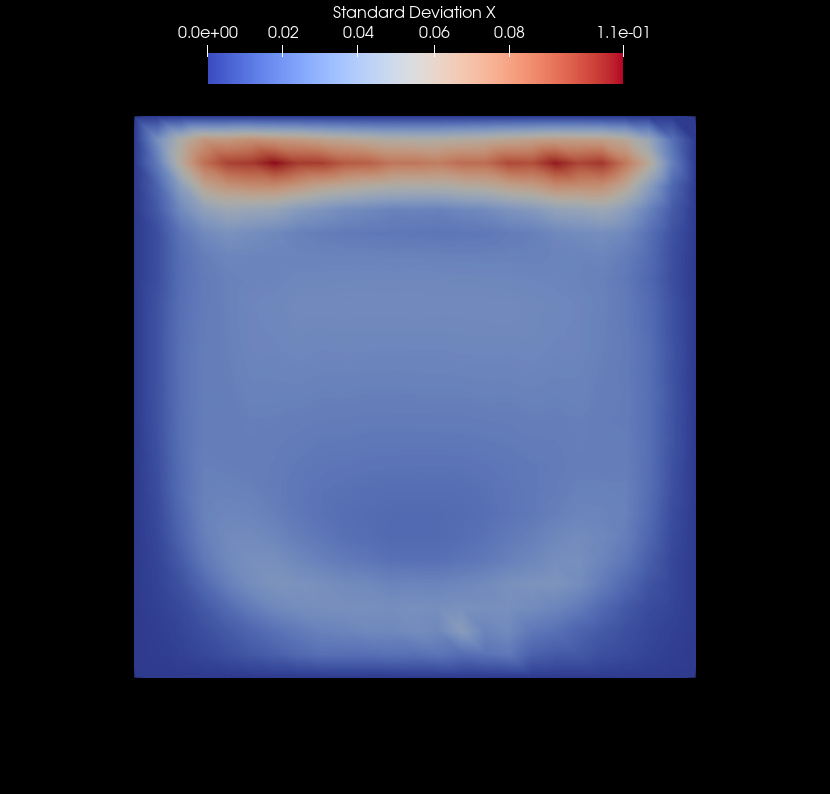}%
    \includegraphics[width=0.5\linewidth]{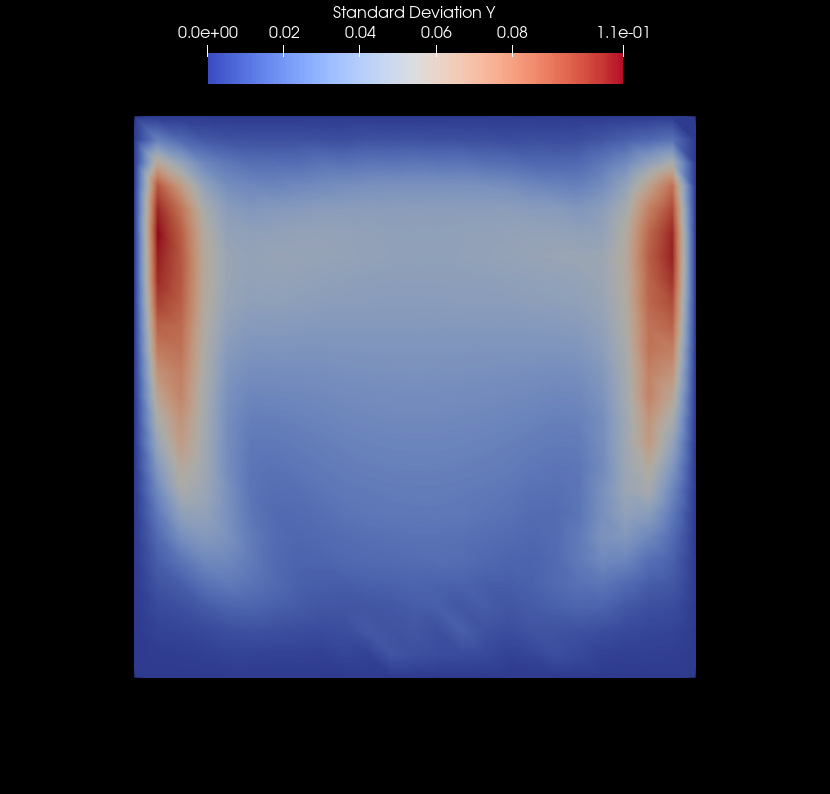}
    \caption{SD of velocity for lid-driven cavity experiment with viscous growth rate $p=2$.}
    \label{fig:lidDev-middleP}
    \end{subfigure}
    \begin{subfigure}{1\textwidth}
    \includegraphics[width=0.5\linewidth]{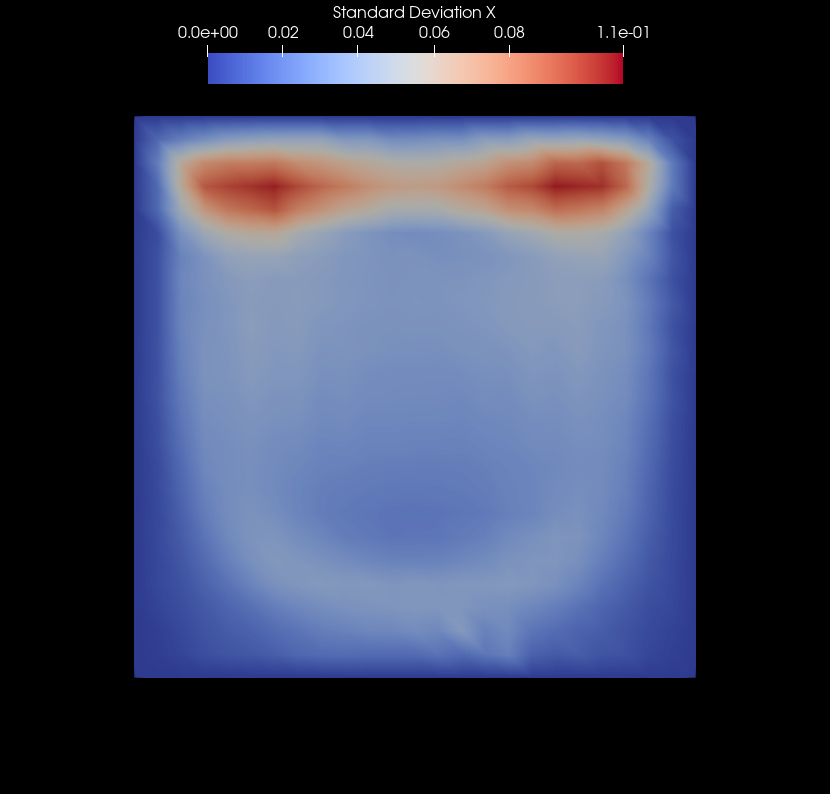}%
    \includegraphics[width=0.5\linewidth]{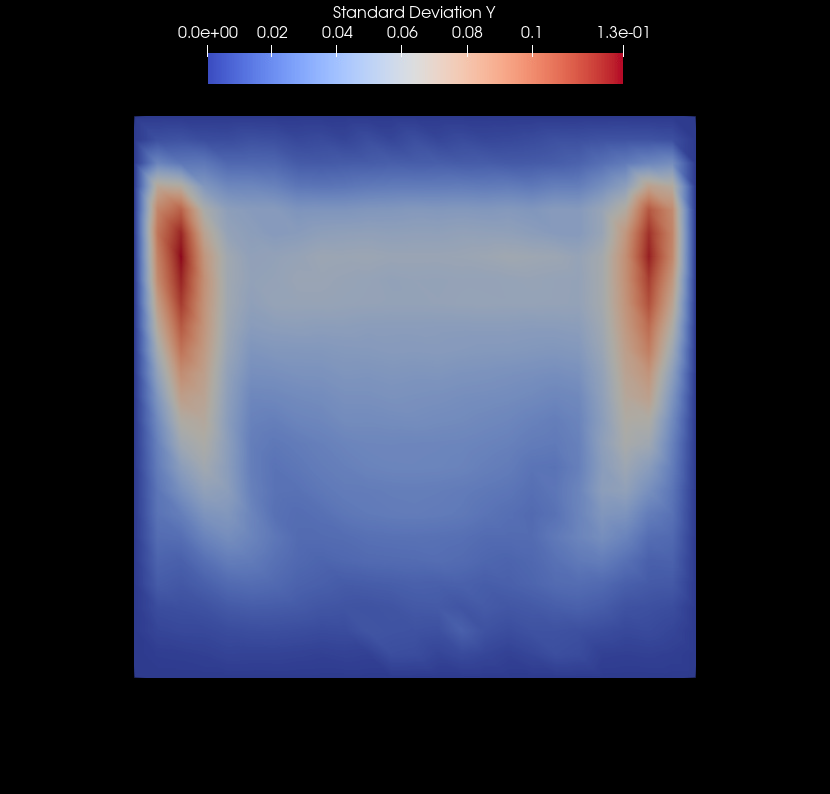}
    \caption{SD of velocity for lid-driven cavity experiment with viscous growth rate $p=3$.}
    \label{fig:lidDev-bigP}
    \end{subfigure}
    \caption{Standard deviation (SD) of velocity for lid-driven cavity experiments with varying viscous growth rates; left: $x$-component; right: $y$-component. Colour encodes the magnitude; its scaling changes from figure to figure.}
    \label{fig:lidDev}
\end{figure}

\begin{figure}
    \centering
    \begin{subfigure}{1\textwidth}
    \includegraphics[width=1\linewidth]{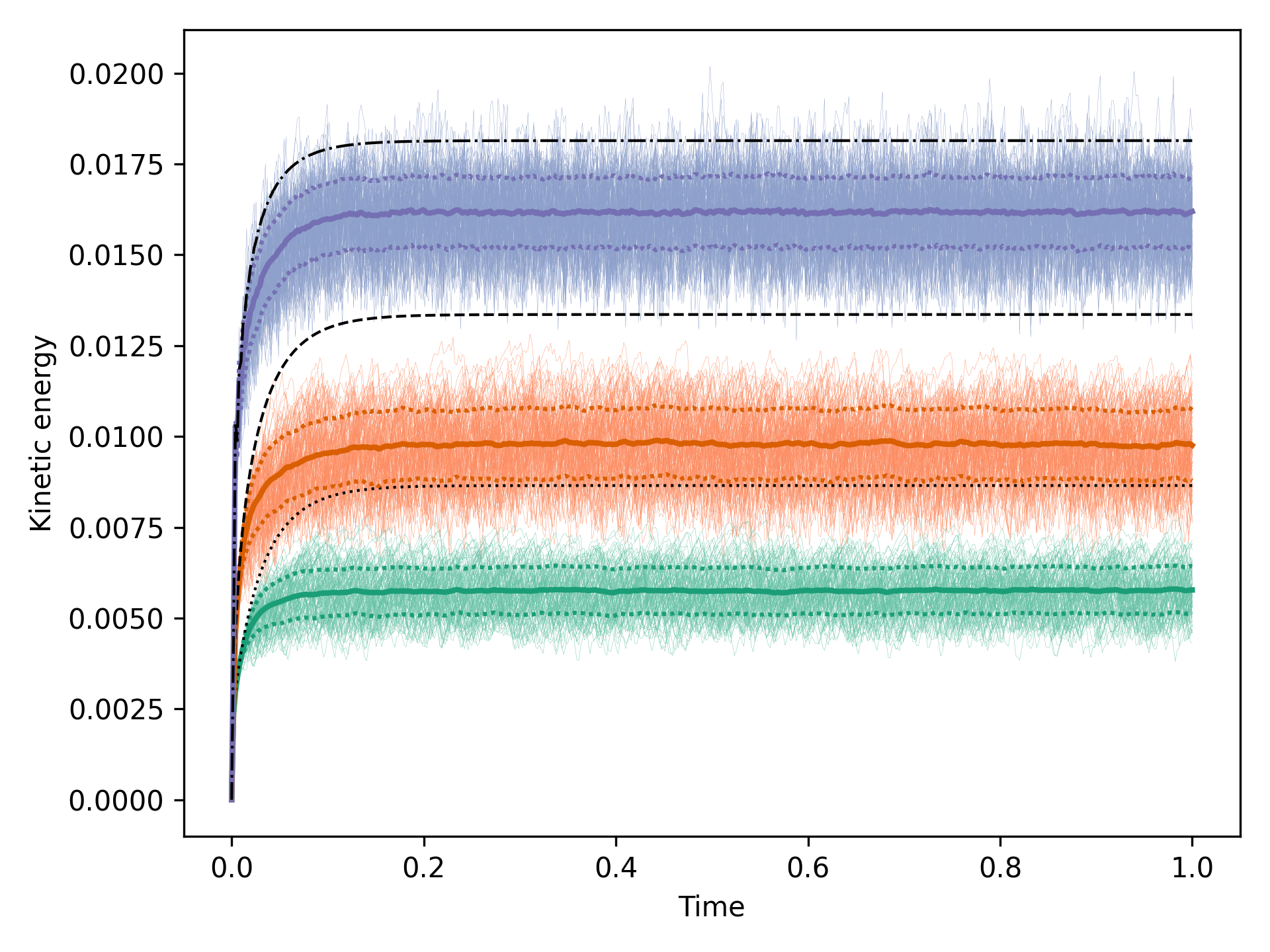}
    \vspace{-2em}
    \caption{Time evolution of kinetic energy: $\tau n \mapsto \tfrac{1}{2}\norm{\Pi_\disc v_\disc^n}_{L^2(\mathcal{O})}^2$, for lid-driven cavity experiments with viscous growth rates: $p=1.5$~(sto:~{\protect\tikz \protect\draw[color=intro_color1, line width=2] (0,0) -- (0.5,0);}; det:~{\protect\tikz \protect\draw[color=intro_color4, line width=1,dotted] (0,0) -- (0.5,0);}), $p=2$~(sto:~{\protect\tikz \protect\draw[color=intro_color2, line width=2] (0,0) -- (0.5,0);}; det:~{\protect\tikz \protect\draw[color=intro_color4, line width=1,dashed] (0,0) -- (0.5,0);}), and $p=3$~(sto:~{\protect\tikz \protect\draw[color=intro_color3, line width=2] (0,0) -- (0.5,0);}; det:~{\protect\tikz \protect\draw[color=intro_color4, line width=1,dashdotted] (0,0) -- (0.5,0);}). Thick lines and dotted lines show the mean energy and the mean energy plus or minus one standard deviation, respectively. The first 100 (out of 1,000) energy trajectories are shown in pale colours.}
    \label{fig:lidkin-traj}
    \end{subfigure}
    \begin{subfigure}{1\textwidth}
    \includegraphics[width=1\linewidth]{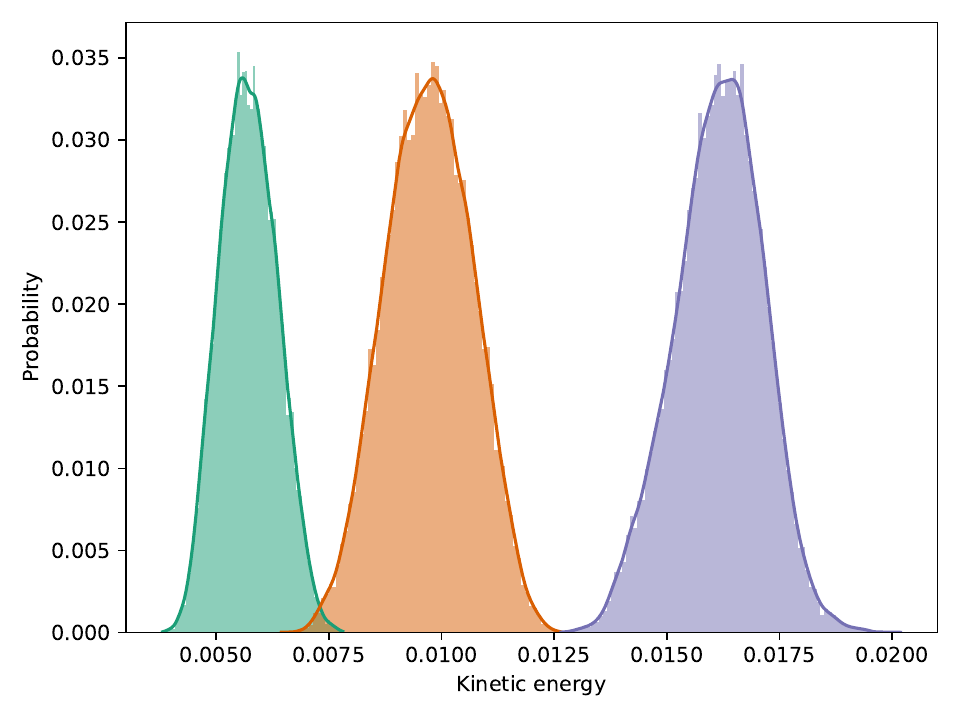}
    \vspace{-2em}
    \caption{Empirical approximation (based on 100 trajectories) of the stationary distributions of the kinetic energy for lid-driven cavity experiments with viscous growth rates: $p=1.5$~({\protect\tikz \protect\draw[color=intro_color1, line width=2] (0,0) -- (0.5,0);}), $p=2$~({\protect\tikz \protect\draw[color=intro_color2, line width=2] (0,0) -- (0.5,0);}), and $p=3$~({\protect\tikz \protect\draw[color=intro_color3, line width=2] (0,0) -- (0.5,0);}).}
    \label{fig:lidkin-hist}
    \end{subfigure}
    \vspace{-2em}
    \caption{Statistics of kinetic energy for EXP-2.}
    \label{fig:lidkin-total}
\end{figure}

\begin{figure}
    \centering
    \begin{subfigure}{1\textwidth}
    \centering
    \includegraphics[scale=0.625]{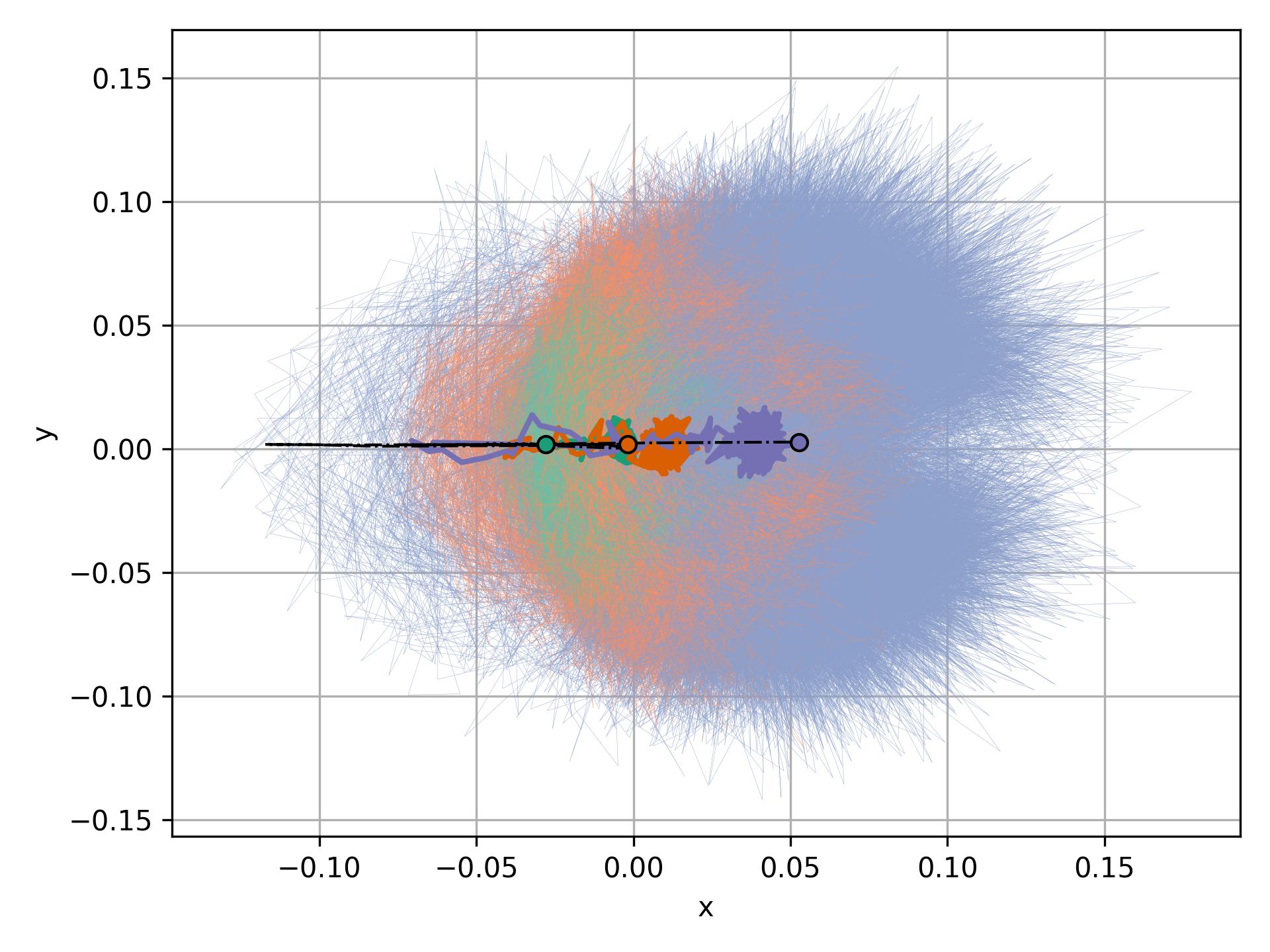}
    \vspace{-1em}
    \caption{Time evolution of velocity vector at the fixed spatial location $(x^*,y^*) = (0.5,0.75)$ for lid-driven cavity experiments with viscous growth rates: $p=1.5$~(sto:~{\protect\tikz \protect\draw[color=intro_color1, line width=2] (0,0) -- (0.5,0);}; det:~{\protect\tikz \protect\draw[color=intro_color4, line width=1,dotted] (0,0) -- (0.5,0);}), $p=2$~(sto:~{\protect\tikz \protect\draw[color=intro_color2, line width=2] (0,0) -- (0.5,0);}; det:~{\protect\tikz \protect\draw[color=intro_color4, line width=1,dashed] (0,0) -- (0.5,0);}), and $p=3$~(sto:~{\protect\tikz \protect\draw[color=intro_color3, line width=2] (0,0) -- (0.5,0);}; det:~{\protect\tikz \protect\draw[color=intro_color4, line width=1,dashdotted] (0,0) -- (0.5,0);}). Thick lines show the mean velocity vector. Individual trajectories are shown in pale colours. Deterministic stationary velocity vectors are denoted by: {\protect\tikz \protect\draw[color=intro_color4,fill=intro_color1] (0,0) circle (.5ex);}, {\protect\tikz \protect\draw[color=intro_color4,fill=intro_color2] (0,0) circle (.5ex);}, and {\protect\tikz \protect\draw[color=intro_color4,fill=intro_color3] (0,0) circle (.5ex);} for $p=1.5$, $p=2$, and $p=3$, respectively.}
    \label{fig:lidpoint-traj}
    \end{subfigure}
    \begin{subfigure}{1\textwidth}
    \centering
    \includegraphics[scale=0.625]{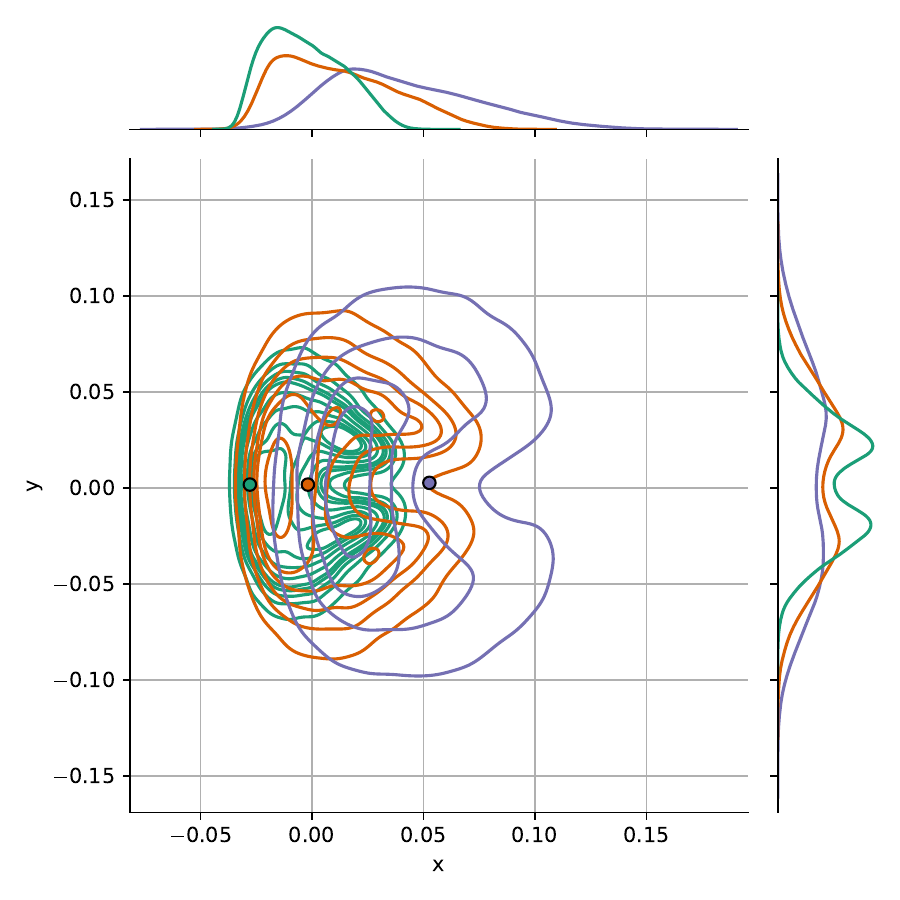}
     \vspace{-1em}
    \caption{Empirical approximation (based on 100 trajectories) of the stationary distributions of the velocity vector at the fixed spatial location $(x^*,y^*) = (0.5,0.75)$ for the lid-driven cavity experiments with viscous growth rates: $p=1.5$~({\protect\tikz \protect\draw[color=intro_color1, line width=2] (0,0) -- (0.5,0);}), $p=2$~({\protect\tikz \protect\draw[color=intro_color2, line width=2] (0,0) -- (0.5,0);}), and $p=3$~({\protect\tikz \protect\draw[color=intro_color3, line width=2] (0,0) -- (0.5,0);}). Deterministic stationary velocity vectors are denoted by: {\protect\tikz \protect\draw[color=intro_color4,fill=intro_color1] (0,0) circle (.5ex);}, {\protect\tikz \protect\draw[color=intro_color4,fill=intro_color2] (0,0) circle (.5ex);}, and {\protect\tikz \protect\draw[color=intro_color4,fill=intro_color3] (0,0) circle (.5ex);} for $p=1.5$, $p=2$, and $p=3$, respectively.}
    \label{fig:lidpoint-hist}
    \end{subfigure}
    \vspace{-1em}
    \caption{Statistics of velocity vectors at fixed spatial location for lid-driven cavity experiments.}
    \label{fig:lidpoint-total}
\end{figure}

\section{Conclusions} \label{sec:conclusion}
In this article, we proposed a new algorithm for non-Newtonian Stokes flows with gradient noise that builds on the classical Crank--Nicolson time-stepping algorithm in combination with a generic Gradient Discretisation. The Gradient Discretisation Method enabled us to derive a unified stability analysis of the algorithm for a broad class of particular spatial discretisations. The long-term stability further enabled the investigation of two sequences of approximate measures; these sequences of measures were found to be promising candidates for the construction of the invariant measure. At the moment, each sequence still lacks one important feature: either the existence of a limit measure, or the invariance with respect to the semigroup. We derived an abstract condition on the distance of consecutive velocity that enabled us to merge the desired properties of both sequences, recovering the existence of an invariant measure. We provided an example for which invariance and existence hold simultaneously, and characterised the invariant measure completely. Moreover, we investigated the abstract condition for more general cases numerically: we found that the condition won't be satisfied in general. We closed the article by conducting two thorough numerical experiments that show the influence of transport noise on the dynamics of power-law fluids; in particular, we found that transport noise enhances the dissipation of energy, the mixing of particles, as well as the size of vortices.  

All numerical experiments indicate that our algorithm gives rise to an invariant measure. But for general boundary conditions, a theoretical justification is still pending. While it would be desirable to understand this convergence on the purely discrete level, we are lacking an ansatz to proceed; an alternative approach, the topic of ongoing work, that leaves the purely discrete setting, is the limit passage to the time-continuous formulation. After the limit passage, the two time scales of the algorithm will be merged. We expect that the time-continuous formulation of the algorithm will have a unique invariant measure, which will be constructed by ergodic averages.  

Establishing quantified convergence results of the approximate measure towards the invariant measure requires further study of the weak and strong errors of our numerical algorithm. Nothing is currently known about this, which is mainly due to the limited regularity of the solutions of non-linear stochastic equations. The only result known to us that shows the convergence of the strong error for an approximation of the generalised Stokes system is~\cite{Le2024Spacetime}. However, the algorithm is limited to gradient-independent, multiplicative noises and is ill-prepared for the study of its transition operator since the time-stepping algorithm depends on two time steps. Whether some of the methods can be transferred to the algorithm proposed in this article remains to be investigated.

\appendix

\section{Implementation aspects of algorithm} \label{sec:algo}
In this appendix, we present pseudo-codes of our proposed algorithm, as well as its building blocks.

\subsection{Initialisation algorithm}
We propose the use of the discrete Helmholtz decomposition for the construction of well-prepared initial data: 
\begin{framed}
\begin{algorithm}[H]
\KwData{ 
\begin{itemize}
    \item Gradient Discretisation~$\disc$
    \item analytic velocity~$v^\mathrm{in}$
\end{itemize}
}
\KwResult{discrete velocity~$v^\mathrm{out}_\disc$ and pressure~$\pi^\mathrm{out}_\disc$}
Find $(v^\mathrm{out}_\disc,\pi^\mathrm{out}_\disc) \in X_{\disc,0} \times Y_{\disc,0}$ such that for all $(\xi, q) \in  X_{\disc,0} \times Y_{\disc,0}$:
\begin{subequations} 
\begin{alignat*}{2} 
&\left( \Pi_\disc v^\mathrm{out}_\disc, \Pi_\disc \xi \right) + \left( \chi_\disc \pi^\mathrm{out}_\disc, \Div_\disc \xi \right) &&= \left( v^\mathrm{in}, \Pi_\disc \xi \right), \\ 
 & \left(  \Div_\disc v^\mathrm{out}_\disc,\chi_\disc q \right) &&= 0.
\end{alignat*}
\end{subequations}
\caption{Discrete Helmholtz decomposition}
\label{algo:Initialisation}
\end{algorithm}
\end{framed}

\subsection{Time-propagation algorithm}
We propose the use of a concrete gradient discretisation in combination with a semi-implicit time-stepping scheme: 
\begin{framed}
\begin{algorithm}[H]
\KwData{ 
\begin{itemize}
    \item Gradient Discretisation~$\disc$
    \item time step-size~$\tau$ 
    \item random update~$\Delta W$
    \item boundary condition~$g$
    \item noise coefficients~$\sigma$ and~$B_\disc$
    \item viscous stress tensor~$S$
    \item discrete input velocity~$v^\mathrm{in}_\disc$
    \item discrete input pressure~$\pi^\mathrm{in}_\disc$
\end{itemize}
}
\KwResult{time propagated discrete velocity~$v^\mathrm{out}_{\disc}$ and pressure~$\pi^\mathrm{out}_{\disc}$}
Find $(v^\mathrm{out}_{\disc}, \pi^\mathrm{out}_{\disc}) \in X_{\disc,0} \times Y_{\disc,0}$ such that for all $(\xi, q) \in  X_{\disc,0} \times Y_{\disc,0}$ and $\mathbb{P}$-a.s.
\begin{subequations} 
\begin{alignat*}{2}  
&\left( \Pi_\disc v^\mathrm{out}_{\disc} -\Pi_\disc v^\mathrm{in}_\disc, \Pi_\disc \xi \right) + \tau \left( S( \varepsilon_\disc v^{1/2}_\disc + \varepsilon g ), \varepsilon_\disc \xi \right) \\ \nonumber
&\hspace{2em} - \left( \chi_\disc \pi^\mathrm{out}_{\disc} - \chi_\disc \pi^\mathrm{in}_{\disc}, \Div_\disc \xi \right)\\  \nonumber
&\hspace{2em} - \left[ B_\disc( v^{1/2}_\disc, \xi)+ \left( (\sigma \cdot \nabla)g , \Pi_{\disc} \xi \right) \right] \Delta W &&= 0, \\ 
 & \left(  \Div_\disc v^{1/2}_\disc,\chi_\disc q \right) &&= 0,
\end{alignat*}
\end{subequations}
where $v^{1/2}_\disc := \tfrac{1}{2}\big( v^\mathrm{out}_{\disc} + v^\mathrm{in}_{\disc}\big)$.
\caption{GDM and Crank--Nicolson scheme}
\label{algo:local-Crank--NicolsonAndGDM}
\end{algorithm}
\end{framed}

\subsection{Global algorithm} Combining an initialisation algorithm and a local time-propagation algorithm, we define the time-global fully discrete algorithm as follows.
\begin{framed}
\begin{algorithm}[H]
\KwData{ 
\begin{itemize}
    \item Gradient Discretisation~$\disc$
    \item time step-size~$\tau$ 
    \item Wiener process~$W$
    \item boundary condition~$g$
    \item noise coefficients~$\sigma$ and~$B_\disc^\sigma$
    \item viscous stress tensor~$S$
    \item analytic initial velocity~$v^\mathrm{in}$
\end{itemize}
}
\KwResult{approximate velocity~$(v^n_\disc)_{n \in \mathbb{N}_0}$ and pressure~$(\pi^n_\disc)_{n \in \mathbb{N}_0}$}
Define $t_0 = 0$; \\
Construct $(v^0_\disc,\pi^0_\disc) \in X_{\disc,0} \times Y_{\disc,0}$ via Algorithm~\ref{algo:Initialisation}; \\
\For{$n = 1,2, \ldots $ }{
Define $t_n = n \tau$ and $\Delta_n W = W(t_n) - W(t_{n-1})$;\\
Construct $(v^n_\disc,\pi^n_\disc) \in X_{\disc,0} \times Y_{\disc,0}$ via Algorithm~\ref{algo:local-Crank--NicolsonAndGDM} applied to $v^{\mathrm{in}}_\disc = v_\disc^{n-1}$, $\pi^{\mathrm{in}}_\disc = \pi_\disc^{n-1}$ and $\Delta W = \Delta_n W$;
}
\caption{GDM and Crank--Nicolson algorithm initialised using the discrete Helmholtz decomposition}
\label{algo:Crank--NicolsonAndGDM}
\end{algorithm}
\end{framed}

\printbibliography 

\end{document}